\documentclass{article} 
\usepackage{authblk}
\usepackage[utf8]{inputenc}
\usepackage[english]{babel}
\usepackage[T1]{fontenc}
\usepackage{amsmath}
\usepackage{amsfonts}
\usepackage{amssymb}
\usepackage{graphicx}
\usepackage{multicol}
\usepackage{array}
\usepackage{mathrsfs}
\usepackage{amsthm}
\usepackage[all]{xy} 
\usepackage{xcolor}
\usepackage{breakurl}
\usepackage[colorlinks=true, linkcolor=violet, citecolor=orange, anchorcolor=orange,]{hyperref}
\usepackage[driver=dvips,text={372pt,533pt},centering]{geometry}
\numberwithin{equation}{section}

\theoremstyle{plain}
\newtheorem{thm}{Theorem}[section]
\newtheorem*{thm*}{Theorem}
\newtheorem{lem}[thm]{Lemma}
\newtheorem{prop}[thm]{Proposition}
\newtheorem{cor}[thm]{Corollary}

\theoremstyle{definition} 
\newtheorem*{defn*}{Definition}
\newtheorem{defn}[thm]{Definition}
\newtheorem{exs}[thm]{Examples}
\newtheorem{exam}[thm]{Example}
\newtheorem{rem}[thm]{Remark}
\newtheorem{rems}[thm]{Remarks}
\usepackage{thmtools}
\usepackage{thm-restate}

\newcommand{\dinf}{\mathrm{D}^{(\infty)}_{A,\underline{\sigma}}}
\newcommand{\N}{\mathbb{N}}

\newcommand{\der}{\partial^{[\underline{k}]}_{\underline{\sigma}}}
\newcommand{\aff}{A\left\lbrace \underline{\xi}/\eta \right\rbrace}

\title{Twisted calculus in several variables}
\author{Pierre Houédry}
\affil{Université de Caen Normandie}
\date{}

\begin{document}

\maketitle
\begin{abstract}
In this paper, we introduce novel concepts and establish a formal framework for twisted differential operators in the context of several variables. The focus is on twisted coordinates within Huber rings, which facilitate the construction of diverse rings of twisted differential operators. We establish an equivalence between modules equipped with twisted connections and those endowed with actions of twisted derivatives. Furthermore, we examine the convergence properties of twisted differential operators under specific conditions. As one of the main results, we extend the confluence theorem of Le Stum and Quirós to several variables. This work aligns with the ongoing advancements in $p$-adic Hodge cohomology and prismatic cohomology.
\end{abstract} 

\tableofcontents

\section{Introduction}
A $q$-difference equation is a relation of the form  
\[
E(x, f, \partial_q(f), \ldots, \partial_q^n(f)) = 0,
\]  
which involves an unknown function $f$ and its iterates under the $q$-difference operator:  
\[
\partial_q(f)(x) = \frac{f(qx) - f(x)}{(q-1)x}.
\]  
The primary motivation for studying such equations is practical. The goal is to discretize a differential equation in order to facilitate its analysis, a process known as ``deformation''. Conversely, recovering a differential equation from a $q$-difference equation as $q$ tends towards 1 is referred to as ``confluence''. From a practical standpoint, the key problem is to quantify the discrepancy between the solutions of the discretized equation obtained via deformation and those of the original differential equation. 

This problem can be approached from an arithmetic perspective, specifically within the realm of $p$-adic numbers. In 2000, Lucia Di Vizio proved a $q$-analogue of Grothendieck's $p$-curvature conjecture (see \cite{DiVizio}). Shortly after, in 2004, Yves André and Lucia Di Vizio investigated the phenomenon of confluence (\emph{cf.} \cite{ADV04}). They computed the Tannakian group of $q$-difference equations over the Robba ring, equipped with a Frobenius action, and demonstrated that it coincides with the Tannakian group of differential equations over the Robba ring with Frobenius action. Simultaneously, during his doctoral thesis, Andrea Pulita also explored these problems and established an equivalence between these categories by removing the assumptions on $q$ from \cite{ADV04} for equations on the affine line. Pulita introduced a principle called the \emph{propagation principle}. Once verified, this principle suffices to prove the equivalence between confluence and deformation on rings of $p$-adic functions. Pulita extended his results to curves and to actions of operators more general than multiplication by $q$ (\emph{cf.} \cite{pulita}). 

Interest in this theory increased substantially when Bhargav Bhatt, Matthew Morrow, and Peter Scholze published their work on integral $p$-adic cohomology (see \cite{BMS1}). Their construction enables specialization to both de Rham and étale cohomologies. They demonstrated that this cohomology can be computed locally using $q$-difference theory. Furthermore, Bhatt and Scholze introduced the concepts of \emph{prism} and \emph{prismatic cohomology} (\cite{prism}) to unify the various cohomologies they had previously defined. Remarkably, it seems that $q$-difference equations can once again be employed to compute prismatic cohomology.

In recent years, Bernard Le Stum, Adolfo Quirós, and Michel Gros have developed a more general theory (see \cite{SQ1}, \cite{SQ2}, \cite{LSQ3}, \cite{GLQ4}, \cite{GLS5}, and \cite{GLQ6}) that encompasses earlier results on $p$-adic $q$-difference equations (see \cite{ADV04}, \cite{pulita}). They generalized the notion of differential operators and extended their work to affinoid algebras. Their theory draws an analogy with the formalism established in EGA-IV (see \cite{Grothe}) and is guided by the classical framework of crystalline cohomology. However, their methods do not extend to multiple variables. As in classical calculus, transitioning from ordinary differential equations to partial differential equations is a highly non-trivial step. 

Matthew Morrow and Takeshi Tsuji have made significant progress by developing a multi-variable theory in the context of $q$-difference equations (see \cite{morrow2}). The present article serves as foundational work, aiming to generalize prior theories by addressing the case of several variables and incorporating more general twists than multiplication by $q$. Additionally, the work of Claude Sabbah (see \cite{CS}), though conducted over $\mathbb{C}$, bears a close connection to the research pursued here.

Section \ref{Huber rings} focuses on Huber rings, starting with their definition from \cite{huber}. We then cover key constructions: completion and tensor products, which are essential for extending one-variable results to our multi-variable setup.

In Section \ref{Topo étale morphisms}, we introduce the concept of a topologically étale morphism, an affine variant of \cite[Definition 2.3.13]{abbes}, along with the notion of \emph{good étale} morphism. On an adic ring, a topologically étale morphism is automatically good étale. This allows us to extend coordinate properties related to differentials with good étale morphisms, facilitating this extension under certain conditions.

In Section \ref{Twisted principal parts}, we follow a similar approach to the classical case to develop twisted differential calculus. However, when $A$ is equipped with multiple automorphisms $\underline{\sigma} = (\sigma_1, \ldots, \sigma_d)$, there is no natural notion of twisted powers of $I$ to apply the same formalism in several variables. Our work consists of finding the right definitions to extend as many results from the one-variable case to the multi-variable case as possible. Unlike \cite{GLQ4}, we proceed by first defining the module of twisted Taylor series $P_{A, (n)_{\underline{\sigma}}}$ for a Huber ring $R$ and an $R$-adic algebra $A$ with continuous linear $R$-endomorphisms $\underline{\sigma}$ that commute with each other (Definition \ref{module of twisted differential parts}). The construction of the different rings we want to compare relies on choosing coordinates $\underline{x}=(x_1,\ldots,x_d)$ with good properties to extend the one-variable results to several variables. The core concept enabling this extension is that of \emph{$\underline{\sigma}$-coordinates} (Definition \ref{coor}) and that of good $\underline{\sigma}$-coordinates (Definition \ref{goodcoor}).

In Section \ref{Derivations}, we define the module of twisted differential operators $\Omega^1_{A, \underline{\sigma}}$ (Definition \ref{twisted differential forms}) which allows us to recover the operators $\partial_{\underline{\sigma},1}, \ldots, \partial_{\underline{\sigma},d}$. However, there is no clear reason for $\partial_{\underline{\sigma},i}$ to be a $\sigma_i$-derivation (Definition \ref{sigma derivation}), which necessitates an additional hypothesis, leading to the definition of \emph{classical coordinates} (Definition \ref{classique}). This allows us to prove the following theorem:

\begin{thm*}[Theorem \ref{derivation}]
The (good) $\underline{\sigma}$-coordinates $\underline{x}$ are classical if and only if, for each $i = 1, \ldots, d$, $\partial_{\underline{\sigma},i}$ is a $\sigma_i$-derivation.
\end{thm*}

Given $\underline{\sigma}$-coordinates, we define the notion of a \emph{twisted connection} (Definition \ref{Twisted connection}) and establish the following
\begin{thm*}[Theorem \ref{equivcate}]
When the (good) $\underline{\sigma}$-coordinates $\underline{x}$ are classical, there is an equivalence of categories
\[
\nabla_{\underline{\sigma}}  \emph{-Mod}(A) \simeq \mathrm{T}_{A,\underline{\sigma}}\emph{-Mod}(A)
\]
between $A$-modules with a twisted connection and $A$-modules with an action of the derivations.
\end{thm*}

By adding the mild hypothesis that the coordinates are symmetric (Definition \ref{sym}), 
we define in Section \ref{twisted differential operators of infinite level} the ring of twisted differential operators of infinite level 
$\mathrm{D}_{A,\underline{\sigma}}^{(\infty)}$ (Definition \ref{twisted differential operator}). To relate to one-variable case of $q$-differences we define $\underline{q}$-coordinates in the following way:

\begin{defn*}
We say that the coordinates $\underline{x}=(x_1,\ldots,x_d)$ are (good) \emph{$\underline{q}$-coordinates} if there exists $q_1,\ldots,q_d \in R$ such that $\underline{x}$ are (good) $\underline{\sigma}$-coordinates for the $\underline{\sigma}$ defined by
\[
\forall i,j=1, \ldots, d, \quad\sigma_i(x_j) = q_i^{\delta_{i,j}}x_j.
\] 
For notational simplicity, we will henceforth write
\[
\nabla_{\underline{q}} := \nabla_{\underline{\sigma}}, \qquad
\mathrm{D}^{(\infty)}_{A,\underline{q}} := \mathrm{D}^{(\infty)}_{A,\underline{\sigma}}
\]
and similarly for the associated categories of modules.
\end{defn*}

This leads to the following result:

\begin{thm*}[Theorem \ref{equivate 3}]
Assume that the coordinates are symmetric (good) $\underline{q}$-coordinates and that $\forall n \in \N, ~\forall i=1,\ldots,d,~\frac{q_i^n-1}{q_i-1} \in R^\times$. 
There exists an equivalence of categories
\[
\nabla_{\underline{q}}^{\mathrm{Int}}\emph{-Mod}(A) \simeq \mathrm{D}_{A,\underline{q}}^{(\infty)}\emph{-Mod},
\]
where $\nabla_{\underline{q}}^{\mathrm{Int}}\emph{-Mod}(A)$ designates the category of $A$-modules endowed with an integrable twisted connection.
\end{thm*}
Thus, the definition of several properties on $\underline{\sigma}$-coordinates provides a suitable framework for constructing the theory of twisted differential operators in several variables, as it permits to recover the usual properties of twisted differential calculus under natural hypotheses.

In the final section, we investigate the convergence of twisted differential operators in several variables (Definition \ref{radius of convergence}). We begin by assuming that \(R\) is a Tate ring, and that \((A, \underline{\sigma})\) is a complete Noetherian twisted \(R\)-adic algebra of order \(d\), where \(A\) is also a Tate ring. Additionally, we assume that the norm on \(A\) is contractive with respect to the endomorphisms $\underline{\sigma}$. Then, under certain convergence assumption on twisted differential operators, we define the ring \(\mathrm{D}^{(\eta)}_{\underline{\sigma}}\) of twisted differential operators of radius \(\eta\) for certain $\eta \in \mathbb{R}_{>0}$ (Definitions \ref{operators of radius eta} and \ref{ring of twisted differential operators of radius eta}), as well as its direct limit \(\mathrm{D}_{\underline{\sigma}}^{(\eta^\dagger)} = \varinjlim_{\eta'<\eta} \mathrm{D}_{\underline{\sigma}}^{(\eta')} \). This framework establishes the foundation for twisted differential calculus and allows us to prove the following result:

\begin{thm*}[Theorem \ref{confluence}]
Assume that \( \underline{x} \) are classical and symmetric (good) \( \underline{q} \)-coordinates and that $\forall n \in \N; ~\forall i=1, \ldots, d,~ \frac{q_i^n-1}{q_i-1} \in R^\times$. Assume moreover that \( A \) is strongly \( \eta^\dagger \)-convergent. 
Then we have the equivalence of categories
\[
\nabla_{\underline{q}}^{\mathrm{Int}}\emph{-Mod}_{\mathrm{tf}}^{(\eta^\dagger)}(A) \simeq \mathrm{D}_{\underline{q}}^{(\eta^\dagger)}\emph{-Mod}_{\mathrm{tf}}(A)
\]
between the category of \( A \)-modules of finite type endowed with an integrable twisted connection that are \( \eta^\dagger \)-convergent and the category of \( \mathrm{D}_{\underline{q}}^{(\eta^\dagger)} \)-modules of finite type over \( A \). This equivalence is compatible with de Rham cohomology on left hand side and $\emph{Ext}$ groups on the right hand side.
\end{thm*}

We conclude by proving the following theorem, which extends the confluence theorem of \cite{GLQ4} to several variables, in resonance with the approaches developed in \cite{ADV04} and \cite{pulita}.

\begin{thm*}[Theorem \ref{thm:confluence-bis}]
Let \( K \) be a non-archimedean field of characteristic 0, complete with respect to a non-archimedean norm. Let \( A \) be a Tate \( K \)-algebra, and let $\underline{q}=(q_1, \ldots,q_d)$ be elements of $K$ such that, for every $i=1,\ldots,d$ and $n\geq1, \tfrac{q_i^n-1}{q_i-1} \in K^\times$. Assume that $A$ admits elements $\underline{x}=(x_1, \ldots,x_d)$ which are classical and symmetric (good) $\underline{q}$-coordinates as well as symmetric (good) $\underline{\emph{Id}}$-coordinates, and asusme moreover that $A$ is strongly $\eta^\dagger$-convergent with respect to  \( \underline{q} \) and with respect to $\underline{\emph{Id}}$, and that \( \eta > \rho(\underline{q}) \). 
Then we have the equivalence of categories
\[
\nabla_{\underline{q}}^{\emph{Int}}\text{-}\emph{Mod}_{\emph{tf}}^{(\eta^\dagger)}(A) \simeq \nabla_{\underline{\emph{Id}}}^{\emph{Int}}\emph{-}\emph{Mod}_{\emph{tf}}^{(\eta^\dagger)}(A)
\]
which is compatible with de Rham cohomologies on both sides.
\end{thm*}

\section{Huber rings}\label{Huber rings}
In what follows, we will assume that every ring is unitary.
\begin{defn} \label{anneauxdehuber}
Let $A$ be a topological commutative ring.
\begin{enumerate}
    \item We say that $A$ is a \emph{Huber ring} if there exists an open subring $A_0$ of $A$ and a finitely generated ideal $I_0$ of $A_0$ such that $\lbrace I_0^n\rbrace_{n \in \mathbb{N}}$ is a fundamental system of neighbourhoods of 0 in $A_0$. In that case, we say that $A_0$ is a \emph{ring of definition}, $I_0 $ is an \emph{ideal of definition} and that $(A_0,I_0)$ is a \emph{couple of definition}.
    \item We say that $A$ is \emph{adic} in the case where it is a Huber ring and we can choose $A_0=A$. 
    \item We say that $A$ is \emph{discrete} in the case where it is a Huber ring and we can choose $A_0=A$ and $I_0=0$.
\end{enumerate}
\end{defn}

\begin{exs}\label{huberexs}
\begin{enumerate}
\item The ring of integers $\mathbb{Z}$ endowed with the discrete topology is a discrete ring.
\item For a prime number $p$, $\mathbb{Z}_p$ endowed with the $p$-adic topology is an adic ring: the ideal $(p)$ is an ideal of definition.
\item Let $A$ be a Huber ring. For every $n \in \N$, the ring $A[X_1,\ldots,X_n]$ can be endowed with a structure of Huber ring that induces the topology on $A$: we can choose $(A_0[X_1,\ldots,X_n],I_0A_0[X_1,\ldots,X_n])$ as a couple of definition.
\item A field \( K \) is referred to as a \emph{non-archimedean field} if its topology can be defined by a nontrivial non-archimedean valuation $|\cdot|$. In this context, there exists an element \( \pi \in K \) such that \( 0 < |\pi| < 1 \), and \( K \) is a Huber ring. Specifically, the subset \( K_0=\left\lbrace x \in K \mid |x| \leq 1 \right\rbrace \) forms a ring of definition, and the principal ideal generated by \( \pi \) is an ideal of definition. \label{non-archimedean field}
\item Let $d \in \mathbb{N}$. In the context previously described, we consider the following ring
\[
K\left\lbrace \underline{X} \right\rbrace=\left\lbrace \sum_{\underline{k}\in \mathbb{N}^d} a_{\underline{k}}\underline{X}^{\underline{k}},a_{\underline{k}} \in K, |a_{\underline{k}}| \rightarrow 0 \text{ when } |\underline{k}| \rightarrow +\infty \right\rbrace ,
\]
equipped with the topology induced by the Gauss norm
\[
\left\lVert  \sum_{\underline{k} \in \mathbb{N}^d} a_{\underline{k}} \underline{X}^{\underline{k}} \right\rVert =  \sup_{\underline{k} \in \mathbb{N}^d} |a_{\underline{k}}|.
\]
With this topology, the ring $K\left\lbrace \underline{X} \right\rbrace$ is a Huber ring. A couple of definition is given by $(K_0\left\lbrace \underline{X} \right\rbrace, I_0)$ where
\begin{align*}
    K_0\left\lbrace \underline{X} \right\rbrace:=&\left\lbrace \sum_{\underline{k} \in \mathbb{N}^d} a_{\underline{k}} \underline{X}^{\underline{k}} \in K\lbrace \underline{X} \rbrace, a_{\underline{k}} \in K_0 \right\rbrace \\
  \text{ and }  I_0:=&\left\lbrace \sum_{\underline{n} \in \mathbb{N}^d} a_{\underline{k}} \underline{X}^{\underline{k}} \in K\lbrace \underline{X} \rbrace, a_{\underline{k}} \in \pi K_0 \right\rbrace.
\end{align*}

\end{enumerate}
\end{exs}

\begin{defn}
Let \( A \) and \( B \) be two Huber rings. A \emph{morphism of Huber rings} \( u: A \rightarrow B \) is a ring homomorphism that is continuous with respect to the topologies on \( A \) and \( B \). This implies that there exist a ring of definition \( A_0 \) for \( A \) (resp. \( B_0 \) for \( B \)) and an ideal of definition \( I_0 \subset A_0 \) (resp. \( J_0 \subset B_0 \)) such that \( u(A_0) \subset B_0 \) and \( u(I_0) \subset J_0 \) (see \cite[Lemma II.1.3.5]{Morel}). If we can take $A_0, I_0, B_0$ as above so that \( u(I_0)B_0 \) is an ideal of definition of \( B \), then the morphism \( u \) is said to be \emph{adic}.
\end{defn}

\begin{exs}
\begin{enumerate}
\item Every morphism of discrete rings is adic.
\item The identity $\mathbb{Z} \rightarrow \mathbb{Z}$ is not continuous if $\mathbb{Z}$ is endowed with the $p$-adic topology on the left side and the discrete topology on the right side.
\item If $A$ is a discrete ring and $A[T]$ is endowed with the $T$-adic topology, then the canonical morphism $A \rightarrow A[T]$ is continuous but not adic.
\end{enumerate}
\end{exs}

\begin{prop}\label{moreladique}
Let $u:A\rightarrow B$ be an adic morphism of Huber rings. If $A_0$ and $B_0$ are rings of definition of $A$ and $B$ such that $u(A_0) \subset B_0$, then, for every ideal of definition $I_0$ of $A_0$, the ideal $u(I_0)B_0$ is an ideal of definition of $B$. 
\end{prop}

\begin{proof}
 We refer the reader to \cite[Proposition II.1.3.3]{Morel}.
\end{proof}

\begin{defn}
If $R \rightarrow A$ is a morphism of Huber rings, $A$ is said to be a \emph{Huber $R$-algebra}. If the morphism is adic $A$ is said to be an \emph{$R$-adic alegbra}.
\end{defn}

\begin{rem}
An $R$-adic algebra is not 
necessarily an adic ring.
%\item Let $A$ be a Huber ring with $(A_0,I_0)$ as a couple of definition. If $M$ is a finite $A$-module, we can endow it with the universal topology for the continuous $A$-linear morphisms $M\rightarrow N$ into topological $A$-modules. In that case, every $A_0$-module $M_0$ of $M$ that generates $M$ is necessarily open in $M$ and the topology of $M_0$ is the $I_0$-adic topology.%
\end{rem}

\begin{defn}
Let $A$ be a topological ring. A subset $E$ of $A$ is
\begin{enumerate}
\item $\emph{bounded}$ if for every neighbourhood $U$ of $0$ in $A$ there exists an open neighbourhood $V$ of $0$ such that $ax \in U$ for every $a \in E$ and $x \in V$,
\item \emph{power-bounded} if $\bigcup_{n \geq 1} E(n)$ is bounded, where $E(n)=\left\lbrace \prod\limits_{i=1}^n f_i, ~ f_i \in E \right\rbrace$,
\item \emph{topologically nilpotent} if for every neighbourhood $W$ of $0$ in $A$, there exists a positive integer $N$ such that $E(n) \subset W$ for $n \geq N$.
\end{enumerate}

If \( E \) is a singleton \( \lbrace f \rbrace \), we say that \( f \) is \emph{power-bounded} (resp. \emph{topologically nilpotent}) if the set \( E \) possesses the corresponding property.

\end{defn}

\begin{prop}
Let \( A \) be a topological ring, and let \( E = \{ f_1, \ldots, f_r \} \) be a finite subset of \( A \). Then, \( E \) is \emph{power-bounded} if and only if for \( i = 1, \ldots, r \), \( f_i \)  is \emph{power-bounded}.
\end{prop}

\begin{proof}
The forward direction is straightforward, so we will provide details for the converse. By definition, for each neighborhood \( U \) of \( 0 \) in \( A \), there exists a neighborhood \( W_r \) of \( 0 \) in \( A \) such that for every \( n_r \in \mathbb{N} \),
\[
f_r^{n_r} W_r \subset U.
\]
By the same argument, there exists a sequence \( (W_1, \ldots, W_{r-1}) \) of neighborhoods of \( 0 \) in \( A \) such that for every \( (n_1, \ldots, n_{r-1}) \in \mathbb{N}^{r-1} \),
\[
\forall i = 1, \ldots, r-1, \quad f_i^{n_i} W_i \subset W_{i+1}.
\]
Hence, we have
\[
f_1^{n_1} \cdots f_r^{n_r} W_1 \subset U.
\]
This completes the proof.
\end{proof}

In subsequent discussions, it will be beneficial to define a semi-norm on our rings to explore issues related to convergence. For this purpose, we will introduce Tate rings. These are Huber rings whose topology is induced by an ultrametric submultiplicative norm.

\begin{defn}\label{def:tate_ring}
Let $A$ be a Huber ring. We say that $A$ is a \emph{Tate ring} if it has a topologically nilpotent unit.
\end{defn}

\begin{exs}
\begin{enumerate}
\item The ring $\mathbb{Q}$ endowed with the $p$-adic topology is a Tate ring.
\item Every ring \( A \) can be endowed with the \( p \)-adic topology, and in this case, the ring \( A\left[\frac{1}{p}\right] \) becomes a Tate ring.
\item Every non-archimedean field is a Tate ring.
\end{enumerate}
\end{exs}

\begin{prop}\label{norme}
The topology of a Huber ring can always be defined by a submultiplicative semi-norm. Furthermore, in the case of a Tate ring, the topology can be defined by a submultiplicative norm.
\end{prop}

\begin{proof}
Let $A$ be a Huber ring and $(A_0,I_0)$ be a couple of definition. For an element $f$ of $A$ we set
\[
\lVert f \rVert = 
\begin{cases}
    e^n & \text{if } fI_0^n \subset A_0 \text{ but } fI_0^{n-1} \not\subset A_0,\\
    e^{-n} & \text{if } f \in I_0^n \text{ but } f \not\in I_0^{n+1}, \\
    0 & \text{if }  f \in \bigcap I_0^n.
\end{cases}
\]
We then have, $A_0= \left\lbrace f \in A, \lVert f \rVert \leq 1 \right\rbrace$ and $I_0=\left\lbrace f \in A, \lVert f \rVert < 1 \right\rbrace$. 
\end{proof}

\begin{defn}\label{contractante}
A submultiplicative norm \( \lVert \cdot \rVert \) on a ring \( A \) is said to be \emph{contractive} with respect to an endomorphism \( \varphi \) of \( A \) if, for all \( x \in A \), it holds that \( \lVert \varphi(x) \rVert \leq \lVert x \rVert \).
\end{defn}

The tensor product of two Huber rings does not generally exist in a well-defined manner. However, when extending scalars in the context of rings of series and polynomials, it becomes necessary to establish certain conditions under which the tensor product can exist within the category of Huber rings.\newline

Let $R$ be a Huber ring, $A$ and $B$ be two $R$-adic algebra. Choose a couple of definition $(R_0,I)$ of $R$ and select rings of definition $A_0$ and $B_0$ of $A$ and $B$ such that the image of $R_0$ in $A$ is contained in $A_0$ and the image of $R_0$ in $B$ is contained in $B_0$. Let $J$ denote the ideal of $\mathrm{Im}(A_0\otimes_{R_0}B_0 \rightarrow A \otimes_R B)$ that is generated by the image of $I$.

\begin{prop}
   With the notations introduced above, if we endow $\mathrm{Im}(A_0\otimes_{R_0}B_0 \rightarrow A \otimes_R B)$ with the $J$-adic topology and if we equip $A\otimes_R B$ with the unique structure of topological group for which $\mathrm{Im}(A_0\otimes_{R_0}B_0 \rightarrow A \otimes_R B)$ is an open subgroup then $A\otimes_R B$ is a Huber ring with $(\mathrm{Im}(A_0\otimes_{R_0}B_0 \rightarrow A \otimes_R B),J)$ as a couple of definition.
\end{prop}

\begin{proof}
    We refer the reader to \cite[Proposition II.3.2.1]{Morel}.
\end{proof}

Later, to address issues related to convergence, it will be necessary to consider complete topological rings. 

\begin{prop}\label{complété}
Let $A$ be a Huber ring and $(A_0,I)$ be a couple of definition of $A$. Set $\hat{A}= \varprojlim_{n \geq 0} A/I^n, \hat{A_0}= \varprojlim_{n \geq 0} A_0/I^n$ (as an abelian group). Then:
\begin{enumerate}
\item If we put the unique topology on $\hat{A}$ for which $\hat{A_0}$ is an open subgroup, then $\hat{A}$ is a complete topological group.
\item There is a unique ring structure on $\hat{A}$ that makes the canonical map $A \rightarrow \hat{A}$ continuous, and $\hat{A}$ is a topological ring.
\item The ring $\hat{A}$ is a Huber ring and $(\hat{A_0}, I_0\hat{A_0})$ is a couple of definition of $\hat{A}$. Moreover, the canonical map $A \rightarrow \hat{A}$ is adic.
\item  The canonical map $\hat{A_0}\otimes_{A_0} A \rightarrow \hat{A}$ is an isomorphism.
\item For every natural integer $n$, the induced map $A_0/I^nA_0 \rightarrow \hat{A_0}/I^n\hat{A_0}$ is an isomorphism.
\end{enumerate}
\end{prop}

\begin{proof}
We refer the reader to \cite[Theorem II.3.1.8, Corollary II.3.1.9]{Morel}. \qedhere
\end{proof}

\begin{defn}
If $A$ is a Huber ring, the Huber ring $\hat{A}$ defined in Proposition \ref{complété} is called the \emph{completion} of $A$.
\end{defn}

\begin{rems}
\begin{enumerate}
\item If \( A \) is a complete Huber ring, then every ring of definition of \( A \) is also complete. Any ring of definition is closed in \( A \) because it is an open subgroup of \( A \).

\item A ring with the discrete topology is a complete adic ring. There exists a fully faithful functor from the category of rings into the category of complete adic rings.
\end{enumerate}
\end{rems}

\section{Topologically étale morphisms}\label{Topo étale morphisms}
We will define topologically étale morphisms. Within this class, we will focus on good étale morphisms. This is the type of étale morphisms along which we will be able to extend the coordinates that we will work with.

\begin{defn}\label{topttf}
A morphism of Huber rings $\varphi: A \rightarrow B$ is \emph{topologically of finite type} (resp. of \emph{finite presentation}) if it is adic and if there exists rings of definition $A_0$ and $B_0$ of $A$ and $B$ such that $\varphi(A_0) \subset B_0$ and
\begin{enumerate}
\item for every ideal of definition $I \subset A_0$, the induced map
\[
A_0/I \rightarrow B_0/ \varphi (I)B_0
\]
is of finite type (resp. finite presentation),
\item the map $A \otimes_{A_0} \widehat{B_0} \rightarrow \widehat{B}$ is of finite type (resp. finite presentation).
\end{enumerate}
\end{defn}

\begin{rem}
This definition is a variant of the definition 2.3.13 of \cite{abbes} that specifies only to formal schemes, which corresponds to adic rings in the affine case. The conditions here are weaker than the the algebraic version of the definition 1.8.18 of \cite{abbes}.
\end{rem}

\begin{exs}
\begin{enumerate}
\item For discrete rings, this definition reduces to the classical algebraic notion of finite type (resp. finite presentation).
\item Let \( \varphi: A \rightarrow B \) be a morphism of finite type (resp. finite presentation) between two rings. For an ideal \( I \) of finite type in \( A \), it is possible to equip \( A \) and \( B \) with the \( I \)-adic topology. Under this topology, the morphism \( \varphi \) is topologically of finite type (resp. finite presentation).

\item For a non-archimedean field $K$, the morphism $K\rightarrow K\left\lbrace \underline{X} \right\rbrace$ is topologically of finite presentation.
\end{enumerate}
\end{exs}

\begin{prop}\label{compotpf}
Let \( f: A \rightarrow B \) and \( g: B \rightarrow C \) be two morphisms between adic rings.
\begin{enumerate}
\item If both \( f \) and \( g \) are topologically of finite type (resp. finite presentation), then \( g \circ f \) is also topologically of finite type (resp. finite presentation).
\item If \( g \circ f \) is topologically of finite type and \( g \) is adic, then \( f \) is topologically of finite type.
\item If \( g \circ f \) is topologically of finite presentation and \( g \) is topologically of finite type, then \( f \) is topologically of finite presentation.
\end{enumerate}
\end{prop}

\begin{proof}
This is the affine version of the Proposition 2.3.18 of \cite{abbes}.
\end{proof}

\begin{prop}\label{ilexiste}
Let $\varphi: A \rightarrow B$ be an adic morphism between adic rings, the following statements are equivalent:
\begin{enumerate}
\item The morphism $\varphi$ is topologically of finite type (resp. finite presentation).
\item For every ideal of definition $I$ of $A$, the induced map 
\[
A/I \rightarrow B/ \varphi (I)B
\]
is of finite type (resp. finite presentation).
\end{enumerate}
\end{prop}

\begin{proof}
The implication \((2) \Rightarrow (1)\) is straightforward. It remains to prove \((1) \Rightarrow (2)\). Without loss of generality, we may assume that the rings \( A \) and \( B \) are complete by Proposition \ref{complété}. By hypothesis, the morphism is adic. Let \( A_0 \) and \( B_0 \) be rings of definition for \( A \) and \( B \) that satisfy the conditions of the definition. Take an ideal of definition $I_0$ of $A_0$. Then, for any $n \geq 1$, the induced map
$A_0/I_0^n \rightarrow B_0/\varphi(I_0^n)B_0$ is of finite type (resp. finite presentation). For every ideal of definition $I$ of $A$, take $n\geq 1$ such that $I_0^n \subset I \cap A_0$. Then, by extending scalars and using the second condition of the definition, we obtain that the map
\begin{align*}
    A/I &\simeq A/I \otimes_{A_0} A_0/I_0^n \rightarrow A/I \otimes_{A_0} B_0/\varphi(I_0^n)B_0  \simeq A/I \otimes_{A_0} \hat{B_0}/\varphi(I_0^n)\hat{B_0} \\
    & \simeq (A \otimes_{A_0}\hat{B_0})/\varphi(I)I(A \otimes_{A_0} \hat{B_0}) \rightarrow \hat{B}/\varphi(I)\hat{B} \simeq B/\varphi(I)B
\end{align*}
is of finite type (resp. finite presentation).
\end{proof}

\begin{prop}
An adic morphism \( A \rightarrow B \) between adic rings is topologically of finite presentation if and only if the induced completed morphism \( \widehat{A} \rightarrow \widehat{B} \) is topologically of finite presentation.
\end{prop}

\begin{proof} 
Let us first consider the converse. Assume that \( \widehat{A} \rightarrow \widehat{B} \) is topologically of finite presentation. Then, for every ideal of definition \( I \) of \( A \), the ideal \( I\widehat{A} \) is an ideal of definition of \( \widehat{A} \), as established by Proposition \ref{complété}. Also, $IB$ is an ideal of definition of $B$ and so $I \widehat{B}$ is an ideal of definition of $\widehat{B}$ again by Proposition \ref{complété}. The same proposition provides the following isomorphisms:
\[
\widehat{A}/I\widehat{A} \simeq A/I \quad \text{and} \quad \widehat{B}/I\widehat{B} \simeq B/IB.
\]
This ensures that for every ideal of definition \( I \) of \( A \), the morphism \( A/I \rightarrow B/IB \) is of finite presentation. Thus, by Proposition \ref{ilexiste}, $A \rightarrow B$ is topologically of finite presentation. 
Now, let us prove the direct implication. Suppose \( I' \) is an ideal of definition for \( \widehat{A} \). For every ideal of definition \( I \) of \( A \), there exists an integer \( n \in \mathbb{N} \) such that \( I^n\widehat{A} \subset I' \). Then, by using Proposition \ref{complété}, we see that the map $\hat{A}/I^n \hat{A} \rightarrow \hat{B}/I^n\hat{B}$ is of finite presentation, and so the map
\[
\hat{A}/I' = \hat{A}/I'\otimes_{\hat{A}}\hat{A}/I^n\hat{A} \rightarrow \hat{A}/I' \otimes_{\hat{A}}\hat{B}/I^n \hat{B} = \hat{B}/I' \hat{B}
\]
is also of finite presentation.
\end{proof}

\begin{lem}\label{extensionf}
Let \( A \), \( B \), and \( A' \) be adic rings, with \( A \rightarrow B \) a morphism that is topologically of finite presentation, and \( A \rightarrow A' \) an adic morphism. Then, the morphism \( A' \rightarrow A' \otimes_A B \) is topologically of finite presentation.
\end{lem}

\begin{proof}
Let $I$ be an ideal of definition of $A$ and let $I'$ be any ideal of definition of $A'$. 
Then, since the morphism $A \to A'$ is adic, there exists some $n \ge 1$ such that
$I^n A' \subset I'$. Since the map $A/I^n \to B/I^n B$ is of finite presentation by assumption, the map
\[
A'/I' \simeq A'/I' \otimes_A A/I^n \to A'/I' \otimes_A B/I^n B 
\simeq (A' \otimes_A B)/I'(A' \otimes_A B)
\]
is also of finite presentation.

\end{proof}

\begin{lem}\label{prodtenstpf}
Let \( R \) be an adic ring, and let \( A \), \( A' \), \( B \), and \( B' \) be \( R \)-adic algebras. If the morphisms \( A \rightarrow B \) and \( A' \rightarrow B' \) are topologically of finite presentation, then the induced morphism \( A \otimes_R A' \rightarrow B \otimes_R B' \) is also topologically of finite presentation.
\end{lem}

\begin{proof}
By extending scalars and applying Lemma \ref{extensionf}, the morphisms
\[
A \otimes_R A' \rightarrow B \otimes_R A' \quad \text{and} \quad B \otimes_R A' \rightarrow B \otimes_R B'
\]
are topologically of finite presentation. Using Proposition \ref{compotpf}, we conclude that the morphism \( A \otimes_R A' \rightarrow B \otimes_R B' \) is topologically of finite presentation.
\end{proof}

\begin{defn}
A morphism of adic rings $A \rightarrow B$ is \emph{formally étale} if for every commutative diagram of adic rings
\[
\xymatrix{ A \ar[d] \ar[r] & C \ar@{->>}[d] \\ B \ar[r] & C/\mathfrak{c} 
}
\]
where $C$ is a complete adic ring and $\mathfrak{c}$ is a nilpotent ideal, there exists a unique morphism $B \rightarrow C$ making the diagram commute.
The morphism \( A \rightarrow B \) is called \emph{topologically étale} if it is both formally étale and topologically of finite presentation.
\end{defn}

\begin{exs}
\begin{enumerate}
 \item In the case of discrete rings, this definition reduces to the classical notion of formally étale and étale morphisms.
\item Let \( \varphi: A \rightarrow B \) be an étale morphism between rings. For an ideal \( I \) of finite type in \( A \), one can equip \( A \) and \( B \) with the \( I \)-adic topology. In this case, the morphism \( \varphi \) is adic, topologically of finite presentation, and formally étale, making it topologically étale.
\item For a non-archimedean field $K$, the canonical morphism $K_0[\underline{X}]\rightarrow K_0\left\lbrace \underline{X} \right\rbrace$ is topologically étale, where $K_0$ is as in 4. of Examples \ref{huberexs}.
\end{enumerate}
\end{exs}

\begin{prop}
Let $A$ and $B$ be two adic rings and $\varphi: A \rightarrow B$ be an adic morphism. The following statements are equivalent
\begin{enumerate}
    \item The morphism $\varphi$ is formally étale.
    \item For every diagram of adic rings
\[
\xymatrix{ A \ar[d] \ar[r] & C \ar@{->>}[d] \\ B \ar[r] & C/\mathfrak{c} 
}
\]
where $C$ is a complete adic ring and $\mathfrak{c}$ is a topologically nilpotent ideal, there exists a unique morphism $B \rightarrow C$ that makes the diagram commute.
\end{enumerate}
\end{prop}

\begin{proof}
The implication \((2) \Rightarrow (1)\) is clear, as a nilpotent ideal is topologically nilpotent. It remains to prove \((1) \Rightarrow (2)\). Consider a commutative diagram of adic rings
\[
\xymatrix{ A \ar[d] \ar[r] & C \ar@{->>}[d] \\ B \ar[r] & C/\mathfrak{c} 
}
\]
where \( C \) is a complete adic ring and \( \mathfrak{c} \) is a topologically nilpotent ideal. 
For an ideal of definition $I$ of $C$ and every natural integer $k$, this diagram
induces the following commutative diagram

\[
\xymatrix{ A \ar[d] \ar[r] & C/I^kC \ar@{->>}[d] \\ B \ar[r] & (C/I^k)/\mathfrak{c}.
}
\]

Since $I^k$ is an ideal of definition of $C$, $\mathfrak{c}$ is nilpotent in $C/I^k$.
Furthermore, $C/I^k$ is complete for the discrete topology, so we can lift
the morphism to obtain $B \to C/I^k$. By passing to the completion, we obtain
the required morphism $B \to C$.
\end{proof}

\begin{prop}\label{composition}
Let \( f: A \rightarrow B \) and \( g: B \rightarrow C \) be two morphisms of adic rings. If both \( f \) and \( g \) are topologically étale, then the composition \( g \circ f \) is also topologically étale.
\end{prop}

\begin{proof}
Since the composition of two formally étale morphisms is itself formally étale, we can conclude that the morphism \( g \circ f \) is formally étale. By applying Proposition \ref{compotpf}, it follows that \( g \circ f \) is also topologically étale.
\end{proof}

\begin{prop}\label{tenseur}
Let \( R \) be an adic ring, and let \( A \), \( B \), \( C \), and \( D \) be \( R \)-adic algebras that are also adic rings. Suppose \( \varphi: A \rightarrow C \) and \( \psi: B \rightarrow D \) are topologically étale morphisms between adic rings. Then the induced morphism 
\[\varphi \otimes_R \psi: A \otimes_R B \rightarrow C \otimes_R D \]
is also topologically étale.
\end{prop}

\begin{proof}
Consider a commutative diagram of adic rings:
\[
\xymatrix{
A \otimes_R B \ar[d] \ar[r] & E \ar[d] \\
C \otimes_R D \ar[r] & E/\mathfrak{e}
}
\]
where \( E \) is complete and \( \mathfrak{e} \) is a nilpotent ideal. This diagram induces the following commutative diagrams:
\[
\xymatrix{
A \ar[d] \ar[r] & E \ar[d] \\
C \ar[r] & E/\mathfrak{e}
}
\quad \quad
\xymatrix{
B \ar[d] \ar[r] & E \ar[d] \\
D \ar[r] & E/\mathfrak{e}
}
\]
By hypothesis, there exist unique liftings \( C \rightarrow E \) and \( D \rightarrow E \) that make these diagrams commute. By the universal property of the tensor product, there exists a unique map \( C \otimes_R D \rightarrow E \) that makes the original diagram commute. The fact that the morphism is topologically of finite presentation follows from Proposition \ref{prodtenstpf}.
\end{proof}

\begin{defn}
A morphism $\varphi: A \rightarrow B$ of Huber rings is said to be \emph{good étale} if there exist rings of definition $A_0$ and $B_0$ of $A$ and $B$ satisfying $\varphi(A_0)\subset B_0$ such that the induced morphism $\varphi_0: A_0 \rightarrow B_0$ is topologically étale and $ B=A \otimes_{A_0} B_0$. The triplet $(\varphi_0,A_0,B_0)$ is said to be a \emph{good model} of $\varphi$. 
\end{defn}

\begin{exs}
\begin{enumerate}
    \item In the case of discrete rings, this reduces to the classical notion of an étale morphism.

    \item For a non-archimedean field $K$, the canonical morphism $K[\underline{X}]\rightarrow K\lbrace \underline{X} \rbrace$ is good étale.

\end{enumerate}
\end{exs}

\begin{prop}\label{id0}
Let $A$ and $B$ be two adic rings and $\varphi: A \rightarrow B$ a morphism. If $\varphi$ is good étale, then it is topologically étale.
\end{prop}

\begin{proof}
Let us consider a commutative diagram of adic rings 
\[
\xymatrix{ A \ar[d] \ar[r] & C \ar@{->>}[d] \\ B \ar[r] & C/\mathfrak{c} 
}
\]
with $C$ complete and $\mathfrak{c}$ nilpotent. For rings of definition $A_0$ and $B_0$ of $A$ and $B$ satisfying the requirements of the definition, we deduce the following diagram 
\[
\xymatrix{ A_0 \ar[d] \ar[r] & C \ar@{->>}[d] \\ B_0 \ar[r] & C/\mathfrak{c}. 
}
\]
By hypothesis, the exists a unique lifting $B_0 \rightarrow C$ that can be extended in a unique way to a map $B \rightarrow C$ using the fact that $B=A \otimes_{A_0} B_0$. Hence $\varphi$ is formally étale.

Moreover, since $A_0 \rightarrow B_0$ is topologically of finite presentation and the canonical inclusion $A_0 \rightarrow A$ is adic by Proposition \ref{moreladique}, $\varphi : A \rightarrow A \otimes_{A_0} B_0=B$ is topologically of finite presentation by Lemma \ref{extensionf}. Thus $\varphi$ is topologically étale.
\end{proof}

\begin{defn}
Let $A$ be a Huber ring. An endomorphism $\sigma$ of $A$ is said to be \emph{infinitesimal} over a subset $V$ of $A$ if there exists a topologically nilpotent subset $U$ of $A$ such that $(\sigma-\mathrm{Id})(V) \subset U$ and $\sigma(V) \subset V$.
\end{defn}

\begin{exam}
Let \( A \) be a Huber ring, and let \( \lVert \cdot \rVert \) denote its submultiplicative semi-norm. Let \( A_0 \) be a ring of definition. For an element \( q \in A \) such that \( \lVert q - 1 \rVert < 1 \), the map
\[
A[X] \rightarrow A[X], \quad X \mapsto qX
\]
is infinitesimal over \( A_0[X] \).
\end{exam}

\begin{prop}\label{releveinfi}
Let \( B \) be a complete Huber ring and $u:A \rightarrow B $ be a good étale morphism of Huber rings with \( (u_0, A_0, B_0) \) as a good model. If \( \sigma_A \) is an endomorphism of \( A \) that is infinitesimal over \( A_0 \), then \( \sigma_A \) can be uniquely extended to an endomorphism \( \sigma_B \) of \( B \), which is also infinitesimal over \( B_0 \).
\end{prop}

\begin{proof}
Let \( J \) denote the ideal of \( A_0 \) generated by \( (\sigma_A - \mathrm{Id})(A_0) \). The ideal \( JB_0 \) is topologically nilpotent in \( B_0 \). To obtain the desired extension, consider the following commutative diagram
\[
\xymatrix{
A_0 \ar[r]^-{u \circ \sigma_A} \ar[d]_-u & B_0 \ar[d] \\
B_0 \ar[r] & B_0/JB_0.
}
\]
Since the morphism \(u: A \rightarrow B \) is good étale, there exists a unique lifting \( B_0 \rightarrow B_0 \). By extending scalars, this morphism can be uniquely extended to an endomorphism of \( B \) that satisfies the desired properties.
\end{proof}

\begin{defn}
Let \( R \) be a Huber ring, and let \( A \) be a Huber \( R \)-algebra. Elements \( x_1, \ldots, x_d \in A \) are called \emph{good étale coordinates} if there exists a good étale morphism
\[
\varphi: R[X_1, \ldots, X_d] \rightarrow A, \quad X_i \mapsto x_i
\]
which admits a good model of the form
\[
\varphi_0: R_0[X_1, \ldots, X_d] \rightarrow A_0.
\]
We call a triple $(\varphi_0, R_0, A_0)$ as above a \emph{good model} of the good étale coordinates $x_1,\ldots,x_d \in A$.
\end{defn}

\begin{prop}
Let \( R \) be an adic ring, and let \( A \) and \( B \) be \( R \)-adic algebras that are also adic rings. Suppose \( A \) and \( B \) are endowed with good étale coordinates \( x_1, \ldots, x_d \) and \( y_1, \ldots, y_e \), respectively. Then the elements \( x_1 \otimes 1, \ldots, x_d \otimes 1, 1 \otimes y_1, \ldots, 1 \otimes y_e \) form good étale coordinates over \( A \otimes_R B \).
\end{prop}

\begin{proof}
This is a consequence of the Proposition \ref{tenseur}.
\end{proof}

\section{Twisted principal parts in several variables and coordinates}\label{Twisted principal parts}
\begin{defn}\label{defor}
Let $R$ be a Huber ring, $A$ an $R$-adic algebra and $d\in \mathbb{N}$. The $R$-algebra $A$ is said to be \emph{twisted of order $d$} if it is endowed with a sequence of $R$-linear continuous endomorphisms $\underline{\sigma}=\left(\sigma_{1},\ldots,\sigma_{d} \right)$ that commute with each other.
\end{defn}

Let \( P_A = A \otimes_R A \). We extend \( \sigma_{i} \) to \( P_A \) by defining \( \sigma_{i}(a \otimes b) = \sigma_{i}(a) \otimes b \). The pair consisting of the ring and the set of endomorphisms will be denoted as \( (A, \underline{\sigma}) \).

\begin{exs}\label{example}
Let $R$ be a Huber ring and $A$ be an $R$-adic algebra. The following are examples of twisted algebras:
\begin{enumerate}
    \item Let \( h_1, \ldots, h_d \) be elements of \( A \). The ring \( A[\underline{X}] \) with \( \underline{\sigma} = \left( \sigma_1, \ldots, \sigma_d \right) \), defined by \( \sigma_i(X_j) = X_j + \delta_{i,j} h_i \), is a twisted \( R \)-algebra of order \( d \). This corresponds to the situation of finite differences.

    \item Let $l \in \mathbb{N}$. The ring $A[\underline{X}]$ with $\underline{\sigma}=\left(\sigma_{1},\ldots,\sigma_{d} \right)$, defined by 
    \[
\forall a \in A, \forall i,j=1,\ldots,d, ~\sigma_i(aX_j)=aX_j^{l^{\delta_{i,j}}},
    \] is a twisted $R$-algebra of order $d$. This is the case of Malher's differences.
   \item In the case of $d$ variables, \( A = R[X_1, \ldots, X_d] \), we can consider the case where \( \sigma_i(X_j) = q_i^{\delta_{i,j}} X_j \) for \( q_1, \ldots,  q_d \in R \).

\end{enumerate}
\end{exs}

\begin{defn}
Let \( R \) be a Huber ring, and let \( (A, \underline{\sigma}) \) and \( (B, \underline{\tau}) \) be twisted \( R \)-adic algebras of order \( d \). A morphism of twisted Huber \( R \)-algebras of order \( d \), $u: (A, \underline{\sigma})\rightarrow (B, \underline{\tau})$, is a morphism of Huber \( R \)-algebras such that
\[
\forall i = 1, \ldots, d, \quad \tau_i \circ u = u \circ \sigma_i.
\]
\end{defn}

In what follows, we consider a Huber ring \( R \) and a twisted \( R \)-adic algebra \( (A, \underline{\sigma}) \) of order \( d \). Additionally, we consider some elements \( \underline{x} = (x_1, \ldots, x_d) \) of \( A \), on which various conditions will later be imposed. For convenience, we will use the same notation \( \sigma_i \) for the corresponding endomorphisms on the ring of polynomials $A[\underline{\xi}]:=A[\xi_1,\ldots, \xi_d]$ defined by
\[
\forall i,j \in \lbrace 1,\ldots,d \rbrace,~\sigma_i(\xi_j)= 
\begin{cases}
    \xi_i+x_i-\sigma_i(x_i),& \text{if }  j=i,\\
    \xi_j,              & \text{otherwise.}
\end{cases}
\]

\begin{defn} \label{module of twisted differential parts}
For $n \in \N$, the $A$-module $P_{A, (n)_{\underline{\sigma}, \underline{x}}}$ of \emph{twisted principal parts of order $n$} is defined by
\[
P_{A, (n)_{\underline{\sigma},\underline{x}}}:=A[\underline{\xi}] \Big/ \left(\underline{\xi}^{(\underline{k})_{\underline{\sigma},\underline{x}}} \text{ such that } |\underline{k}|=n+1\right),
\]
where for  $\underline{k} \in \mathbb{N}^d$
\[
\underline{\xi}^{(\underline{k})_{\underline{\sigma},\underline{x}}}:=\xi_1^{(k_1)_{\sigma_1}}\ldots\xi_d^{(k_d)_{\sigma_d}}~  \text{ and }~  \xi_i^{(k_i)_{\sigma_i}}:=\xi_i \sigma_i(\xi_i) \ldots \sigma_i^{k_i-1} (\xi_i).
\]
\end{defn}

\begin{rems}
\begin{enumerate}
    \item Following the definition, we have
\begin{equation}
\underline{\xi}^{(\underline{k})_{\underline{\sigma},\underline{x}}} = \prod_{i=1}^d \prod_{j=0}^{k_i-1} \left(\xi_i+x_i-\sigma_i^j(x_i)\right). \label{twist}
\end{equation}
\item In the case where $d=1$, this is the module specified by Le Stum and Quirós in \cite[Definition 1.5]{LSQ3}.
\item As noted in the Examples \ref{huberexs}, the ring of polynomials \( A[\underline{\xi}] \) can be given a structure of Huber ring, inducing the topology on \( A \). The ring \( P_{A, (n)_{\underline{\sigma}, \underline{x}}} \), when equipped with the quotient topology, also becomes a Huber ring. Henceforth, this topology will be assumed on \( P_{A, (n)_{\underline{\sigma}, \underline{x}}} \).
\end{enumerate}
\end{rems}

\begin{exam}
When $A=R\left[X_1,X_2\right]$, and $\sigma_i(X_j)=q_i^{\delta_{i,j}}X_j$ for $i,j=1,2$, where $q_1$ and $q_2$ are elements of $R$, we obtain the following description
\[
P_{A, (n)_{\underline{\sigma},\lbrace X_1, X_2 \rbrace }}=A[\xi_1,\xi_2] \Big/ \left( \prod_{i=1}^2 \prod_{j=0}^{k_i-1} (\xi_i+(1-q_i^j)X_i) \text{ such that } k_1+k_2=n+1\right).
\]
\end{exam}

\begin{prop}\label{surjection}
For every natural integer $n$, the composed morphism
\[ 
A[\underline{\xi}]_{\leq n} \hookrightarrow A[\underline{\xi}]\twoheadrightarrow P_{A, (n)_{\underline{\sigma}, \underline{x}}}
\]
is an isomorphism.
\end{prop}

\begin{proof}
Consider the morphism of $A$-modules
\begin{equation}\label{defoxi}
\begin{aligned}
 A[\underline{\xi}] &\rightarrow A[\underline{\xi}] \\
\underline{\xi}^{\underline{k}} &\mapsto \underline{\xi}^{(\underline{k})_{\underline{\sigma},\underline{x}}}.
\end{aligned}
\end{equation}
Formula (\ref{twist}) ensures that for all $\underline{k} \in \mathbb{N}^d$, the term $\underline{\xi}^{(\underline{k})_{\underline{\sigma},\underline{x}}}$ is of degree $|\underline{k}|$. Hence, the morphism (\ref{defoxi}) preserves the degree and induces for every $n \in \N$ a morphism 
\[
A[\underline{\xi}]_{\leq n} \rightarrow A[\underline{\xi}]_{\leq n},
\]
where $A[\underline{\xi}]_{\leq n}$ designates the set of polynomials of degree less than or equal to $n$. 
For $n=0$, this morphism is an isomorphism. By induction, assume that it is an isomorphism for some natural integer $n$. It induces a morphism
\[
A[\underline{\xi}]_{\leq n+1}/A[\underline{\xi}]_{\leq n} \rightarrow A[\underline{\xi}]_{\leq n+1}/A[\underline{\xi}]_{\leq n}.
\]
This morphism is an isomorphism. Indeed, given $\underline{k}=(k_1,\ldots,k_d) \in \mathbb{N}^d$ such that $|\underline{k}|=n+1$ by formula (\ref{twist}) we have $\underline{\xi}^{(\underline{k})_{\underline{\sigma},\underline{x}}} \equiv \xi_1^{k_1}\ldots \xi_d^{k_d} \text{ mod } A[\underline{\xi}]_{\leq n}  $.
The following diagram
\[
\xymatrix{
0 \ar[r] & A[\underline{\xi}]_{\leq n} \ar[d] \ar[r] & A[\underline{\xi}]_{\leq n+1} \ar[d] \ar[r] &A[\underline{\xi}]_{\leq n+1}/A[\underline{\xi}]_{\leq n} \ar[d] \ar[r] & 0 \\
0 \ar[r] & A[\underline{\xi}]_{\leq n}  \ar[r] & A[\underline{\xi}]_{\leq n+1}  \ar[r] & A[\underline{\xi}]_{\leq n+1}/A[\underline{\xi}]_{\leq n}  \ar[r] & 0
}
\]
allows us to conclude, using the Five lemma, that the morphism 
\[
A[\underline{\xi}]_{\leq n+1} \rightarrow A[\underline{\xi}]_{\leq n+1}
\]
is an isomorphism. Therefore, the morphism in the statement of the proposition is an isomorphism.
\end{proof}

\begin{rem}
The module $P_{A, (n)_{\underline{\sigma}, \underline{x}}}$ is a free $A$-module of basis $\left\lbrace\underline{\xi}^{(\underline{k})_{\underline{\sigma},\underline{x}}}\right\rbrace_{|\underline{k}|\leq n}.$
\end{rem}

\begin{prop}\label{u}
Let $u: (A, \underline{\sigma})\rightarrow (B, \underline{\tau})$ be a morphism of twisted $R$-adic algebras of order $d$ then, 
\[
B\otimes_A P_{A, (n)_{\underline{\sigma}, \underline{x}}}\simeq P_{B, (n)_{\underline{\tau}, \underline{u(x)}}}.
\]
\end{prop}

\begin{proof}
By definition,
\[
\begin{aligned}
B\otimes_A P_{A, (n)_{\underline{\sigma}, \underline{x}}}&=B\otimes_A A[\underline{\xi}] \Big/ \left(\underline{\xi}^{(\underline{k})_{\underline{\sigma},\underline{x}}} \text{ such that } |\underline{k}|=n+1\right)	\\
 &\simeq B[\underline{\xi}] \Big/ \left(\underline{\xi}^{(\underline{k})_{\underline{\tau},\underline{u(x)}}} \text{ such that } |\underline{k}|=n+1\right)= P_{B, (n)_{\underline{\tau}, \underline{u(x)}}}. \qedhere
\end{aligned}
\]
\end{proof}

\begin{defn}\label{coor}
The elements $\underline{x}$ are called \emph{$\underline{\sigma}$-coordinates} over $A$ if, for every $n \in \N$, there exists a unique adic morphism of $R$-algebras
\[
\begin{aligned}
\Theta_{A,(n)_{\underline{\sigma},\underline{x}}} \colon A &\to P_{A, (n)_{\underline{\sigma},\underline{x}}} \\
x_i  &\mapsto x_i+\xi_i
\end{aligned}
\]
such that the composition of the projection on $A$ with $\Theta_{A,(n)_{\underline{\sigma},\underline{x}}}$ is the identity. 
The morphism $\Theta_{A,(n)_{\underline{\sigma},\underline{x}}}$ is called the \emph{$n$-th Taylor morphism} of $A$ with respect to $\underline{\sigma}$ and $\underline{x}$.
We say that the coordinates $\underline{x}=(x_1,\ldots,x_d)$ are 
\emph{$\underline{q}$-coordinates} if there exists $q_1,\ldots,q_d \in R$ such that $\underline{x}$ are $\underline{\sigma}$-coordinates for the $\underline{\sigma}$ defined by
\[
\forall i,j=1, \ldots, d, \quad\sigma_i(x_j) = q_i^{\delta_{i,j}}x_j.
\]
For notational simplicity and to emphasize the link with elements $q_1,\ldots,q_d$, we will henceforth replace the subscript $\underline{\sigma}$ by $\underline{q}$ in the case of $\underline{q}$-coordinates. In the case $d=1$, this is the situation described in \cite{ADV04}.
\end{defn}

\begin{exam}
Let $A$ be a complete adic ring and $\sigma_1=\ldots=\sigma_d=\mathrm{Id}_A$. Suppose there exists a formally étale morphism
\[
R[\underline{X}] \rightarrow A, ~X_i \mapsto x_i,
\]
then, $x_1,\ldots,x_d$ are $\underline{\sigma}$-coordinates over $A$. Indeed, for every natural integer $n$ the following diagram commutes:
\[
\xymatrix{
R[\underline{X}] \ar[rr]^{X_i \mapsto x_i+\xi_i} \ar[d]_{X_i \mapsto x_i} & ~ & P_{A, (n)_{\underline{\sigma}, \underline{x}}} \ar[d] \\
A \ar[rr]  & ~ &P_{A, (n)_{\underline{\sigma}, \underline{x}}}/(\xi_1,\ldots,\xi_d)
}
\]
Since we are in the case where $\sigma_1=\ldots=\sigma_d=\mathrm{Id}_A$, the elements $\xi_i$ are nilpotent elements of $P_{A, (n)_{\underline{\sigma}, \underline{x}}}$. Given that the left vertical arrow is formally étale, there exists a unique lift $A \rightarrow P_{A, (n)_{\underline{\sigma}, \underline{x}}}$ that satisfies the desired property.
\end{exam}

Let \( x_1, \ldots, x_d \) be \( \underline{\sigma} \)-coordinates over \( A \). For \( n \in \mathbb{N} \), consider the extension of \( \Theta_{A,(n)_{\underline{\sigma},\underline{x}}} \) to \( P_A \) by \( A \)-linearity
\[
\begin{aligned}
\tilde{\Theta}_{A, (n)_{\underline{\sigma}, \underline{x}}}: & P_A \rightarrow P_{A, (n)_{\underline{\sigma}, \underline{x}}} \\
& a \otimes b \mapsto a \, \Theta_{A, (n)_{\underline{\sigma}, \underline{x}}}(b).
\end{aligned}
\]
For $i=1,\ldots,d,$
\[
\tilde{\Theta}_{A,(n)_{\underline{\sigma},\underline{x}}}(1\otimes x_i - x_i \otimes 1)= \xi_i.
\]
Proposition \ref{surjection} guarantees that this morphism is surjective. We define
\[
I_{A}^{(n+1)_{\underline{\sigma}, \underline{x}}} := \ker\left( \tilde{\Theta}_{A, (n)_{\underline{\sigma}, \underline{x}}} \right).
\]
By the First Isomorphism Theorem, we have
\[
P_A / I_{A}^{(n+1)_{\underline{\sigma}, \underline{x}}} \simeq P_{A, (n)_{\underline{\sigma}, \underline{x}}}.
\]

\begin{rem}
For every $n \in \N$, there exists, by factorization, a morphism 
\[
\phi_n: P_{A,(n)_{\underline{\sigma},\underline{x}}} \rightarrow P_{A,(n-1)_{\underline{\sigma},\underline{x}}}
\]
that makes the diagram commute
\[
\xymatrix{
A[\underline{\xi}] \ar@{->>}[r] \ar@{->>}[d] &  P_{A ,(n-1)_{\underline{\sigma},\underline{x}}}. \\
 P_{A ,(n)_{\underline{\sigma},\underline{x}}} \ar[ur]_-{\phi_n} & ~ 
}
\]
Using the uniqueness of
 $\tilde{\Theta}_{A,(n-1)_{\underline{\sigma},\underline{x}}}$ we can easily check that the diagram below is commutative:
\[
\xymatrix{
0 \ar[r] & I_{A}^{(n)_{\underline{\sigma},\underline{x}}} \ar@{^{(}->}[r] & P_A \ar[rr]^-{\tilde{\Theta}_{A,(n-1)_{\underline{\sigma}}}} & & P_{A,(n-1)_{\underline{\sigma},\underline{x}}} \ar[r]  & 0 \\
0 \ar[r] & I_{A}^{(n+1)_{\underline{\sigma},\underline{x}}} \ar[u] \ar@{^{(}->}[r] & P_A \ar@{=}[u] \ar[rr]^-{\tilde{\Theta}_{A,(n)_{\underline{\sigma}}}} &  & P_{A,(n)_{\underline{\sigma},\underline{x}}} \ar[u]_{\phi_n} \ar[r] & 0.
}
\]
\end{rem}

\begin{lem}\label{inclusionideaux}
Let  $\underline{x}=(x_1,\ldots,x_d)$ be $\underline{\sigma}$-coordinates such that for $i=1,\ldots,d, ~\sigma_i(x_i) \in R[x_i]$. For every $n,m \in \N$, there exists a unique morphism of algebras
\[
\delta_{n,m}\colon P_{A, (n+m)_{\underline{\sigma},\underline{x}}}  \rightarrow P_{A, (n)_{\underline{\sigma},\underline{x}}} \otimes_A' P_{A, (m)_{\underline{\sigma},\underline{x}}}, ~\xi_i  \mapsto \xi_i \otimes 1 + 1 \otimes \xi_i,
\]
where $\otimes_A'$ means that the action on the left is given by $\Theta_{A,(n)_{\underline{\sigma},\underline{x}}}$.
\end{lem}

\begin{proof}
There exists a unique morphism of $A$-algebras
\[
\delta: A[\underline{\xi} ] \rightarrow  P_{A, (n)_{\underline{\sigma},\underline{x}}} \otimes_A' P_{A, (m)_{\underline{\sigma},\underline{x}}}, ~\xi_i \mapsto 1\otimes \xi_i + \xi_i \otimes 1.
\]
For $i=1,\ldots,d$ and $l \in \N$ denote by $\mathfrak{a}_{i,l}$ the free sub-module with basis $\left\lbrace \xi_i^{(k)_{\underline{\sigma},\underline{x}}}, k \geq l \right\rbrace$ and denote by $\bar{\mathfrak{a}}_{i,l}$ the image of $\mathfrak{a}_{i,l}$ in $P_{A,(n)_{\underline{\sigma}, \underline{x}}}$ or $P_{A,(m)_{\underline{\sigma}, \underline{x}}}$.
First we prove the inclusion
\begin{equation}\label{1}    
\delta (\mathfrak{a}_{i,l}) \subset \sum\limits_{j=0}^l \bar{\mathfrak{a}}_{i,j} \otimes_A' \bar{\mathfrak{a}}_{i,l-j}.
\end{equation}
Since the right hand side is decreasing with respect to $l$, it suffices to prove that
\begin{equation}\label{2}
   \delta (\xi_i^{(l)_{\underline{\sigma}, \underline{x}}}) \subset \sum\limits_{j=0}^l \bar{\mathfrak{a}}_{i,j} \otimes_A' \bar{\mathfrak{a}}_{i,l-j}. 
\end{equation}
We prove it by induction on $l$. The case $l=0$ is trivial. Suppose (\ref{2}) holds for $l$. To prove (\ref{2}) with $l$ replaced by $l+1$, we prove that 
\begin{equation}\label{eq:3}
    \delta\left(\sigma_i^l(\xi_i)\right)\in  P_{A, (n)_{\underline{\sigma},\underline{x}}} \otimes_A' (\sigma_i^{j}(\xi_i)) + (\sigma_i^{l-j}(\xi_i)) \otimes_A' P_{A, (m)_{\underline{\sigma},\underline{x}}}
\end{equation}
holds for any $0 \leq j \leq l.$ Since $\sigma_i(x_i) \in R[x_i]$, we may assume $A = R[x_i]$ to prove (\ref{eq:3}). In this case, there exists an isomorphism $\iota : A[\underline{\xi}]=A[\xi_i] \rightarrow P_A$ defined by $\xi_i \mapsto 1 \otimes x_i - x_i \otimes 1$. If we define the action of $\sigma_i$ on $P_A =A \otimes_R A$ by $\sigma_i \otimes \text{Id}$, the isomorphism $\iota$ is compatible with the action of $\sigma_i$. Moreover, if we define the morphisms 
$\delta_A \colon A[\xi_i] \to A[\xi_i] \otimes_A' A[\xi_i]$ and 
$\delta_P \colon P_A \to P_A \otimes_A' P_A$ 
by $\delta_A(\xi_i) = \xi_i \otimes 1 + 1 \otimes \xi_i$ and 
$\delta_P(a \otimes b) = (a \otimes 1) \otimes (1 \otimes b)$ respectively,
the equality $(\iota \otimes \iota) \circ \delta_A = \delta_P \circ \iota$ holds.
Here $\otimes_A'$ means that the $A$-module structures on the left factor of
$P_A$ and $A[\xi_i] \simeq P_A$ are defined by the right $A$-action. By \cite[Lemma 3.1]{LSQ3}, we have the equality
\begin{equation*}
\delta_P(\sigma_i^l(1\otimes x_i - x_i \otimes 1))=1\otimes'\sigma_i^j(1\otimes x_i - x_i \otimes 1)+\sigma_i^{l-j}(1\otimes \sigma_i^j(x_i)-\sigma_i^j(x_i) \otimes 1)\otimes'1
\end{equation*}
in $P_A \otimes_A'P_A$. If we put $\eta_i:=\iota^{-1}(1\otimes \sigma_i^j(x_i)-\sigma_i^j(x_i) \otimes 1) \in A[\xi_i]$, the above equality implies the equality
\begin{equation}\label{4}
\delta_A(\sigma_i^l(\xi_i))=1\otimes'\sigma_i^j(\xi_i)+\sigma_i^{l-j}(\eta_i)\otimes'1
\end{equation}
in $A[\xi_i]$. Moreover, noting that the projection $A[\xi_i] \rightarrow A$ is compatible with the multiplication map $P_A = A \otimes_RA \rightarrow A$ via $\iota$, we have
\[
\eta_i=\iota^{-1}(1\otimes \sigma_i^j(x_i)-\sigma_i^j(x_i) \otimes 1)  \in \iota^{-1}(\text{Ker}(P_A \rightarrow A)) = \text{Ker}(A[\xi_i] \rightarrow A) = (\xi_i).
\]
Because the map $\delta_A$ is compatible with the map $\delta$ via the natural
projections
\[
A[\xi_i] \otimes_{A}' A[\xi_i] \twoheadrightarrow P_{A,(n)_{\underline{\sigma},\underline{x}}} \otimes_A' P_{A,(m)_{\underline{\sigma},\underline{x}}},
\] 
(\ref{4}) implies (\ref{eq:3}). Then, by induction hypothesis,
\[
\begin{aligned}
\delta(\xi_i^{(l+1)_{\underline{\sigma},\underline{x}}})&=\delta(\xi_i^{(l)_{\underline{\sigma},\underline{x}}}) \delta(\sigma_i^l(\xi_i)) \\
& \in  \sum\limits_{j=0}^l \left(\bar{\mathfrak{a}}_{i,j} \otimes_A' \bar{\mathfrak{a}}_{i,l-j}\right)\left(P_{A, (n)_{\underline{\sigma},\underline{x}}} \otimes_A' (\sigma_i^{l-j}(\xi_i)) + (\sigma_i^j(\xi_i)) \otimes_A' P_{A, (m)_{\underline{\sigma},\underline{x}}} \right) \\
&=\sum\limits_{j=0}^l \left(  \bar{\mathfrak{a}}_{i,j} \otimes_A' \bar{\mathfrak{a}}_{i,l+1-j}  +  \bar{\mathfrak{a}}_{i,j+1} \otimes_A' \bar{\mathfrak{a}}_{i,l-j} \right)  = \sum\limits_{j=0}^{l+1}  \bar{\mathfrak{a}}_{i,j} \otimes_A' \bar{\mathfrak{a}}_{i,l+1-j}.
\end{aligned}
\]
So we have proved (\ref{2}), hence (\ref{1}).

Now let $\mathfrak{a}_l$ denote the free $A$-sub-module of $A[\underline{\xi}]$ with basis $\left\lbrace \underline{\xi}^{(\underline{k})_{\underline{\sigma},\underline{x}}}, |\underline{k}|\geq l \right\rbrace$ and denote by $\bar{\mathfrak{a}}_{l}$ the image of $\mathfrak{a}_{l}$ in $P_{A,(n)_{\underline{\sigma}, \underline{x}}}$ or $P_{A,(m)_{\underline{\sigma}, \underline{x}}}$.
 We will show that
\[
\delta (\mathfrak{a}_{l}) \subset \sum\limits_{j=0}^l \bar{\mathfrak{a}}_{j} \otimes_A' \bar{\mathfrak{a}}_{l-j}.
\]
The equality $\mathfrak{a}_l=\sum_{|\underline{k}|=n} \prod_{i=1}^d \mathfrak{a}_{i,k_i}$ gives
\begin{align*}
\delta (\mathfrak{a}_l) 
&= \delta \left(\sum_{|\underline{k}|=l} \prod_{i=1}^d \mathfrak{a}_{i,k_i} \right)\\
& \subset \sum_{|\underline{k}|=l} \prod_{i=1}^d \delta (\mathfrak{a}_{i,k_i}) \\
& \subset \sum_{|\underline{k}|=l} \prod_{i=1}^d \sum\limits_{j=0}^{k_i} \bar{\mathfrak{a}}_{k_i,j} \otimes_A' \bar{\mathfrak{a}}_{i,k_i-j} \\
& \subset \sum_{|\underline{k}|=l} \sum_{|\underline{j}| \leq |\underline{k}|} \prod_{i=1}^d (\bar{\mathfrak{a}}_{i,j_i} \otimes_A' \bar{\mathfrak{a}}_{i,k_i-j_i} ) \\
& \subset \sum_{|\underline{k}|=l} \sum_{|\underline{j}| \leq |\underline{k}|} \left( \prod_{i=1}^d \bar{\mathfrak{a}}_{i,j_i} \otimes_A' \prod_{i=1}^d \bar{\mathfrak{a}}_{i,k_i-j_i} \right) \\
& \subset \sum_{|\underline{k}|=l} \sum_{|\underline{j}| \leq |\underline{k}|} \left( \bar{\mathfrak{a}}_{|\underline{j}|}\otimes_A' \bar{\mathfrak{a}}_{|\underline{k-j}|} \right) \\
& \subset  \sum_{j=0}^l \bar{\mathfrak{a}}_j \otimes_A' \bar{\mathfrak{a}}_{l-j}.
\end{align*}
Apply this inclusion to the case $l=n+m+1$:
For $0 \leq i \leq n+m+1$, either $i>n$, in which case $\bar{\mathfrak{a}}_i \subset \bar{\mathfrak{a}}_{n+1}$, or $j \leq n$, in which case $m+n+1-j >m$ and so $\bar{\mathfrak{a}}_{n+m+1-j} \subset \mathfrak{a}_{m+1}$. Hence
\[
\delta (\mathfrak{a}_{n+m+1} )\subset \bar{\mathfrak{a}}_{n+1} \otimes_A' P_{A,(m)_{\underline{\sigma}, \underline{x}}} + P_{A,(n)_{\underline{\sigma}, \underline{x}}}\otimes_A'\bar{\mathfrak{a}}_{m+1} =0.
\]
Hence, passing to the quotient, we obtain the desired map.\qedhere
\end{proof}

\begin{rem}
The map $\delta_{n,m}$ is automatically continuous. 
\end{rem}

We present the hypotheses that will be utilized throughout the rest of this section. Let \( R \) be a Huber ring, and let \( (A, \underline{\sigma}) \) denote a twisted \( R \)-adic algebra of order \( d \) (see Definition \ref{defor}). For notational convenience, we will omit the explicit dependence on the coordinates \( \underline{x} \) and use this notation to represent a tuple \( (x_1, \ldots, x_d) \) of elements in \( A \).

\begin{defn}\label{goodcoor}
The elements $\underline{x}$ are said to be \emph{good $\underline{\sigma}$-coordinates} if they are good \'{e}tale coordinates with a good model $(\varphi_0, R_0, A_0)$ such that $\sigma_i(A_0) \subset A_0$ for $i = 1, \ldots, d$ and they are $\underline{\sigma}$-coordinates of the twisted $R_0$-adic algebra $(A_0,\underline{\sigma})$.

By definition, we have a morphism of $R_0$-algebras
\[
\Theta_{A_0,(n)_{\underline{\sigma}}} : A_0 \to P_{A_0,(n)_{\underline{\sigma}}}; 
\qquad x_i \mapsto x_i + \xi_i
\]
such that the composition of the projection to $A_0$ with $\Theta_{A_0,(n)_{\underline{\sigma}}}$ is the identity. We denote the morphism
\[
\mathrm{Id} \otimes \Theta_{A_0,(n)_{\underline{\sigma}}} : 
A = R \otimes_{R_0} A_0 \longrightarrow 
R \otimes_{R_0} P_{A_0,(n)_{\underline{\sigma}}} 
= P_{A,(n)_{\underline{\sigma}}}
\]
by $\Theta_{A,(n)_{\underline{\sigma}}}$.
\end{defn}

\begin{prop}\label{bonsigmacord}
Assume that \( A \) is complete. Let \( \underline{x} \) be good étale coordinates over \( A \), and let \( (\varphi_0, R_0, A_0) \) denote the associated good model. If for all \( i = 1, \ldots, d \), the morphism \( \sigma_i \) is infinitesimal over \( A_0 \), then \( \underline{x} \) are good \( \underline{\sigma} \)-coordinates over \( A \).
\end{prop}

\begin{proof}
By definition, the morphism $\varphi_0 : R[X] \to A_0$ is topologically \'{e}tale. For each $i=1,\ldots,d$, the ideal $J_i$ of $A_0$ generated by $(\sigma_i - \mathrm{Id})(A_0)$ is topologically nilpotent. Hence $J:= \sum_{i=1}^d J_i P_{A_0,(n)_{\underline{\sigma}}}$ is a topologically nilpotent ideal of $P_{A_0,(n)_{\underline{\sigma}}}$. Consider the following commutative diagram:
\[
\xymatrix{
R_0[\underline{X}] \ar[rr]^{X_i \mapsto x_i + \xi_i} \ar[d] 
&  &P_{A_0,(n)_{\underline{\sigma}}} \ar[d] \\
A_0 \ar[rr] 
&  & A_0
}
\]
where the right vertical arrow is the projection. Note that the kernel $(\xi_1,\ldots,\xi_d)$ of this projection is topologically nilpotent. Indeed, for each $i=1,\ldots,d$, the equality $\xi_i^{(n+1)_{\underline{\sigma}}}=0$ implies that $\xi_i^{n+1} \in J$. So the above diagram ensures that there exists a unique morphism $A_0 \to P_{A_0,(n)_{\underline{\sigma}}}$ such that the composition of the projection to $A_0$ with it is the identity. So $\underline{x}$ are good $\underline{\sigma}$-coordinates over $A$.
\end{proof}

\begin{prop}\label{need?}
Let $u: A \rightarrow B$ be a good étale morphism of complete $R$-adic algebras and let $\underline{x}$ be good $\underline{\sigma}$-coordinates, with $(\varphi_0,R_0,A_0)$ the associated good model. Assume that there exists a ring of definition $B_0$ of $B$ such that $(u_{|A_0},A_0,B_0)$ forms a good model for $u$. If for each $i=1, \ldots, d,$ the endomorphism $\sigma_i$ is infinitesimal over $A_0$, then, it extends uniquely to $B$ as an endomorphism $\sigma_{i,B}$. Furthermore, $u(x_1),\ldots,u(x_d)$ are good $\underline{\sigma_B}$-coordinates over $B$.
\end{prop}

\begin{proof}
The result can be easily deduced from the Propositions \ref{composition}, \ref{releveinfi} and \ref{bonsigmacord}.
\end{proof}

\begin{rem}
    If $\underline{x}$ are good $\underline{\sigma}$-coordinates of an $R$-adic algebra $A$ with a good model $(\varphi_0,R_0,A_0)$, they are $\underline{\sigma}$-coordinates of an $R_0$-adic algebra $A_0$ by definition, but we do not know if they are $\underline{\sigma}$-coordinates of an $R$-adic algebra $A$. In the following, we will write ``$\underline{x}$ are (good) $\underline{\sigma}$-coordinates of an $R$-adic algebra $A$" if $\underline{x}$ are $\underline{\sigma}$-coordinates of an $R$-adic algebra $A$ or they are good $\underline{\sigma}$-coordinates of an $R$-adic algebra $A$.
\end{rem}

\section{Twisted Connections}\label{Derivations}
We maintain the hypothesis introduced in the previous section: let \( R \) be a Huber ring, and let \( (A, \underline{\sigma}) \) be a twisted \( R \)-adic algebra of order \( d \) (Definition \ref{defor}). In the following discussion, we further assume the existence of some elements \( \underline{x} = (x_1, \ldots, x_d) \) in \( A \), which are (good) \( \underline{\sigma} \)-coordinates (Definitions \ref{coor}, \ref{goodcoor}).

\begin{defn}\label{twisted differential forms}
The $A$-module of \emph{twisted differential forms} of $A$ over $R$ is defined as
\[
\Omega^1_{A, \underline{\sigma}}:=I_{A}^{(1)_{\underline{\sigma}}}/I_{A}^{(2)_{\underline{\sigma}}}.
\]
\end{defn}

\begin{rems}
\begin{enumerate}
\item The module \( \Omega^1_{A, \underline{\sigma}} \) is a free module of rank \( d \), with \( (1 \otimes x_i - x_i \otimes 1)_{i=1,\ldots,d} \) serving as a basis. Moreover, for \( i=1, \ldots, r \), the element \( \xi_i \) in \( P_{A,(1)_{\underline{\sigma}}} \) can be identified with the equivalence class of \( 1 \otimes x_i - x_i \otimes 1 \) in \( P_A / I_A^{(2)_{\underline{\sigma}}} \). Indeed, for $n \in \N$, the next diagram is commutative and its rows and columns are exact:
\[
\xymatrix{
& & & 0  \ar[d]& \\
0 \ar[r] & 0 \ar[d] \ar[r] & 0 \ar[r] \ar[d] & \oplus_{|\underline{k}|=n}A\underline{\xi}^{(\underline{k})_{\underline{\sigma},\underline{x}}} \ar[d] \\
0 \ar[r] & I_{A}^{(n+1)_{\underline{\sigma},\underline{x}}} \ar[r]  \ar[d] & P_A \ar[r] \ar[d] & P_{A, (n)_{\underline{\sigma},\underline{x}}} \ar[d] \ar[r] & 0 \\
0 \ar[r] & I_{A}^{(n)_{\underline{\sigma},\underline{x}}} \ar[r]  \ar[d] & P_A \ar[r] \ar[d] & P_{A ,(n-1)_{\underline{\sigma},\underline{x}}} \ar[d] \ar[r] & 0 \\
& I_{A}^{(n)_{\underline{\sigma},\underline{x}}}/I_{A}^{(n+1)_{\underline{\sigma},\underline{x}}} \ar[r] \ar[d] & 0 \ar[r] & 0 \ar[r] & 0. \\
& 0 & & &
}
\]
By the Snake lemma, we obtain $\oplus_{|\underline{k}|=n}A\underline{\xi}^{(\underline{k})_{\underline{\sigma},\underline{x}}} \simeq I_{A}^{(n)_{\underline{\sigma},\underline{x}}}/I_{A}^{(n+1)_{\underline{\sigma},\underline{x}}}$.
\item Using the notations established in the proof of Lemma \ref{inclusionideaux}, we obtain the following:
\[
\Omega^1_{A, \underline{\sigma}}=\mathfrak{a}_1/\mathfrak{a}_2.
\]
\item The isomorphism $P_{A,(1)_{\underline{\sigma}}} \simeq P_A/I_{A}^{(2)_{\underline{\sigma}}}$ allows us to identify $\Omega^1_{A, \underline{\sigma}}$ with an ideal of $P_{A,(1)_{\underline{\sigma}}}$.
\end{enumerate}
\end{rems}

\begin{prop}\label{suite exacte}
The following sequence of $A$-modules,
\[
0 \rightarrow \Omega^1_{A, \underline{\sigma}} \rightarrow P_{A,(1)_{\underline{\sigma}}} \rightarrow A  \rightarrow  0
\]
where the first morphism $\Omega^1_{A, \underline{\sigma}} \rightarrow P_{A,(1)_{\underline{\sigma}}}$ is the inclusion and the morphism $P_{A,(1)_{\underline{\sigma}}} \rightarrow A$ is the canonical projection, is a split exact sequence.
\end{prop}

\begin{proof}
The exactness of the sequence follows directly from its construction. To demonstrate that it is split exact, it suffices to consider the map \( A \rightarrow P_{A}/I_{A}^{(2)_{\underline{\sigma}}} \simeq P_{A,(1)_{\underline{\sigma}}} \), which sends an element \( a \) to \( a \otimes 1 \).
\end{proof}

\begin{prop}\label{extension}
Let \( A \) and \( B \) be two twisted \( R \)-adic algebras of order \( d \), and let \( u: A \rightarrow B \) be a morphism of twisted \( R \)-adic algebras of  order $d$. Suppose that \( u(x_1), \ldots, u(x_d) \) are \( \underline{\sigma_B} \)-coordinates. Then, there exists an isomorphism
\[
B \otimes_A \Omega^1_{A, \underline{\sigma_A}} \simeq \Omega^1_{B, \underline{\sigma_B}}.
\]
\end{prop}

\begin{proof}
By applying the tensor product and using Propositions \ref{u} and \ref{suite exacte}, we obtain the following exact sequence
\[
0 \rightarrow B \otimes_A \Omega^1_{A, \underline{\sigma_A}} \rightarrow P_{B,(1)_{\underline{\sigma_B}}} \rightarrow B \rightarrow 0. \qedhere
\]
\end{proof}

\begin{defn}\label{sigma derivation}
Let $1 \leq i \leq d$. We recall from \cite{SQ2} that a \emph{$\sigma_i$-derivation} $D$ of $A$ is an $R$-linear endomorphism on $A$ that verifies the twisted Leibniz rule
\[
\forall x,y \in A, ~ D(xy)=xD(y)+\sigma_i(y)D(x).
\]
\end{defn}

\begin{rem} \label{defofd}
Denote by \( (e_i)_{i=1,\ldots,d} := ((1 \otimes x_i - x_i \otimes 1)^*)_{i=1,\ldots,d} \) the basis of \( \mathrm{Hom}_A\left(\Omega^1_{A, \underline{\sigma}}, A\right) \) dual to the basis \( (1 \otimes x_i - x_i \otimes 1)_{i=1,\ldots,d} \) of \( \Omega^1_{A, \underline{\sigma}} \). Define the map
\[
\text{d}: A \rightarrow \Omega^1_{A, \underline{\sigma}}, \quad f \mapsto \Theta_{A,(1)_{\underline{\sigma}}}(f) - f.
\]
This map allows us to define the following
\[
\forall i = 1 , \ldots,d; \quad \partial_{\underline{\sigma},i} := e_i \circ \text{d}.
\]

We denote by \( \mathrm{T}_{A,\underline{\sigma}} \) the left $A$-submodule of $\text{Hom}_R(A,A)$ generated by $\partial_{\underline{\sigma},1}, \ldots, \partial_{\underline{\sigma},d}$.  This module is free over \( A \): For any \( a_1, \ldots, a_d \in A \) such that
$\sum_{i=1}^d a_i \partial_{\underline{\sigma},i} = 0,$
evaluating this sum at \( x_1, \ldots, x_d \) gives \( a_1 = \ldots = a_d = 0 \). Also, define \( \mathrm{T}_{A,\underline{\sigma}}\text{-}\mathrm{Mod}(A) \) as the category of \( A \)-modules endowed with $R$-linear actions $\partial_{M,\underline{\sigma},i}:M \rightarrow M$ for $i=1,\ldots,d$ satisfying the following form of twisted Leibniz rule:
\[
\forall a \in A, \forall m \in M, ~ \partial_{M, \underline{\sigma},i}(am)=\partial_{ \underline{\sigma},i}(a)m+\sigma_i(a)\partial_{M, \underline{\sigma},i}(m).
\]

\end {rem}

\begin{lem}\label{formder}
The following formula is satisfied
\[
\forall f,g \in A, ~ \partial_{\underline{\sigma},i}(fg)= f\partial_{\underline{\sigma},i}(g)+(g+(\sigma_i(x_i)-x_i)\partial_{\underline{\sigma},i}(g))\partial_{\underline{\sigma},i}(f).
\]
\end{lem}

\begin{proof}
By the definition of the maps $\partial_{\underline{\sigma},1}, \ldots, \partial_{\underline{\sigma},d}$ we have a ring morphism
\[
\Theta_{A,(1)_{\underline{\sigma}}}:A \rightarrow P_{A, (1)_{\underline{\sigma}}},~ f \mapsto f+\sum\limits^d_{i=1} \partial_{\underline{\sigma},i}(f)\xi_i.
\]
In $P_{A, (1)_{\underline{\sigma}}}$ for $i \neq j$, $\xi_i\xi_j=0$ and $\sigma_i\left(\xi_i\right)\xi_i=0$ which implies
\[
\xi_i(\xi_i+x_i-\sigma_i(x_i))=0.
\]
Thus, we have
\[
\xi_i^2=(\sigma_i(x_i)-x_i)\xi_i.
\]
Now, for $f,g \in A$, on one hand we have
\[\Theta_{A,(1)_{\underline{\sigma}}}(fg)=fg+\sum\limits^d_{i=1} \partial_{\underline{\sigma},i}(fg)\xi_i,
\]
and on the other hand
\[
\Theta_{A,(1)_{\underline{\sigma}}}(fg)=\Theta_{A,(1)_{\underline{\sigma}}}(f)\Theta_{A,(1)_{\underline{\sigma}}}(g),
\]
which gives
\begin{align*}
\Theta_{A,(1)_{\underline{\sigma}}}(fg)&=\left(f+\sum\limits^d_{i=1} \partial_{\underline{\sigma},i}(f)\xi_i \right)\left(g+\sum\limits^d_{i=1} \partial_{\underline{\sigma},i}(g)\xi_i \right) \\
&=fg+g\sum\limits^d_{i=1} \partial_{\underline{\sigma},i}(f)\xi_i + f \sum\limits^d_{i=1} \partial_{\underline{\sigma},i}(g)\xi_i + \sum\limits_{i,j=1}^d  \partial_{\underline{\sigma},i}(f)\partial_{\underline{\sigma},j}(g)\xi_i\xi_j \\
&=fg + \sum\limits^d_{i=1} (f\partial_{\underline{\sigma},i}(g)+g\partial_{\underline{\sigma},i}(f))\xi_i+\sum\limits_{i=1}^d \partial_{\underline{\sigma},i}(g)\partial_{\underline{\sigma},i}(f)(\sigma_i(x_i)-x_i)\xi_i \\
&=fg+\sum\limits_{i=1}^d(f\partial_{\underline{\sigma},i}(g)+(g+(\sigma_i(x_i)-x_i)\partial_{\underline{\sigma},i}(g))\partial_{\underline{\sigma},i}(f))\xi_i.
\end{align*}
By identifying both sides, we derive the desired formula.
\end{proof}

\begin{rem}
For \( 1 \leq i \leq d \), the map \( \partial_{\underline{\sigma},i} \) is automatically continuous, as it is a component of \( \Theta_{A,(1)_{\underline{\sigma}}} \), which is assumed to be continuous by hypothesis.
\end{rem}

Unlike the situation in \cite{GLQ4}, it is not immediately evident that the endomorphism $\partial_{\underline{\sigma},i}$ on $A$ is a $\sigma_i$-derivation. To ensure this, we must impose an additional hypothesis on our coordinates.

\begin{defn}\label{classique}
The (good) $\underline{\sigma}$-coordinates $\underline{x}$ are said to be \emph{classical} if
\[
\forall i=1, \ldots,d;~ \forall f \in A, ~\sigma_i(f)=f+(\sigma_i(x_i)-x_i)\partial_{\underline{\sigma},i}(f).
\]
\end{defn}

\begin{rem}
When $\sigma_i(x_i)-x_i$ is invertible, we recover the formula
\[
\partial_{\underline{\sigma},i}(f)=\frac{\sigma_i(f)-f}{\sigma_i(x_i)-x_i}.
\]
\end{rem}

\begin{thm}\label{derivation}
The (good) $\underline{\sigma}$-coordinates $\underline{x}$ are classical coordinates if and only if for each $i\in \lbrace 1,\ldots,d \rbrace;$ $\partial_{\underline{\sigma},i}$ is a $\sigma_i$-derivation. 
\end{thm}

\begin{proof}
This follows from Lemma \ref{formder}.
\end{proof}

\begin{prop}\label{restriction}
Assume that $A$ is complete, and let $\underline{x}$ be good \'etale coordinates associated with a good model $(\varphi_0, R_0, A_0)$. Assume that $\sigma$ is an endomorphism of $A$ which is infinitesimal over $A_0$. Then, the endomorphism $\sigma$ is uniquely determined by its restriction to $R[X]$: if $\tau$ is an $R$-linear endomorphism of $A$ that is infinitesimal over $A_0$ such that $\tau|_{R[X]} = \sigma|_{R[X]}$, then $\tau = \sigma$.
\end{prop}

\begin{proof}
Given the following diagram
\[
\xymatrix{
R_0[X] \ar[r]^{X_i \mapsto \sigma(x_i)} \ar[d]_{X_i \mapsto x_i}
& A_0 \ar[d] \\
A_0 \ar[r]
& A_0 \big/ \big( (\sigma - \mathrm{Id})(A_0) + (\tau - \mathrm{Id})(A_0) \big)
}
\]
The left vertical morphism being formally \'etale, there exists a unique lifting from $A_0$ to $A_0$ that agrees with $\sigma$ on $R_0[X]$. This lifting uniquely extends to $A$ using the fact that $A = A_0 \otimes_{R_0} R$, and it can be easily verified that the extension is indeed equal to $\sigma$ and also equal to $\tau$.
\end{proof}

\begin{defn}
    Elements $\underline{x}$ of $A$ are called \emph{transversal} coordinates with respect to $\underline{\sigma}$ if $\sigma_i(x_j)=0$ for every $1 \leq i,j \leq d, i \neq j.$
\end{defn}

\begin{prop}
Assume that \( A \) is complete. If \( \underline{x} \) are transversal good $\underline{\sigma}$-coordinates over \( A \) associated with a good model \( (u_0, R_0, A_0) \), and for all \( i = 1, \ldots, d \), \( \sigma_i \) is infinitesimal over \( A_0 \), then \( \underline{x} \) are classical coordinates.
\end{prop}

\begin{proof}
Define
\[
\forall i = 1, \ldots, d,~ \forall f \in A, \quad \sigma_i'(f) = f + (\sigma_i(x_i) - x_i) \partial_{\underline{\sigma},i}(f).
\]
We begin by demonstrating that \( \sigma_i' \) is a ring morphism. For \( f, g \in A \), we have
\[
\sigma_i'(fg) = fg + (\sigma_i(x_i) - x_i) \partial_{\underline{\sigma},i}(fg),
\]
and so
\[
\sigma_i'(fg) = fg + (\sigma_i(x_i) - x_i) \left( f \partial_{\underline{\sigma},i}(g) + \sigma_i'(g) \partial_{\underline{\sigma},i}(f) \right),
\]
which simplifies to
\[
\sigma_i'(fg) = f(g + (\sigma_i(x_i) - x_i) \partial_{\underline{\sigma},i}(g)) + (\sigma_i(x_i) - x_i) \sigma_i'(g) \partial_{\underline{\sigma},i}(f).
\]
Thus, we can write:
\[
\sigma_i'(fg) = (f + (\sigma_i(x_i) - x_i) \partial_{\underline{\sigma},i}(f)) \sigma_i'(g) = \sigma_i'(f) \sigma_i'(g).
\]
This confirms that \( \sigma_i' \) is a ring morphism. One can easily check by transversality that for $i,j=1, \ldots, d, ~ \sigma_i'(x_j)=\sigma_i(x_j),$ and by $R$-linearity we have
\[
\sigma_{i|_{R[\underline{X}]}} = \sigma_{i|_{R[\underline{X}]}}'.
\]
On the other hand, by definition of $\sigma_i'$,
\[
\sigma_i' \equiv \text{Id} \, \text{mod} \left( (\sigma_i - \text{Id})(A_0) \right).
\]
Here, we use the inclusion \( \partial_{\underline{\sigma},i}(A_0) \subset (A_0) \), which follows from the assumption that \(\underline{x} \) are good $\underline{\sigma}$-coordinates (see the proof of Proposition \ref{bonsigmacord}).
Thus, by applying Proposition \ref{restriction}, we conclude that \( \sigma_i = \sigma_i' \). \qedhere
\end{proof}

The following definition is derived from classical calculus.

\begin{prop}\label{nul}
If the (good) \( \underline{\sigma} \)-coordinates \( \underline{x} \) are classical, and for each \( i ~=~ 1, \ldots, d \), \( \sigma_i(x_i) \in R[x_i] \), then for every \( i \neq j \), we have \( \partial_{\underline{\sigma},j} (\sigma_i (x_i)) = 0 \).

\end{prop}

\begin{proof}
Let \( \sigma_i(x_i) = \sum_{k=1}^m a_k x_i^k \), where \( a_k \in R \). We perform a direct computation using the \( R \)-linearity of the derivations and the twisted Leibniz rule:
\[
\partial_{\underline{\sigma},j}(\sigma_i(x_i)) = \partial_{\underline{\sigma},j}\left(\sum_{k=1}^m a_k x_i^k\right) = \sum_{k=1}^m a_k \partial_{\underline{\sigma},j}(x_i^k).
\]
We proceed by induction to show that for every \( k \) and every \( i \neq j \), \( \partial_{\underline{\sigma},j}(x_i^k) = 0 \). For \( k = 0 \), this follows from the classical argument:
\[
\partial_{\underline{\sigma},i}(1) = \partial_{\underline{\sigma},i}(1 \cdot 1) = 2 \partial_{\underline{\sigma},i}(1),
\]
which implies \( \partial_{\underline{\sigma},i}(1) = 0 \). For \( k = 1 \), this follows directly from the definition of \( \partial_{\underline{\sigma},i} \).
Now, assume the property holds for some integer \( k \). Since the coordinates are assumed to be classical, we have:
\[
\partial_{\underline{\sigma},j}(x_i^{k+1}) = \partial_{\underline{\sigma},j}(x_i^k \cdot x_i) = x_i^k \partial_{\underline{\sigma},j}(x_i) + \sigma_j(x_i) \partial_{\underline{\sigma},j}(x_i^k) = 0. \qedhere
\]
\end{proof}

\begin{defn}
A twisted $R$-adic algebra $A$ of order $d$ endowed with (good) $\underline{\sigma}$-coordinates $\underline{x}$ is said to be \emph{twisted of order $d$ of type Schwarz} if the following equalities are satisfied:
\[
\forall i,j=1, \ldots, d,~ \partial_{\underline{\sigma},i} \circ \partial_{\underline{\sigma},j}=\partial_{\underline{\sigma},j} \circ \partial_{\underline{\sigma},i} \text { and } \forall i \neq j, ~ \sigma_i  \circ \partial_{\underline{\sigma},j}= \partial_{\underline{\sigma},j} \circ \sigma_i. 
\]
\end{defn}

\begin{prop}\label{schwarz} If the (good) $\underline{\sigma}$-coordinates $\underline{x}$ are classical and for each $i~=~1, \ldots, d,$ $\sigma_i(x_i) \in R[x_i]$
then $A$ is a twisted $R$-algebra of order $d$ of type Schwarz.
\end{prop}

\begin{proof}
The endomorphisms $\sigma_i$ commute by hypothesis.
Now, we aim to show that the derivatives also commute. The following diagram is commutative
\[
\xymatrix{
A \ar[dd]   \ar[rr]^-{\text{d}} & & P_{A} \ar[rr]^-\delta \ar[dd]^-{\tilde{\Theta}_{A, (2)_{\underline{\sigma}}}} \ar@<-2ex>[rr]^-{p_1} \ar@<2ex>[rr]^-{p_2}& & P_{A} \otimes_A' P_{A} \ar[dd]^-{\tilde{\Theta}_{A, (1)_{\underline{\sigma}}} \otimes \tilde{\Theta}_{A, (1)_{\underline{\sigma}}}} \\
\\
A \ar[rr]_-{\Theta_{A,(2)_{\underline{\sigma}}}-\text{Id}} &  & P_{A, (2)_{\underline{\sigma}}}\ar@<-2.5ex>[rr]^-{p_1'} \ar@<2.5ex>[rr]^-{p_2'} \ar[rr]^-{\delta_{1,1}} & & P_{A, (1)_{\underline{\sigma}}} \otimes_A' P_{A, (1)_{\underline{\sigma}}}
}
\]
In this diagram, we define
\begin{align*}
&\text{d}:A \rightarrow P_A, ~a \mapsto 1 \otimes a - a \otimes 1,\\
&p_1: P_{A} \rightarrow P_{A} \otimes_A' P_{A}, ~ a \otimes b \mapsto a\otimes b\otimes 1 \otimes 1\\
\text{and } &p_2: P_{A} \rightarrow P_{A} \otimes_A' P_{A}, ~ a \otimes b \mapsto 1 \otimes 1 \otimes a \otimes b.
\end{align*}
By analogy $p_2'$ and $p_1'$ are the morphisms associated with the morphism $P_{A, (2)_{\underline{\sigma}}} ~\rightarrow~ P_{A, (1)_{\underline{\sigma}}}$. In the upper part of the diagram, $\otimes_A'$ denotes that the action on the right is given by $p_1$ and the action on the left is given by $p_2$. Similarly, in the lower part of the diagram, $\otimes_A'$ indicates that the action on the left is given by $\Theta_{A,(2)_{\underline{\sigma}}}$.
We now define
\[
d_1: =p_2 - \delta + p_1 ~\text{  and  }~ d_2:=p_2' - \delta_{1,1} + p_1'.
\]
It is straightforward to verify that $d_1 \circ \text{d} =0$, which implies $d_2 \circ \left(\Theta_{A,(2)_{\underline{\sigma}}} - \text{Id}\right)=0$.
Therefore, for every $f \in A$, we have
\[
\delta_{1,1} \left(\left(\Theta_{A,(2)_{\underline{\sigma}}} - \text{Id}\right)(f)\right)=p_1'\left(\left(\Theta_{A,(2)_{\underline{\sigma}}} - \text{Id}\right)(f)\right)+p_2'\left(\left(\Theta_{A,(2)_{\underline{\sigma}}} - \text{Id}\right)(f)\right).
\]
Using the morphism $P_{A, (2)_{\underline{\sigma}}} \rightarrow P_{A, (1)_{\underline{\sigma}}}$ we deduce that $\left(\Theta_{A,(2)_{\underline{\sigma}}} - \text{Id}\right)(f)$ can be expressed as
\[
\sum\limits_i \partial_{\underline{\sigma},i}(f)\xi_i + \sum\limits_{i<j} g_{i,j}\xi_i \xi_j + \sum\limits_i g_i \xi_i^{(2)_{\underline{\sigma}}}
\]
where $g_i$ and $g_{i,j}$ are elements of $A$. Thus, we get
\[
\delta_{1,1}\left(\left(\Theta_{A,(2)_{\underline{\sigma}}} - \text{Id}\right)(f)\right)=\delta_{1,1}\left(\sum\limits_i \partial_{\underline{\sigma},i}(f)\xi_i + \sum\limits_{i<j} g_{i,j}\xi_i \xi_j + \sum\limits_i g_i \xi_i^{(2)_{\underline{\sigma}}}\right). 
\]

Since $\xi_i^{(2)_{\underline{\sigma}}}=\xi_i(\xi_i+x_i-\sigma_i(x_i))=0$ in $P_{A, (1)_{\underline{\sigma}}}$ we compute

\begin{align*}
\delta_{1,1}\left(\xi_i^{(2)_{\underline{\sigma}}}\right)&=\delta_{1,1}(\xi_i(\xi_i+x_i-\sigma_i(x_i)))=\delta_{1,1}(\xi_i^2+(x_i-\sigma_i(x_i))\xi_i) \\
&=(1\otimes' \xi_i+\xi_i \otimes' 1)^2 + (x_i -\sigma_i (x_i))\xi_i \otimes' 1 + (x_i -\sigma_i (x_i)) \otimes' \xi_i \\
&=\xi_i^2 \otimes' 1 + 2(\xi_i \otimes' \xi_i) + 1 \otimes' \xi_i^2 + (x_i-\sigma_i(x_i))\xi_i \otimes' 1    \\
&\hspace{0.4cm} + (x_i-\sigma_i(x_i))\otimes'  \xi_i \\
&= -(x_i-\sigma_i(x_i))\xi_i \otimes' 1 + (2\xi_i) \otimes' \xi_i +1 \otimes' \xi_i^2+ (x_i-\sigma_i(x_i))\xi_i \otimes' 1   \\
&\hspace{0.4cm}+ (x_i-\sigma_i(x_i))\otimes'  \xi_i\\
&=(2 \xi_i) \otimes' \xi_i  -1 \otimes' (x_i-\sigma_i(x_i))\xi_i + (x_i-\sigma_i(x_i)) \otimes' \xi_i \\
&=(2 \xi_i) \otimes' \xi_i -\text{d}(x_i-\sigma_i(x_i))\otimes' \xi_i \\
&=(2 \xi_i) \otimes' \xi_i - \xi_i \otimes' \xi_i +\text{d}(\sigma_i(x_i))\otimes' \xi_i \\
&= (\xi_i+\text{d}(\sigma_i(x_i))\otimes' \xi_i \\
&= \left( \xi_i +\partial_{\underline{\sigma},i}(\sigma_i(x_i))\xi_i \right) \otimes' \xi_i \\
&=  (1+\partial_{\underline{\sigma},i}(\sigma_i(x_i)))\xi_i \otimes' \xi_i.
\end{align*}
Thus, we conclude
\begin{align*}
\delta_{1,1}\left(\Theta_{A,(2)_{\underline{\sigma}}}- \text{Id})(f)\right)&=\delta_{1,1}\left(\sum\limits_i \partial_{\underline{\sigma},i}(f)\xi_i + \sum\limits_{i<j} g_{i,j}\xi_i \xi_j + \sum\limits_i g_i \xi_i^{(2)_{\underline{\sigma}}}\right) \\
&=\sum\limits_i \partial_{\underline{\sigma},i}(f)\left(\xi_i \otimes' 1 + 1 \otimes' \xi_i\right) +\sum_{i<j} g_{i,j} \xi_i \otimes' \xi_j  \\
&\phantom{=}+ \sum_{i>j} g_{j,i} \xi_i \otimes' \xi_j+\sum_{i=1}^d g_i \left(1+\partial_{\underline{\sigma},i}(\sigma_i(x_i))\right) \xi_i \otimes' \xi_i\\
&=\sum\limits_i \partial_{\underline{\sigma},i}(f) \xi_i \otimes' 1 + \sum\limits_j \partial_{\underline{\sigma},j}(f) \otimes' \xi_j \\
&\phantom{=}+ \sum\limits_{i<j} g_{i,j} \xi_i \otimes'\xi_j \\
&\phantom{=}+\sum\limits_{i>j} g_{j,i} \xi_i \otimes' \xi_j \\
&\phantom{=}+  \sum\limits_{i=1}^d g_i(1+\partial_{\underline{\sigma},i}(\sigma_i(x_i))) \xi_i \otimes' \xi_i.
\end{align*}
Here we have used that $\xi_i\xi_j=0$ in $P_{A, (1)_{\underline{\sigma}}}$.
On the other hand, we can directly compute $p_1'\left((\Theta_{A,(2)_{\underline{\sigma}}} - \text{Id})(f)\right)$ and $p_2'\left((\Theta_{A,(2)_{\underline{\sigma}}} - \text{Id})(f)\right)$
as follows:
\[
p_1'\left((\Theta_{A,(2)_{\underline{\sigma}}} - \text{Id})(f)\right)=\sum\limits_i \partial_{\underline{\sigma},i}(f) \xi_i \otimes' 1
\]
and
\[
\begin{aligned}
p_2'\left((\Theta_{A,(2)_{\underline{\sigma}}} - \text{Id})(f)\right)&=\sum\limits_j 1 \otimes' \partial_{\underline{\sigma},j}(f) \xi_j \\
&= \sum\limits_j \Theta_{A,(1)_{\underline{\sigma}}}\left(\partial_{\underline{\sigma},j}(f)\right) \otimes' \xi_j\\
&=\sum\limits_j\left( \partial_{\underline{\sigma},j}(f)+\sum\limits_i \partial_{\underline{\sigma},i}(\partial_{\underline{\sigma},j}(f))\xi_i\right)  \otimes' \xi_j. 
\end{aligned}
\]
By further simplifying
\begin{align*}
p_2'\left((\Theta_{A,(2)_{\underline{\sigma}}} - \text{Id}\right)(f))=
&\sum\limits_j \partial_{\underline{\sigma},j}(f) \otimes' \xi_j + \sum\limits_{i<j} \partial_{\underline{\sigma},i}(\partial_{\underline{\sigma},j}(f)) \xi_i \otimes' \xi_j\\
&+\sum\limits_{i>j} \partial_{\underline{\sigma},i}(\partial_{\underline{\sigma},j}(f)) \xi_i \otimes' \xi_j +  \sum\limits_i \partial^2_{\underline{\sigma},i}(f) \xi_i \otimes' \xi_i.
\end{align*}
By comparing the coefficients, we obtain the relations
\begin{equation}\label{eq1}
\forall i<j, ~ g_{i,j}=\partial_{\underline{\sigma},i}(\partial_{\underline{\sigma},j}(f))=\partial_{\underline{\sigma},j}(\partial_{\underline{\sigma},i}(f))
\end{equation}
and
\[
\forall i, ~ (1+\partial_{\underline{\sigma},i}(\sigma_i(x_i)))g_i=\partial^2_{ \underline{\sigma},i}(f).
\]
This enables us to determine that the derivatives commute by equation (\ref{eq1}).
To demonstrate that the derivatives commute with the endomorphisms $\sigma_i$, we use the formula from Definition \ref{classique}. On one side, we have:
\[
\sigma_i(\partial_{\underline{\sigma},j}(f))=\partial_{\underline{\sigma},j}(f)+(\sigma_i(x_i)-x_i)\partial_{\underline{\sigma},i}(\partial_{\underline{\sigma},j}(f))=\partial_{\underline{\sigma},j}(f)+(\sigma_i(x_i)-x_i)\partial_{\underline{\sigma},j}(\partial_{\underline{\sigma},i}(f)),
\]
and on the other side:
\[
\begin{aligned}
\partial_{\underline{\sigma},j}(\sigma_i(f))&=\partial_{\underline{\sigma},j}(f+(\sigma_i(x_i)-x_i)\partial_{\underline{\sigma},i}(f))\\
&=\partial_{\underline{\sigma},j}(f)+\partial_{\underline{\sigma},j}((\sigma_i(x_i)-x_i)\partial_{\underline{\sigma},i}(f))\\
&=\partial_{\underline{\sigma},j}(f)+(\sigma_i(x_i)-x_i)\partial_{\underline{\sigma},j}(\partial_{\underline{\sigma},i}(f))+\sigma_j(\partial_{\underline{\sigma},i}(f))\partial_{\underline{\sigma},j}(\sigma_i(x_i)-x_i)\\
&=\partial_{\underline{\sigma},j}(f)+(\sigma_i(x_i)-x_i)\partial_{\underline{\sigma},j}(\partial_{\underline{\sigma},i}(f)). 
\end{aligned}
\]
Here, we more employed Proposition \ref{nul}. Thus, we confirm the commutativity of the derivatives with the endomorphisms $\sigma_i$, completing the proof. \qedhere
\end{proof}

\begin{defn}\label{sym}
The (good) $\underline{\sigma}$-coordinates $\underline{x}$ are said to be \emph{symmetric} if
\begin{align*}
&\forall i =  1, \ldots,d; ~\sigma_i(x_i) \in R[x_i] \text{ and } \\
&\forall n,m \in \mathbb{N}, \forall f \in A, ~\delta_{n,m}\left(\Theta_{A,(n+m)_{\underline{\sigma}}}(f)\right)=1 \otimes'\Theta_{A,(m)_{\underline{\sigma}}}(f).
\end{align*}
Where $\otimes'$ means that the action on the left is given by $\Theta_{A,(n)_{\underline{\sigma}}}$.
\end{defn}

We recall the following general result on surjective morphisms of modules.

\begin{lem}\label{lemnoyau}
Let $f: A \rightarrow B$ and $g: C \rightarrow D$ be surjective morphism of $R$-modules. Then, following equality holds:
\[
\emph{Ker}(f \otimes  g) = \emph{Im}\left(A \otimes \emph{Ker}(g) +  \emph{Ker}(f) \otimes C \rightarrow A \otimes_R C\right). 
\]
\end{lem}

\begin{prop}\label{symcom}
Assume that the (good) $\underline{\sigma}$-coordinates $\underline{x}$ are symmetric. Set
\[
\begin{aligned}
\delta \colon P_{A} &\rightarrow P_{A} \otimes_A' P_{A}, \\
a\otimes b &\mapsto (a \otimes 1) \otimes' (1 \otimes b).
\end{aligned}
\]
Then,
\[
\delta\left(I_{A}^{(n+m+1)_{\underline{\sigma}}}\right) \subset I_{A}^{(n+1)_{\underline{\sigma}}} \otimes_A' P_{A} + P_{A} \otimes_A' I_{A}^{(m+1)_{\underline{\sigma}}}.
\]
\end{prop}

\begin{proof}
By assumption on the coordinates, the following diagram is commutative:
\[
\xymatrix{
 &P_{A} \ar[r]^-\delta \ar[d]_-{\tilde{\Theta}_{A,(n+m)_{\underline{\sigma}}}}  & P_{A} \otimes_A' P_{A} \ar[d]^-{\tilde{\Theta}_{A,(n)_{\underline{\sigma}}} \otimes ~\tilde{\Theta}_{A,(m)_{\underline{\sigma}}}} \\
  & P_{A, (n+m)_{\underline{\sigma}}} \ar[r]_-{\delta_{n,m}} & P_{A, (n)_{\underline{\sigma}}} \otimes_A' P_{A, (m)_{\underline{\sigma}}}.
}
\]
This implies that
\[
\delta \left(I_{A}^{(n+m+1)_{\underline{\sigma}}}\right) \subset \text{ker}\left(\tilde{\Theta}_{A,(n)_{\underline{\sigma}}} \otimes \tilde{\Theta}_{A,(m)_{\underline{\sigma}}}\right).
\]
We now conclude by applying lemma \ref{lemnoyau}
\end{proof}

\begin{defn} \label{Twisted connection}
When the (good) $\underline{\sigma}$-coordinates $\underline{x}$ are classical, a \emph{twisted connection} on an $A$-module $M$ is defined as a map
\[
\nabla_{\underline{\sigma}}:M \rightarrow M \otimes_A \Omega^1_{\underline{\sigma}}
\]
which satisfies the condition that the induced map
\[
\Theta_{M}: M \rightarrow M \otimes_A P_{A, (1)_{\underline{\sigma}}},~ s \rightarrow s \otimes 1 + \nabla_{\underline{\sigma}}(s)
\]
is $\Theta_{A,(1)_{\underline{\sigma}}}$-linear. That is, for all $f \in A$ and $s \in M$, the following relation holds:
\[
 \Theta_M(fs)=\Theta_{A,(1)_{\underline{\sigma}}}(f)\Theta_M(s).
\]
\end{defn}

\begin{rem}
For $n \in \mathbb{N}^*$, set $\Omega^n_{\underline{\sigma}}= \bigwedge^n \Omega^1_{\underline{\sigma}}.$  Since $\Omega^1_{\underline{\sigma}}$ is freely generated by $\text{d}x_i=1\otimes x_i - x_i \otimes 1$ for $i=1,\ldots,d$, (where $\text{d}: A \rightarrow \Omega^1_{\underline{\sigma}}$ is as in Remark \ref{defofd}) as $A$-module, $\Omega^n_{\underline{\sigma}}$ is freely generated by $\text{d}x_I =  \text{d}x_i \wedge \ldots \wedge \text{d}x_{i_n}$ with $i_1<\ldots<i_n$ as $A$-module. We define the additive map
\[
\text{d}_n : \Omega^n_{\underline{\sigma}} \rightarrow \Omega^{n+1}_{\underline{\sigma}}
\]
for $n \in \mathbb{N}$ by the formula $\text{d}(a\text{d}x_I)= \text{d}(a) \wedge \text{d}x_I.$
A twisted connection $\nabla_{\underline{\sigma}}$ on an $A$-module $M$ induces a sequence of additive maps 
 \[
\nabla_{n,\underline{\sigma}}: M \otimes_A \Omega^n_{\underline{\sigma}} \rightarrow M \otimes_A \Omega^{n+1}_{\underline{\sigma}}
 \]
defined by the formula
\[
\nabla_{n,\underline{\sigma}}(m \otimes w)=m \otimes \text{d}_n(w)+\nabla_{\underline{\sigma}}(m)\wedge w.
\]
\end{rem}

\begin{defn}
A twisted connection $\nabla_{\underline{\sigma}}$ on an $A$-module $M$ is said to be \emph{integrable} if $\nabla_{1,\underline{\sigma}} \circ \nabla_{\underline{\sigma}}=0.$
\end{defn}

\begin{rem}\label{7.23}
Suppose that the $\underline{\sigma}$-coordinates $\underline{x}$ are of Schwarz type. Then, for $I \subset \lbrace 1,\ldots,d \rbrace$ with $|I|=n$,
\begin{align*}
    \text{d}_{n+1}\circ \text{d}_n(a\text{d}x_{I})&=\text{d}_{n+1}\left(\sum_{i=1}^d \partial_{\underline{\sigma},i}(a)\text{d}x_i \wedge \text{d}x_I \right)\\
    &= \sum_{i,j=1}^d \partial_{\underline{\sigma},j} \circ \partial_{\underline{\sigma},i}(a)\text{d}x_j\wedge \text{d}x_i \wedge \text{d}x_I\\
    &=\sum_{1 \leq i \leq j \leq d} (\partial_{\underline{\sigma},j}\circ\partial_{\underline{\sigma},i}(a)-\partial_{\underline{\sigma},i}\circ \partial_{\underline{\sigma},j}(a))\text{d}x_j \wedge \text{d}x_i \wedge \text{d}x_I =0,
\end{align*}
namely, $(A, \text{d})$ is integrable. So we can construct the de Rham complex:
\[
\text{DR}(A)=\left[   A \rightarrow  \Omega^1_{\underline{\sigma}} \rightarrow \Omega^2_{\underline{\sigma}} \rightarrow \dots \right].
\]
More generally, for an $A$-module equipped with an integrable twisted connection $\nabla_{\underline{\sigma}}=\sum_{i=1}^d\partial_{M, \underline{\sigma},i} \text{d}x_i,$
\begin{align*}
  \nabla_{n+1,\underline{\sigma}} \circ \nabla_{n,\underline{\sigma}}(m \otimes w) &= \nabla_{n+1,\underline{\sigma}}(m \otimes \text{d}_n(w)+\nabla_{\underline{\sigma}}(m) \wedge w) \\
  & = \nabla_{n+1,\underline{\sigma}}\left(m \otimes \text{d}_n(w) + \sum_{i=1}^d\partial_{M, \underline{\sigma},i}(m) \otimes (\text{d}x_i \wedge w) \right)\\
  &= m \otimes (\text{d}_{n+1}\circ \text{d}_n(w))+\nabla_{\underline{\sigma}}(m) \wedge \text{d}_n(w) \\
  & ~~~~~~-  \sum_{i=1}^d\partial_{M, \underline{\sigma},i}(m) \otimes (\text{d}x_i \wedge \text{d}_n(w)) \\
  &~~~~~~+ \sum_{i,j=1}^d\partial_{M, \underline{\sigma},j} \circ \partial_{M, \underline{\sigma},i} (m)\otimes (\text{d}x_j \wedge \text{d}x_i \wedge w)\\
  &=m \otimes (\text{d}_{n+1} \circ \text{d}_n(w))+(\nabla_{1,\underline{\sigma}} \circ \nabla_{\underline{\sigma}}(m))\wedge w = 0,
\end{align*}
where the last equality follows from the integrability of $(A,\text{d})$ and $(M,\nabla_{\underline{\sigma}}).$ So one can construct the de Rham complex
\[
\text{DR}(M)=\left[   M \rightarrow  M \otimes_A\Omega^1_{\underline{\sigma}} \rightarrow M \otimes_A\Omega^2_{\underline{\sigma}} \rightarrow \dots \right].
\]
\end{rem}

We denote by \( \nabla_{\underline{\sigma}} \)-\(\text{Mod}(A)\) the category of \( A \)-modules equipped with a twisted connection, and by \( \nabla_{\underline{\sigma}}^{\text{Int}} \)-\(\text{Mod}(A)\) the category of \( A \)-modules equipped with an integrable twisted connection.

\begin{thm}\label{equivcate}
When the (good) $\underline{\sigma}$-coordinates $\underline{x}$ are classical, there exists an equivalence of categories
\[
\nabla_{\underline{\sigma}}  \text{-}\mathrm{Mod}(A)
 \simeq \mathrm{T}_{A,\underline{\sigma}}\text{-}\mathrm{Mod}(A).
\]
\end{thm}

\begin{proof}
We begin by constructing a functor 
\[
\nabla_{\underline{\sigma}} \text{-} \text{Mod}(A) \rightarrow \text{T}_{A,\underline{\sigma}}\text{-}\text{Mod}(A).
\]
Let $M$ be an $A$-module equipped with a twisted connection $\nabla_{\underline{\sigma}}$.
For an element $s$ of $M$, we write
\[
\nabla_{\underline{\sigma}}(s)=\sum_{i=1}^d \nabla_i(s) \otimes \xi_i.
\]
The condition on $\Theta_M$ can be translated to: For all $f \in A$ and $s \in M$,
\[
fs\otimes 1 + \sum_{i=1}^d \nabla_i(fs) \otimes \xi_i=\left(f\otimes 1 + \sum_{i=1}^d \partial_{\underline{\sigma},i}(f)(1 \otimes \xi_i) \right) \left( s \otimes 1 + \sum_{i=1}^d \nabla_i(s) \otimes \xi_i\right).
\]
Expanding the right-hand side, we get
\begin{align*}
fs \otimes 1 + \sum_{i=1}^d \left(f \nabla_i(s)+\partial_{\underline{\sigma},i}(f)s\right)\otimes \xi_i + \sum_{i,j=1}^d  \partial_{\underline{\sigma},i}(f) \nabla_j(s) \otimes \xi_i \xi_j.
\end{align*}
We recall the following relations in $\Omega^1_{A,\underline{\sigma}}$
\[ 
\xi_i\xi_j=0 \text{ for } i \neq j \text{ and } \xi_i^2=(\sigma_i(x_i)-x_i)\xi_i.
\]
Thus, we have
\[
\sum_{i,j=1}^d  \partial_{\underline{\sigma},i}(f) \nabla_j(s) \otimes \xi_i \xi_j=\sum_{i=1}^d (\sigma_i(x_i)-x_i) \partial_{\underline{\sigma},i}(f) \nabla_i(s) \otimes \xi_i.
\]
By identification, for $i=1,\ldots,d,$ we obtain
\begin{align*}
\nabla_i(fs)&=f\nabla_i(s)+\partial_{\underline{\sigma},i}(f)s+(\sigma_i(x_i)-x_i)\partial_{\underline{\sigma},i}(f)\nabla_i(s)\\
&=\partial_{\underline{\sigma},i}(f)s+(f+(\sigma_i(x_i)-x_i)\partial_{\underline{\sigma},i}(f))\nabla_i(s)\\
&=\partial_{\underline{\sigma},i}(f)s+\sigma_i(f)\nabla_i(s).
\end{align*}
Hence, given a twisted connection $\nabla_{\underline{\sigma}}$, the endomorphisms $\nabla_{1}, \ldots, \nabla_{d}$ on $M$ define an object of $\text{T}_{A,\underline{\sigma}}\text{-Mod}(A)$.
Next, we construct an inverse functor
\[
 \text{T}_{A,\underline{\sigma}}\text{-Mod}(A) \rightarrow \nabla_{\underline{\sigma}} \text{-Mod}(A).
\]
Assume that $M$, endowed with actions by $\partial_{\underline{\sigma},1},\ldots,\partial_{\underline{\sigma},d}$, is an object of $\text{T}_{A,\underline{\sigma}}\text{-Mod}(A)$. Set 
\[\nabla_{\underline{\sigma}}(s)=\sum\limits_{i=1}^d \partial_{\underline{\sigma},i}(s) \otimes \xi_i.\]
We need to verify that the induced map $\Theta_M$ is $\Theta_{A,(1)_{\underline{\sigma}}}$-linear. We compute
\begin{align*}
\Theta_{A,(1)_{\underline{\sigma}}}(f)\Theta_M(s)&=\left(f+\sum_{i=1}^d \partial_{\underline{\sigma},i}(f)(1\otimes \xi_i) \right)\left( s\otimes1 + \sum_{i=1}^d \partial_{\underline{\sigma},i}(s) \otimes \xi_i \right)\\
&=fs \otimes + \sum_{i=1}^d (f \partial_{\underline{\sigma},i}(s) + s \partial_{\underline{\sigma},i}(f)) \otimes \xi_i + \sum_{i,j=1}^d  \partial_{\underline{\sigma},i}(f)\partial_{\underline{\sigma},i}(s) \otimes \xi_i \xi_j \\
&= fs \otimes  1 + \sum_{i=1}^d ((f+(\sigma_i(x_i)-x_i)\partial_{\underline{\sigma},i}(f))\partial_{\underline{\sigma},i}(s)+s\partial_{\underline{\sigma},i}(f))\otimes \xi_i \\
&=fs \otimes 1 + \sum_{i=1}^d (\sigma_i(f)\partial_{\underline{\sigma},i}(s)+s\partial_{\underline{\sigma},i}(f))\otimes \xi_i \\
&=fs \otimes 1 + \sum_{i=1}^d \partial_{\underline{\sigma},i}(fs)\otimes \xi_i=\Theta_M(fs).
\end{align*}
It is easy to see that the functors we constructed are quasi-inverses of each other. So we have established our claim.
\end{proof}

\section{Twisted differential operators of infinite level} \label{twisted differential operators of infinite level}
We maintain the hypotheses from the previous section: \( R \) is a Huber ring, and \( (A, \underline{\sigma}) \) is a twisted \( R \)-adic algebra of order \( d \) (Definition \ref{defor}). Additionally, there exist elements \( \underline{x} = (x_1, \ldots, x_d) \) in \( A \) that are (good) \( \underline{\sigma} \)-coordinates (Definition \ref{coor}, \ref{goodcoor}). In this section, we further assume that the coordinates \( \underline{x} \) are symmetric (Definition \ref{sym}). This assumption will be crucial in defining the composition of twisted differential operators.

\begin{defn}\label{twisted differential operator}
Let $M$ and $N$ be $A$-modules. An $R$-linear morphism $\phi:M \rightarrow N$ is called  \emph{twisted differential operator of order at most $n$} if its $A$-linearization $\tilde{\phi}:P_A \otimes_A' M = A \otimes_R M \rightarrow M$ factorizes through $P_{A,(n)_{\underline{\sigma}}}\otimes_A' M$. This condition is equivalent to requiring that $\tilde{\phi}$ is zero on $I_{A}^{(n+1)_{\underline{\sigma}}}$. We denote by $\text{Diff}_{n, \underline{\sigma}}(M,N)$ the set of twisted differential operators of order at most $n$.
\end{defn}

By definition $\text{Diff}_{n, \underline{\sigma}}(M,N) \simeq \text{Hom}_A(P_{A,(n)_{\underline{\sigma}}}\otimes_A' M,N)$. In the case where $M=N$ we will simply write $\text{Diff}_{n, \underline{\sigma}}(M)$. We also define
\[
\text{Diff}_{\underline{\sigma}}(M,N)=\varinjlim_n \text{Hom}_A\left(P_{A,(n)_{\underline{\sigma}}}\otimes_A' M,N\right) \text{ and } \text{D}^{(\infty)}_{A,\underline{\sigma}}=\text{Diff}_{\underline{\sigma}}(A).
\]

\begin{defn}
We define the  basis $\lbrace \partial^{[\underline{k}]}_{\underline{\sigma}}, |\underline{k}|\leq n \rbrace$ of $\text{Diff}_{n,\underline{\sigma}}(A)$ as the dual basis of the basis $\lbrace \underline{\xi}^{(\underline{k})_{\underline{\sigma}}}, |\underline{k}|\leq n \rbrace$ of $P_{A,(n)_{\underline{\sigma}}}$.
\end{defn}

\begin{prop}
Let $u:(A, \underline{\sigma})\rightarrow (B, \underline{\tau})$ be a flat morphism of twisted $R$-adic algebras of order $d$. If $B$ is an $A$-module of finite presentation, then there exists an isomorphism
\[
\mathrm{Diff}_{\underline{\tau}}(B \otimes_A M, B \otimes_A N) \simeq B \otimes_A \mathrm{Diff}_{\underline{\sigma}}(M,N).
\]
\end{prop}

\begin{proof}
By the commutativity of the tensor product with inductive limits, we have
\[
 B \otimes_A \text{Diff}_{\underline{\sigma}}(M,N) = \varinjlim B \otimes_A \text{Hom}_A(P_{A,(n)_{\underline{\sigma}}}\otimes_A' M,N).
\]
We then apply \cite[Chapitre 1, Section 2, Sous-section 10, Proposition 11]{Bourbaki}
to obtain
\[
B \otimes_A \text{Diff}_{\underline{\sigma}}(M,N) \simeq \varinjlim \text{Hom}_B (B \otimes_A P_{A,(n)_{\underline{\sigma}}} \otimes_A' M, B \otimes_A N). 
\]
Finally, by applying Proposition \ref{u}, we conclude the proof.
\end{proof}

\begin{prop}
Let $M,~ N$ and $L$ be $A$-modules. For natural integers $m$ and $n$, given $\varphi \in \emph{Diff}_{n,\underline{\sigma}}(M,N)$ and $\psi \in \emph{Diff}_{m,\underline{\sigma}}(L,M)$, the composition $\varphi \circ \psi$ belongs to $\emph{Diff}_{m+n,\underline{\sigma}}(L,N)$. In particular, $\emph{Diff}_{\underline{\sigma}}(M)$ is a subring of $\emph{End}_{R}(M)$.
\end{prop}

\begin{proof}
Let $\varphi: M \rightarrow N$ be a twisted differential operator of order at most $n$ and $\psi: L \rightarrow M$ be a twisted differential operator of order at most $m$.
We have a diagram
\[
\xymatrix{
P_{A} \otimes_A' L \ar[r]^-{\delta \otimes \text{Id}} & P_{A} \otimes_A' P_{A} \otimes_A' L \ar[r]^-{\text{Id} \otimes \tilde{\psi}} & P_{A} \otimes_A' M \ar[r]^-{\tilde{\varphi}} & N.
}
\]
Since $\tilde{\varphi}$ factorizes through $P_{A,(n)_{\underline{\sigma}}} \otimes_A' M$, and $\text{Id} \otimes_A' \tilde{\psi}$ factorizes through $ P_{A}\otimes_A' P_{A,(m)_{\underline{\sigma}}} \otimes_A' L$, the composition factorizes through $P_{A,(n)_{\underline{\sigma}}} \otimes_A' P_{A,(m)_{\underline{\sigma}}} \otimes_A' L$.  By Proposition \ref{symcom}, it follows that the composition factorizes through $P_{A,(n+m)_{\underline{\sigma}}} \otimes_A' L$, thus proving the claim. 
\end{proof}

\begin{prop}\label{zeub}
We have the following equalities
\[
\mathrm{Diff}_{0,\underline{\sigma}}(A)=A \text{ and } \mathrm{Diff}_{1,\underline{\sigma}}(A)=A \oplus \mathrm{T}_{A, \underline{\sigma}}.
\]
\end{prop}

\begin{proof}
The first equality follows directly from the construction of differential operators. For the second equality, we apply Proposition \ref{suite exacte} and use the fact that $\text{T}_{A, \underline{\sigma}}$ is the dual of $\Omega^1_{A, \underline{\sigma}}$.
\end{proof}

Recall that for an element $q$ of $R$ and every $n \in \N$, we define the \emph{$q$-analogue} of $n$ as
\[
(n)_q:=1+q+\ldots+q^{n-1}.
\]
Similarly, we define the $q$-analogue of the factorial of $n$ as
\[
(n)_q!:=(0)_q \ldots (n)_q,
\]
with the convention $(0)_q!=1$. Also, we define the $q$-analogue of the binomial coefficient by induction as
\[
\binom{n}{k}_q=\binom{n-1}{k-1}_q+q^k\binom{n-1}{k}_q
\]
with the convention
\[
\binom{0}{k}_q = 
\begin{cases}
    1, & \text{ if } k=0, \\
    0, & \text{otherwise.}
\end{cases}
\]
If $(n)_q$'s for $n\geq 1$ are invertible we have
\[
 \binom{n}{k}_q = \frac{(n)_q!}{(k)_q!(n-k)_q!}
\]
(see \cite{SQ1} for details).

\begin{exam}
Let $\underline{q}=(q_1,\ldots,q_d )$ be elements of $R$. In the case where $\underline{x}$ are $\underline{q}$-coordinates, by the same proof as Corollary 6.2 of \cite{LSQ3}, we have
\[
\forall i = 1, \ldots, d, ~ \forall k \in \mathbb{N}, ~ \forall z \in A,  \partial^k_{q_i}(z)=(k)_{q_i}! \partial_{q_i}^{[k_i]}(z).
\]
where $k_i:=(0,\ldots,0,k,0,\ldots,0)$ with $k$ placed in the $i$-th entry.
Thus, if we set
\[
\partial_{\underline{\sigma}}^{\underline{k}}:=\partial_{\underline{\sigma},1}^{k_1} \circ \ldots  \circ  \partial_{\underline{\sigma},d}^{k_d}
\]
for $\underline{k}=(k_1,\ldots,k_d) \in \N^d$, we obtain
\begin{equation}\label{qderiv}
\partial_{\underline{\sigma}}^{\underline{k}}=(k_1)_{q_1}!\ldots(k_d)_{q_d}!\partial_{\underline{\sigma}}^{[\underline{k}]}.
\end{equation}
\end{exam}

\begin{thm}\label{equivate 3}
If there exist some elements $\underline{q}=(q_1,\ldots,q_d ) $ of $R$ such that $\underline{x}=(x_1,\ldots,x_d)$ are classical (good) $\underline{q}$-coordinates and
\[
\forall i=1, \ldots, d, ~ \forall n \geq 1, ~ (n)_{q_i} \in R^\times,
\]
then there exists an equivalence of categories
\[
\nabla_{\underline{q}}^{\emph{Int}}\text{-}\emph{Mod}(A) \simeq \emph{D}_{A,\underline{q}}^{(\infty)}\text{-}\emph{Mod}.
\]
\end{thm}

\begin{proof}
If we are given an object $M$ in $\nabla_{\underline{q}}^{\text{Int}}\text{-}\text{Mod}(A)$, by Theorem \ref{equivcate}, there exists an action of elements of $\text{T}_{A,\underline{q}}$ on $M$, and by (\ref{qderiv}), we can endow an action of elements of $\text{D}_{A,\underline{q}}^{(\infty)}$ on $M$. The twisted Leibniz rule and integrability for $M$ and Proposition \ref{schwarz} imply that this action defines a $\text{D}_{A,\underline{q}}^{(\infty)}$-module structure on $M$

Conversely, if $M$ is a $\text{D}_{A,\underline{q}}^{(\infty)}$-module, we have an action of elements of $\text{T}_{A,\underline{q}}$ on $M$, and the ring structure of $\text{D}_{A,\underline{q}}^{(\infty)}$ implies the twisted Leibniz rule for this action and integrability. So $M$ is an object in $\nabla_{\underline{q}}^{\text{Int}}\text{-}\text{Mod}(A)$. 

It is clear the the two construction are inverses each other, and so we have the equivalence.
\end{proof}

We denote by $\overline{\text{D}}_{A, \underline{\sigma}}$ the smallest subring of $\text{End}_R(A)$ containing $A$ and $\text{T}_{A, \underline{\sigma}}$. Additionally, we define $\text{D}_{A, \underline{\sigma}}$ as the non-commutative polynomial ring in $\partial_{\underline{\sigma},1},\ldots,\partial_{\underline{\sigma},d}$ over $A$, subject to the following relations
\begin{equation}\label{cond1}
\forall i \neq j,~  \partial_{\underline{\sigma},i} \circ \partial_{\underline{\sigma},j}=\partial_{\underline{\sigma},j} \circ \partial_{\underline{\sigma},i},
\end{equation}
\begin{equation}\label{cond2}
\text{ and   }~ \forall a \in A, ~ \partial_{\underline{\sigma},i} a=\partial_{\underline{\sigma},i}(a)+\sigma_i(a)\partial_{\underline{\sigma},i}.
\end{equation}

\begin{prop}
If $\underline{x}$ are classical coordinates, there exists a factorization
\[
\emph{D}_{A, \underline{\sigma}}\twoheadrightarrow \overline{\emph{D}}_{A, \underline{\sigma}} \hookrightarrow \emph{D}^{(\infty)}_{A,\underline{\sigma}}.
\]
\end{prop}

\begin{proof}
The injection comes from Proposition \ref{zeub}. For the surjection, we can define, for $i=1, \ldots, d,$ a map that sends $\partial_{\underline{\sigma},i}$ to the corresponding endomorphism of $A$. Since the coordinates are symmetric and classical, the conditions (\ref{cond1}) and (\ref{cond2}) are satisfied. This ensures that the maps is well-defined and surjective by the definition of $\text{T}_{A, \underline{\sigma}}$.
\end{proof}

Finally, we describe the relation between $\text{D}^{(\infty)}_{A,\underline{\sigma}}$-modules and $A$-modules endowed with twisted Taylor structure, which is established in \cite{LSQ3}.

\begin{lem}
    If $M$ is a $\emph{D}^{(\infty)}_{A,\underline{\sigma}}$-module, the canonical map $\emph{D}^{(\infty)}_{A,\underline{\sigma}} \rightarrow \emph{End}_R(M)$ induces, for all $n \in \N$, a $P_A$-linear map $\emph{Diff}_{n, \underline{\sigma}}(A) \rightarrow \emph{Diff}_{n, \underline{\sigma}}(M).$
\end{lem}

\begin{proof}
    The proof is the same as in \cite[Lemma 5.1]{LSQ3}.
\end{proof}

\begin{defn}\label{twistedtaylorstructure}
A \emph{twisted Taylor structure} on an $A$-module $M$ is a compatible family of maps
\[
\theta_n : M \longrightarrow M \otimes_A P_{A,(n)_{\underline{\sigma}}}
\]
with $\theta_n(am) = \Theta_{A,(n)_{\underline{\sigma}}}(a)\theta_n(m)$ for $a \in A, m \in M$ and $\theta_0 = \mathrm{Id}$, such that the following diagram is commutative for all $n,m \in \mathbb{N}$:
\begin{equation}\label{5}
    \xymatrix{
M \ar[rr]^-{\theta_m} \ar[d]_{\theta_{m+n}} 
 & & M \otimes_A P_{A,(m)_{\underline{\sigma}}} \ar[d]^{\theta_n \otimes \mathrm{Id}} \\
M \otimes_A P_{A,(m+n)_{\underline{\sigma}}} \ar[rr]_-{\mathrm{Id} \otimes \delta_{n,m}} 
 & & M \otimes_A P_{A,(n)_{\underline{\sigma}}} \otimes_A' P_{A,(m)_{\underline{\sigma}}}.
}
\end{equation}

We also define 
\[
\widehat{\theta} := \varprojlim_n \theta_n : 
M \longrightarrow \varprojlim_n \bigl(M \otimes_A P_{A,(n)_{\underline{\sigma}}}\bigr).
\]
\end{defn}

\begin{prop}
If $M$ is a $\emph{D}^{(\infty)}_{A,\underline{\sigma}}$-module, there exists a canonical way to associate a twisted Taylor structure on $M$. Concretely, it is given by
\begin{equation}\label{6}
\theta_n(m) = \sum_{|\underline{k}| \leq n} \partial^{[\underline{k}]}_{\underline{\sigma}} (m) \otimes 
\underline{\xi}^{(\underline{k})}, 
\qquad
\widehat{\theta}(m) = 
\sum_{|\underline{k}| \in \N^d} \partial^{[\underline{k}]}_{\underline{\sigma}} (m) \otimes  \underline{\xi}^{(\underline{k})}.
\end{equation}
This correspondence gives an equivalence between the category of $\emph{D}^{(\infty)}_{A,\underline{\sigma}}$-modules and that of $A$-modules endowed with twisted Taylor structure.
\end{prop}

\begin{proof}
The proof is the same as that in \cite[Proposition 5.3]{LSQ3}, but we recall it here. 
If we are given a $\text{D}^{(\infty)}_{A,\underline{\sigma}}$-module $M$, we have the associated 
$P_A$-linear map
\begin{equation}\label{7}
\mathrm{Diff}_{n,\underline{\sigma}}(A) \longrightarrow \mathrm{Diff}_{n,\underline{\sigma}}(M),
\end{equation}
namely, a $P_A$-linear map
\begin{equation}\label{8}
    \mathrm{Hom}_A(P_{A,(n)_{\underline{\sigma}}},A) \longrightarrow 
\mathrm{Hom}_A(P_{A,(n)_{\underline{\sigma}}} \otimes_A M, M).
\end{equation}
This induces a $P_A$-linear map
\begin{equation}\label{9}
P_{A,(n)_{\underline{\sigma}}} \otimes_A M \longrightarrow 
\mathrm{Hom}_A\bigl(\mathrm{Hom}_A(P_{A,(n)_{\underline{\sigma}}},A),M\bigr) 
\xleftarrow{\ \simeq\ } 
M \otimes_A P_{A,(n)_{\underline{\sigma}}}.
\end{equation}
(The right isomorphism follows from the fact that $P_{A,(n)_{\underline{\sigma}}}$ is a free $A$-module of finite type). This induces the map
\[
\theta_n : M \longrightarrow M \otimes_A P_{A,(n)_{\underline{\sigma}}}
\]
satisfying $\theta_n(am) = \Theta_{A,(n)_{\underline{\sigma}}}(a)\theta_n(m)$ for $a \in A, m \in M$.
It is clear from the construction that $\theta_0 = \mathrm{Id}$ and that 
the $\theta_n$'s are compatible. Also, using the fact that the union of the maps
(\ref{7}) is a ring homomorphism, we can check the commutativity of the diagram \ref{5}.
Thus we have defined a twisted Taylor structure on $M$. Since the above argument 
can be reversed, we obtain the equivalence of categories in the statement.

We check the concrete description of the map $\theta_n$. The map (\ref{8}) is given by
\[
\partial^{[\underline{k}]}_{\underline{\sigma}}
\mapsto \bigl( 1 \otimes m \mapsto 
\partial^{[\underline{k}]}_{\underline{\sigma}}(m)\bigr),
\]
and so the maps in the diagram (\ref{9}) are given by
\[
1 \otimes m \mapsto \biggl(\partial^{[\underline{k}]}_{\underline{\sigma}}
\mapsto \partial^{[\underline{k}]}_{\underline{\sigma}}(m)\biggr)
\longleftrightarrow 
\sum_{|\underline{k}| \leq n} 
\partial^{[\underline{k}]}_{\underline{\sigma}} (m)
\otimes \underline{\xi}^{(\underline{k})_{\underline{\sigma}}}.
\]
Hence the maps $\theta_n, \widehat{\theta}$ are expressed as in \ref{6}.
\end{proof}

\begin{defn}
For a $\text{D}^{(\infty)}_{A,\underline{\sigma}}$-module $M$, we call the maps $\theta_n, \widehat{\theta}$ 
the \emph{twisted Taylor maps} associated to $M$.
\end{defn}

\section{Twisted calculus of finite radius}
From now on, we make the following assumptions: \( R \) is a Tate ring (see Definition \ref{def:tate_ring}), and \( (A, \underline{\sigma}) \) is a complete Noetherian twisted \( R \)-adic algebra of order \( d \) (see Definition \ref{defor}). This implies that \( A \) is also a Tate ring by \cite[Proposition II.2.5.6]{Morel} and is thus equipped with an ultrametric norm \( \lVert \cdot \rVert \). We assume that there exist elements \( \underline{x} = (x_1, \ldots, x_d) \) in \( A \) that are classical and symmetric  (good) \( \underline{\sigma} \)-coordinates (Definitions \ref{coor}, \ref{goodcoor}, \ref{classique}, and \ref{sym}). We will assume that this norm is contractive (Definition \ref{contractante}) with respect to the endomorphisms \( \sigma_1, \ldots, \sigma_d \).

\begin{rems}
\begin{enumerate}
\item Let $A$ be a Huber ring with $(A_0, I_0)$ a couple of definition.
If $M$ is an $A$-module of finite type, we can equip it with the weakest topology 
such that any $A$-linear homomorphism $M \to N$ with $N$ a topological $A$-module 
is continuous.

If we take any surjection $\varphi : A^n \to M$ and denote by $M_\varphi$ 
the topological $A$-module $M$ endowed with the quotient topology induced by $\varphi$, 
any $A$-linear homomorphism $M_\varphi \to N$ with $N$ a topological $A$-module 
is continuous because
\[
A^n \xrightarrow{\varphi} M_\varphi \longrightarrow N
\]
is continuous. Thus, the topology on $M$ in the previous paragraph is weaker than 
that on $M_\varphi$. Then, since $\mathrm{Id}: M \to M_\varphi$ should be 
continuous, we conclude that the topology on $M$ in the previous paragraph 
coincides with that on $M_\varphi$. In particular, the topology on $M_\varphi$ 
is independent of the choice of $\varphi$. Also, we can endow 
$M = M_\varphi$ with a seminorm induced by that on $A^n$.

Also, we see that any $A_0$-module $M_0 \subset M$ of finite type which generates $M$ 
as $A$-module is open in $M$, and the topology on $M_0$ is the $I_0$-adic topology. 
Indeed, we can take a surjection of the form 
$\varphi_0 : A_0^n \to M_0$ which induces a surjection 
$\varphi : A^n \to M$, and since $A_0^n \subset\varphi^{-1}(M_0)$, 
$M_0$ is open in $M$. We have that $I_0^m M_0$ is open for any $m \in \mathbb{N}$ by 
similar reasoning, and for any open submodule $L_0 \subset M_0$, $\varphi^{-1}(L_0)$ contains 
$(I_0^m)^n$ for some $m$ and so $I_0^mM_0 \subset L_0$.

\item If \( A \) is a ring equipped with a submultiplicative norm \( \lVert \cdot \rVert \), then every finite-type \( A \)-module \( M \) can be endowed with a semi-norm. Given a set of generators \( (e_1, \ldots, e_m) \) for \( M \), the semi-norm on \( M \) is defined as follows:
\[ \forall x \in M, \quad \lVert x \rVert_M = \inf_{x = \sum_{i=1}^m a_i e_i} \left\lbrace \sum_{i=1}^m \lVert a_i \rVert \right\rbrace. \] 
For more details, we refer the reader to Section 2 of \cite{kedliu}. 

\item In what follows, it will be necessary to assume that $A$-modules of ﬁnite
type are complete in order for certain limits to exist and to be unique. This
assumption holds under the hypotheses we are working with: Indeed, let $A$ be
as in the beginning of this section with $(A_0 , I_0 )$ a couple of deﬁnition, let $M$
be an $A$-module of ﬁnite type and let $M_0 \subset M$ be an $A_0$-module of ﬁnite type
generating $M$ as $A$-module. Then, with repect to the $I_0$-adic topology, $M_0$ is
complete by \cite[0 Corollary 8.2.20]{FujiwaraKato2018}. Thus, $M$ is complete. In particular, the
seminorm on $M$ explained in 1. is a norm in our situation.
Another approach would be to require that there exists a Noetherian ring of
deﬁnition $A_0$ for $A$ such that $A$ is of ﬁnite presentation over $A_0$, or alternatively, that the module $M$ is ﬂat of ﬁnite presentation over $A$.
\end{enumerate}
\end{rems}

\subsection{Twisted differential principal parts of finite radius}

\begin{defn}
The \emph{$\underline{x}$-radius} of $\underline{\sigma}=(\sigma_1,\ldots,\sigma_d)$ is defined as
\[
\rho(\underline{\sigma})=\max_{1 \leq i \leq d} \lVert x_i - \sigma_i(x_i) \rVert.
\]
\end{defn}

\begin{lem}\label{3.4}
If we set $\underline{\sigma}^n=\left\lbrace \sigma_1^n,\ldots,\sigma_d^n \right\rbrace$, then
\[
\rho(\underline{\sigma}^n) \leq \rho(\underline{\sigma}).
\]
\end{lem}

\begin{proof}
We proceed by induction on $n$:
\[
\begin{aligned}
\sup_{1 \leq i \leq d} \lVert x_i - \sigma_i^{n+1}(x_i)\rVert & \leq \max \left\lbrace \sup_{1 \leq i \leq d} \lVert x_i - \sigma_i^{n}(x_i)\rVert, \sup_{1 \leq i \leq d} \lVert \sigma_i^n(x_i-\sigma_i(x_i))\rVert \right\rbrace \\
& \leq \max \left\lbrace \rho( \underline{\sigma}^n), \rho(\underline{\sigma }) \right\rbrace = \rho(\underline{\sigma}).
\end{aligned}
\]
The last inequality holds because the norm is contractive with respect to the endomorphisms $\sigma_1,\ldots,\sigma_d$.
\end{proof}

Let $\eta \in \mathbb{R}$ such that $0<\eta<1$. Set
\[
A \left\lbrace \underline{\xi}/\eta \right\rbrace = \left\lbrace \sum\limits_{\underline{n} \in \mathbb{N}^d} z_{\underline{n}} \underline{\xi}^{\underline{n}}, ~z_{\underline{n}} \in A \text{ and } \lVert z_{\underline{n}} \rVert \eta^{|\underline{n}|} \rightarrow 0 \text{ when } |\underline{n}| \rightarrow \infty \right\rbrace.
\]
This forms a Banach algebra with respect to the sup norm
\[
\left\lVert \sum\limits_{\underline{n} \in \mathbb{N}^d} z_{\underline{n}} \underline{\xi}^{\underline{n}} \right\rVert_\eta = \max \lVert z_{\underline{n}} \rVert \eta^{|\underline{n}|}.
\]
We refer the reader to section 2.2 of \cite{kedliu} for more details. This ring is also a Tate ring.

\begin{lem}\label{bla}
If $\eta \geq \rho(\underline{\sigma})$, then the $A$-linear map
\[
\underline{\xi}^{\underline{n}} \mapsto \underline{\xi}^{(\underline{n})_{\underline{\sigma}}}
\]
is an isometric automorphism of the $A$-module $A \left\lbrace \underline{\xi}/\eta \right\rbrace$.
\end{lem}

\begin{proof}
Recall that for $\underline{k}\in\mathbb{N}^d$,
\[
\underline{\xi}^{(\underline{k})_{\underline{\sigma}}} = \prod_{i=1}^d \prod_{j=0}^{k_i-1} (\xi_i+x_i-\sigma_i^j(x_i))=\xi_1^{k_1}\ldots\xi_d^{k_d}+f_{\underline{k}}
\]
with $f_{\underline{k}} \in A[\underline{\xi}]_{< |\underline{k}|}$. By Lemma \ref{3.4}, for every $j \in \mathbb{N}$ and every $i=1, \ldots, d,$ we have $\lVert x_i - \sigma_i^j(x_i) \rVert \leq \eta$. It follows that $\lVert f_{\underline{k}} \rVert_\eta \leq \eta^{\left(\sum k_i\right)} = \eta^{|\underline{k}|}$, and thus $\left\lVert \underline{\xi}^{(\underline{k})_{\underline{\sigma}}} \right\rVert_\eta= \eta^{|\underline{k}|}$.
Hence, the unique $A$-linear endomorphism of $A[\underline{\xi}]$ that sends $\underline{\xi}^{\underline{k}}$ to $\underline{\xi}^{(\underline{k})_{\underline{\sigma}}}$ is an isometry that preserves the degree. Consequently, it extends uniquely to an isometry of $A \left\lbrace \underline{\xi}/\eta \right\rbrace$ to itself.
\end{proof}

\begin{defn}
An \emph{orthogonal Schauder basis} of a normed \( A \)-module is a family \( \{ s_n \}_{n \in \mathbb{N}} \) of elements of \( M \) such that every element \( s \) in \( M \) can be uniquely expressed as a convergent sum
\[
s = \sum_{n=0}^{\infty} z_n s_n,
\]
where \( z_n \in A \) and \( \lVert s \rVert = \sup \left\lbrace \lVert z_n \rVert \lVert s_n \rVert \right\rbrace \).
\end{defn}

\begin{prop}\label{Schauder}
If \( \eta \geq \rho(\underline{\sigma}) \), the family \( \{ \underline{\xi}^{(\underline{n})_{\underline{\sigma}}} \}_{\underline{n} \in \mathbb{N}^d} \) forms an orthogonal Schauder basis for the \( A \)-module \( A \left\lbrace \underline{\xi}/\eta \right\rbrace \).
\end{prop}

\begin{proof}
This follows directly from Lemma \ref{bla} and the fact that the family \( \{ \underline{\xi}^{\underline{k}} \}_{\underline{k} \in \mathbb{N}^d} \) is an orthogonal Schauder basis of \( A \left\lbrace \underline{\xi}/\eta \right\rbrace \).
\end{proof}

\begin{prop}\label{pr}
If \( \eta \geq \rho(\underline{\sigma}) \), then for all \( n \in \mathbb{N} \), there exists an isomorphism of \( A \)-algebras
\[
A \left\lbrace \underline{\xi}/ \eta \right\rbrace / \left( \underline{\xi}^{(\underline{k})_{\underline{\sigma}}} \; \text{with} \; \underline{k} \in \mathbb{N}^d \; \text{such that} \; |\underline{k}| = n+1 \right) \simeq P_{A,(n)_{\underline{\sigma}}}.
\]
\end{prop}

\begin{proof}
By Proposition \ref{surjection}, we have
\[
P_{A,(n)_{\underline{\sigma}}} = A[\underline{\xi}] / \left( \underline{\xi}^{(\underline{k})_{\underline{\sigma}}} \; \text{with} \; \underline{k} \in \mathbb{N}^d \; \text{such that} \; |\underline{k}| = n \right) \simeq A[\underline{\xi}]_{\leq n}.
\]
Let \( I_n := \left( \underline{\xi}^{(\underline{k})_{\underline{\sigma}}} \; \text{with} \; \underline{k} \in \mathbb{N}^d \; \text{such that} \; |\underline{k}| = n+1 \right) \) denote the ideal generated by these elements in $A[\underline{\xi}]$. 
We can conclude by considering the map
\[
A[\underline{\xi}]_{\leq n} \simeq A[\underline{\xi}] / I_n \rightarrow A \left\lbrace \underline{\xi}/ \eta \right\rbrace / I_n A \left\lbrace \underline{\xi}/ \eta \right\rbrace,
\]
which is bijective by Proposition \ref{Schauder}.
\end{proof}

\begin{prop}\label{tenseurschauder}
Let \( M \) be an \( A \)-module of finite type and suppose it is normed by a surjection $\varphi: A^n \rightarrow M$. Let $\tilde{\varphi}: A \left\lbrace \underline{\xi}/\eta \right\rbrace^n \rightarrow M \otimes_A A\left\lbrace \underline{\xi}/\eta \right\rbrace$ be the surjection induced by $\varphi$ and endow $M \otimes_A A\left\lbrace \underline{\xi}/\eta \right\rbrace$ with the norm induced by $\tilde{\varphi}$. Then, any element $m \in M \otimes_A A\left\lbrace \underline{\xi}/\eta \right\rbrace$ can be expressed uniquely as a convergent sum
\begin{equation}\label{eq:10}
m = \sum_{\underline{i} \in \N^d} m_{\underline{i}} \otimes \underline{\xi}^{(\underline{i})_{\underline{\sigma}}}.
\end{equation}
In particular, there exists a canonical injection
\[
M \otimes_A A\{\underline{\xi}/\eta\}
\longrightarrow
\varprojlim_n \bigl(M \otimes_A P_{A,(n)_{\underline{\sigma}}}\bigr);
\sum_{\underline{i} \in \mathbb{N}^d} m_{\underline{i}} \otimes \underline{\xi}^{(i)_{\underline{\sigma}}}
\longmapsto
\left(
\sum_{|i| \le n} m_{\underline{i}} \otimes \underline{\xi}^{(i)_{\underline{\sigma}}}
\right)_n .
\]

Moreover we have the equality of norms
\[
\lVert m\rVert = \sup_{\underline{i} \in \N^d} \lVert m_{\underline{i}} \rVert \eta^{|\underline{i}|}.
\]
\end{prop}

\begin{proof}
First we prove the uniqueness of the expression~(\ref{eq:10}).
Suppose that $m$ in~(\ref{eq:10}) is equal to $0$.
By Proposition~\ref{pr}, for any $n \in \mathbb{N}$, there exists a map
\[
A\{\underline{\xi}/\eta \rbrace \longrightarrow
A\{\underline{\xi}/\eta \rbrace / (\underline{\xi}^{(k)_{\underline{\sigma}}},\, |k| = n+1)
\simeq P_{A,(n)_\sigma}.
\]
This induces the map
\[
M \otimes_A A\{\underline{\xi}/\eta \rbrace \longrightarrow M \otimes_A P_{A,(n)_{\underline{\sigma}}},
\]
which sends
\[
0 = m = \sum_{\underline{i} \in \mathbb{N}^d} m_{\underline{i}} \otimes \underline{\xi}^{(\underline{i})_{\underline{\sigma}}}
\quad \text{to} \quad
\sum_{|\underline{i}| \le n} m_{\underline{i}} \otimes \underline{\xi}^{(\underline{i})_{\underline{\sigma}}}.
\]
Hence $\sum_{|\underline{i}| \le n} m_{\underline{i}} \otimes \underline{\xi}^{(\underline{i})_{\underline{\sigma}}} = 0$ in
$M \otimes_A P_{A,(n)_{\underline{\sigma}}}
\simeq \bigoplus_{|\underline{i}| \le n} M \cdot \underline{\xi}^{(\underline{i})_{\underline{\sigma}}}$,
and thus $m_{\underline{i}} = 0$ for all $\underline{i}$ with $|\underline{i}| \le n$.
Since this is true for any $n \in \mathbb{N}$, we conclude that
$m_i = 0$ for all $i$, and hence the uniqueness holds.

To prove the existence, let $m \in M \otimes_A A\left\lbrace \underline{\xi}/\eta \right\rbrace$ and write it as
\begin{equation}\label{eq:11}    
m = \tilde{\varphi}\left(\left( \sum_{\underline{i} \in \N^d} a_{j,\underline{i}} \underline{\xi}^{(\underline{i})_{\underline{\sigma}}} \right)_{j=1}^n\right).
\end{equation}
Then, if we put
\begin{equation}\label{eq:12}
m_{\underline{i}} = \varphi((a_{j,\underline{i}})_{j=1}^n),
\end{equation}
we have the equality (\ref{eq:10}). By definition, the norm of $m = \sum_{\underline{i} \in \N^d} m_{\underline{i}} \otimes \underline{\xi}^{(\underline{i})_{\underline{\sigma}}}$ and $m_{\underline{i}}$ are given by
\[
\lVert m \rVert = \inf_{(*)}(\sup_{\underline{i},j} \lVert a_{j, \underline{i}} \rVert \eta^{|\underline{i}|}), ~ \lVert m_{\underline{i}} \rVert = \inf_{(*)_{\underline{i}}}(\sup_{j} \lVert a_{j, \underline{i}} \rVert ),
\]
where $\inf_{(*)}$ is the infimum among $(a_{j,\underline{i}})_{j,\underline{i}}$'s satisfying (\ref{eq:11}), and for each $\underline{i} \in \N^d$, $\inf_{(*)_{\underline{i}}}$ is the infimum among $(a_{j,\underline{i}})_{j}$'s satisfying (\ref{eq:12}). Because the condition (\ref{eq:11}) is equivalent to the conditions (\ref{eq:12}) for all $\underline{i}$, we see that
\begin{align*}
    \lVert m_{\underline{i}} \rVert \eta^{|\underline{i}|} & = \inf_{(*)_{\underline{i}}}(\sup_{j} \lVert a_{j, \underline{i}} \rVert  \eta^{|\underline{i}|} ) \\
    & \leq \inf_{(*)}(\sup_{\underline{i},j} \lVert a_{j, \underline{i}} \rVert \eta^{|\underline{i}|}) = \lVert m \rVert,
\end{align*}
hence $\sup_{\underline{i} \in \N^d} \lVert m_{\underline{i}} \rVert \eta^{|\underline{i}|} \leq \lVert m \rVert $. On the other hand, for any $\varepsilon >0$, we can take for each $\underline{i}$ the elements $(a_{j,\underline{i}})_j$ satisfying (\ref{eq:12}) and
\[
\sup_j \lVert a_{j, \underline{i}} \rVert  \eta^{|\underline{i}|} \leq \lVert m_{\underline{i}} \rVert \eta^{|\underline{i}|} + \varepsilon.
\]
Then the elements $(a_{j,\underline{i}})_{j,\underline{i}}$ satisfy (\ref{eq:11}) and so we have
\[
\lVert m \rVert \leq \sup_{\underline{i},j} \lVert a_{j, \underline{i}} \rVert \eta^{|\underline{i}|} = \sup_{\underline{i}}(\sup_{j} \lVert a_{j, \underline{i}} \rVert \eta^{|\underline{i}|}) \leq \sup_{\underline{i}} \lVert m_{\underline{i}} \rVert \eta^{|\underline{i}|} + \varepsilon.
\]
Hence we conclude that $\lVert m\rVert = \sup_{\underline{i} \in \N^d} \lVert m_{\underline{i}} \rVert \eta^{|\underline{i}|}.$
\end{proof}

\subsection{Radius of convergence}

\begin{defn}\label{radius of convergence}
Let $M$ be a $\dinf$-module of finite type over $A$ and $\eta \in \mathbb{R}$.
\begin{enumerate}
\item Let $s \in M$. Then $s$ is \emph{$\eta$-convergent} if
\[
\left\lVert \partial_{\underline{\sigma}}^{[\underline{k}]} (s) \right\rVert \eta^{|\underline{k}|} \rightarrow 0 \text{ when } |\underline{k}| \rightarrow \infty.
\]
\item $M$ is \emph{$\eta$-convergent} if all elements of $M$ are $\eta$-convergent.

\item $M$ is \emph{strongly $\eta$-convergent} if $M$ is $\eta$-convergent and
\[
\sup_{s\in M, \neq 0} \left( \frac{\sup_{\underline{k} \in \N^d} \lVert\partial_{\underline{\sigma}}^{[\underline{k}]}(s) \rVert \eta^{|\underline{k}|}}{\lVert s \rVert } \right) < \infty.
\]

\item $M$ is \emph{very strongly $\eta$-convergent} if $M$ is $\eta$-convergent and
\[
\sup_{s\in M, \neq 0} \left( \frac{\sup_{\underline{k} \in \N^d} \lVert\partial_{\underline{\sigma}}^{[\underline{k}]}(s) \rVert \eta^{|\underline{k}|}}{\lVert s \rVert } \right) =1.
\]
\item $M$ is \emph{$\eta^\dagger$-convergent} (resp.  \emph{strongly $\eta^\dagger$-convergent}) if it is $\eta'$-convergent (resp. strongly $\eta'$-convergent) for any $\eta'<\eta$.
\end{enumerate}

\end{defn}

\begin{exam}
Assume that \( A = K \lbrace\underline{X}\rbrace \), where \( K \) is a non-archimedean field, and $\underline{X}$ are (good) $\underline{q}$-coordinates for non-zero elements $\underline{q}=(q_1,\ldots,q_d)$ of $K$ with $\lVert q_i \rVert \leq 1$ for all $i$.
Under this setting, and by using formula (\ref{qderiv}), we obtain that for any \(\underline{k}\) and \(\underline{n}\) in \(\mathbb{N}^d\) with \(k_i \leq n_i\) for all \(i = 1, \ldots, d\) the following formula holds:
\[
\partial_{\underline{\sigma}}^{[\underline{k}]}(\underline{X}^{\underline{n}}) = \prod_{i=1}^d \binom{n_i}{k_i}_{q_i} X_i^{n_i - k_i}.
\]
Therefore, using \cite[Proposition 1.4]{GLQ4}, we have
\[
\left\lVert \partial_{\underline{\sigma}}^{[\underline{k}]}(\underline{X}^{\underline{n}}) \right\rVert \leq 1.
\]
As a result, if \( z = \sum_{\underline{n} \in \mathbb{N}^d} a_{\underline{n}} \underline{X}^{\underline{n}} \in K\lbrace \underline{X} \rbrace \), then we obtain
\[
\left\lVert \partial_{\underline{\sigma}}^{[\underline{k}]} (z) \right\rVert \leq \max_{\underline{n}} \left\lVert a_{\underline{n}} \partial_{\underline{\sigma}}^{[\underline{k}]}(\underline{X}^{\underline{n}}) \right\rVert \leq \max_{\underline{n} \geq \underline{k}} \left\lVert a_{\underline{n}} \right\rVert \le \rVert z\lVert.
\]
This implies that the algebra \( A \) is very strongly \(\eta\)-convergent for any \(\eta < 1\).
\end{exam}

\begin{lem}\label{découle}
Let $M$ be a $\emph{D}^{(\infty)}_{A,\underline{\sigma}}$-module of finite type over $A$ and $\eta \geq \rho (\underline{\sigma})$. An element $s$ of $M$ is $\eta$-convergent if and only if its twisted Taylor series
\[
\widehat{\theta}(s)=\sum_{\underline{k} \in \N^d} \der (s) \otimes \underline{\xi}^{(\underline{k})_{\underline{\sigma}}}
\]
is an element of $M \otimes_A A \left\lbrace \underline{\xi}/ \eta \right\rbrace \subset\varprojlim_n ( M \otimes_A  P_{A,(n)_{\underline{\sigma}}})$.
\end{lem}

\begin{proof}
This is an immediate consequence of the definition of $\eta$-convergence and using Proposition \ref{tenseurschauder}. 
\end{proof}

\begin{prop}\label{3.3}
Let \( M \) be a \( \emph{D}^{(\infty)}_{A,\underline{\sigma}} \)-module of finite type over \( A \), and let \( \eta \geq \rho(\underline{\sigma}) \). The module \( M \) is \( \eta \)-convergent if and only if its Twisted Taylor series factorizes in the following way:
\[
\xymatrix{
M \ar[r]^-{\widehat{\theta}} \ar[rd]_-{\theta_\eta} & \varprojlim (M \otimes_A P_{A,(n)_{\underline{\sigma}}})  \\
 & M \otimes_A \aff. \ar[u]
}
\]
Moreover, if $A$ is strongly $\eta$-convergent, the map $\theta_\eta$ is continuous.
\end{prop}

\begin{proof}
The former statement is an immediate consequence of Lemma \ref{découle}. Moreover, by definition, $\theta_\eta$ is continuous if and only if $M$ is strongly $\eta$-convergent. Thus, if $A$ is strongly $\eta$-convergent, the map $\theta_\eta$ for $A$
\[
\theta_{\eta,A}: A \rightarrow \aff
\]
is continuous. So we can take $C>0$ such that $\lVert \theta_{\eta,A }(a)\rVert \leq C \lVert a \rVert$ for any $a \in A$. If we take a generator $m_1,\ldots,m_n$ of $M$ and we consider the norm on $M$ and that on $M \otimes_A \aff$ using this generator, we have for $m=\sum_{i=1}^n a_i m_i \in M$ the inequalities
\[
\lVert \theta_\eta(m)\rVert = \left\lVert \theta_\eta\left(\sum_{i=1}^ma_i m_i \right) \right\rVert \leq \max_i \lVert \theta_{\eta,A}(a_i)\theta_\eta (m_i) \rVert \leq (C \max_i \lVert \theta_\eta (m_i) \rVert)\max_i \lVert a_i \rVert,
\]
hence the inequality $\lVert \theta_\eta (m) \rVert \leq \tilde{C} \lVert m\rVert $ for $\tilde{C} =C\max_i\lVert \theta_\eta (m_i) \rVert$. So $\theta_\eta$ is continuous, as required. 
\end{proof}

\begin{defn}
The map $\theta_\eta$ is called the \emph{twisted Taylor map of radius $\eta$}.
\end{defn}

\begin{prop}
The property of \( \eta \)-convergence for a \( \text{D}^{(\infty)}_{A,\underline{\sigma}} \)-module of finite type over \( A \) is stable under quotients and submodules, provided these submodules and quotients are themselves of finite type over \( A \).
\end{prop}

\begin{proof}
In our setting, for a module $M$ of finite type over $A$ or $\aff$ and a submodule $N \subset M$, the topology on $N$ induced from $M$ coincides with the original topology on $N$: This follows from \cite[Theorem 8.2.19]{FujiwaraKato2018}. Noting this fact, the proof is identical to the one of \cite[Proposition 3.5]{GLQ4}.
\end{proof}

\subsection{Twisted differential operators of finite radius}

\begin{defn}
Assume that $A$ is strongly $\eta$-convergent. The \( A \)-module structure on \( A \left\lbrace \underline{\xi}/\eta \right\rbrace \), induced by the twisted Taylor map of radius \( \eta \), is called the \emph{right structure}. In this case, we denote by \( \aff \otimes' - \) the tensor product where the right structure is applied on the left-hand side. The action is given by
\[
f \otimes' zs = \theta_\eta(z)f \otimes' s.
\]
\end{defn}

\begin{rem}
If \( A \) is \( \eta \)-convergent, it is possible to linearize the twisted Taylor map of radius \( \eta \) to obtain an \( A \)-linear map
\[
\tilde{\theta}_\eta: P_{A} \rightarrow \aff
\]
which induces the commutative diagram below
\begin{equation}\label{diagram}
\xymatrix{
& & A \ar[dl]_-\theta \ar[dr]^-{\widehat{\theta}} \ar[d]^-{\theta_\eta} \\
A[\underline{\xi}] \ar[r] & P_{A} \ar[r]_-{\tilde{\theta}_\eta} & \aff \ar[r] & \varprojlim_n  P_{A,(n)_{\underline{\sigma}}}
} 
\end{equation}
where $\theta$ is defined as
\[
\theta: A \rightarrow P_{A}, ~ f \mapsto 1 \otimes f.
\]
\end{rem}

\begin{lem}\label{9.19}
Assume that $A$ is strongly $\eta$-convergent. Let \( M \) and \( N \) be two \( A \)-modules of finite type. The map
\[
M \rightarrow \aff \otimes'_A M, \quad s \mapsto 1 \otimes' s
\]
induces an injective map:
\[
\emph{Hom}_{A\emph{-cont}}(\aff \otimes'_A M, N) \rightarrow \emph{Hom}_{R\emph{-cont}}(M, N).
\]
\end{lem}

\begin{proof}
Note first that the composition
\[
A[\underline{\xi}] \rightarrow P_A \xrightarrow{\tilde{\theta}_\eta} \aff.
\]
is the canonical inclusion, because $\xi_i$ is sent to $1\otimes x_i - x_i \otimes 1$ in $P_A$ and it is sent to $(x_i+\xi_i)-x_i=\xi_i$ in $\aff$.
Hence the image of \( \tilde{\theta}_\eta \) is dense in \( \aff \). Hence the map
\[
\text{Hom}_{A\text{-cont}}(\aff \otimes'_A M, N) \rightarrow \text{Hom}_{A\text{-cont}}(P_A \otimes'_A M, N) \simeq  \text{Hom}_{R\text{-cont}}(M, N)
\]
induced by the map
\[
M \xrightarrow{i} A \otimes_R M \simeq P_A \otimes_A M
\xrightarrow{\ \tilde{\theta}_{\eta} \otimes \text{Id}\ }
\aff \otimes_A M
\]
(where $i$ is the map $s \mapsto 1 \otimes s$) is injective.
\end{proof}

\begin{defn}\label{operators of radius eta}
Let \( M \) and \( N \) be two \( A \)-modules of finite type. An \( R \)-linear map \( \varphi: M \rightarrow N \) is called a \emph{twisted differential operator of radius \( \eta \)} if it extends to a continuous \( A \)-linear map \( \tilde{\varphi_{\eta}}: \aff \otimes'_A M \rightarrow N \), called the \emph{\( \eta \)-linearization}.
\end{defn}

We will denote by $\text{Diff}^{(\eta)}_{\underline{\sigma}}(M,N)$ the set of twisted differential operators of radius $\eta$.

\begin{rem}
Let $M$ and $N$ be two $A$-modules of finite type. The set $\text{Hom}_{R\text{-cont}}(M,N)$ of $R$-linear continuous morphisms from $M$ to $N$ is a $P_{A}$-module for the action defined by
\[
\forall \psi \in \text{Hom}_{R\text{-cont}}(M,N), ~\forall a,b \in A, \forall x \in M ~(a\otimes b)\cdot \psi(x)=a \psi(bx).
\]
\end{rem}

\begin{prop}\label{dualité}
Assume that $A$ is strongly $\eta$-convergent. Let \( M \) and \( N \) be two \( A \)-modules of finite type. The set \( \emph{Diff}^{(\eta)}_{\underline{\sigma}}(M, N) \) is a \( P_A \)-submodule of \( \emph{Hom}_{R\emph{-cont}}(M, N) \), which contains \( \emph{Diff}_{\underline{\sigma}}^{(\infty)}(M, N) \). Moreover, we have the following isomorphism
\[
\emph{Hom}_{A\emph{-cont}}(\aff \otimes'_A M, N) \simeq \emph{Diff}^{(\eta)}_{\underline{\sigma}}(M, N).
\]
\end{prop}

\begin{proof}
By definition and Lemma \ref{9.19}, we have the isomorphism 
\[
\text{Hom}_{A\text{-cont}}(\aff \otimes'_A M, N) \simeq \text{Diff}^{(\eta)}_{\underline{\sigma}}(M, N).
\]
Also, since the $P_A$-module structure on $\aff$ is defined by $\tilde{\theta}_\eta:P_A \rightarrow\aff$, the morphism
\[
\text{Hom}_{A\text{-cont}}(\aff \otimes'_A M, N) \rightarrow \text{Hom}_{A\text{-cont}}(P_A \otimes'_A M, N) = \text{Hom}_{R\text{-cont}}(M,N)
\]
is $P_A$-linear. So $\text{Diff}^{(\eta)}_{\underline{\sigma}}(M, N)$ is a $P_A$-submodule of $\text{Hom}_{R\text{-cont}}(M,N)$. Furthermore, it contains $\text{Diff}^{(\infty)}_{\underline{\sigma}}(M, N)$, since we have the canonical surjection $\aff \rightarrow P_{A, (n)_{\underline{\sigma}}}$ for every natural number $n$.
\end{proof}

We define the ring $\aff (1)$ by
\[
\aff (1) := \left\lbrace \sum_{\underline{i}, \underline{j} \in \N^d} a_{\underline{i}, \underline{j}} (\underline{\xi}^{\underline{i}} \otimes' \underline{\xi}^{\underline{j}}), \lVert a_{\underline{i}, \underline{j}} \rVert \eta^{|\underline{i}|+|\underline{j}|} \rightarrow 0 \text{ when } |\underline{i}|+|\underline{j}| \rightarrow \infty \right\rbrace.
\]
and define the norm on it by
\[
\left\lVert  \sum_{\underline{i}, \underline{j} \in \N^d} a_{\underline{i}, \underline{j}} (\underline{\xi}^{\underline{i}} \otimes' \underline{\xi}^{\underline{j}}) \right\rVert_\eta := \sup_{\underline{i}, \underline{j} \in \N^d} \lVert a_{\underline{i}, \underline{j}} \rVert \eta^{|\underline{i}|+|\underline{j}|}.
\]
As in the proof of Lemma \ref{bla}, one can prove that the map
\[
\aff (1) \rightarrow \aff (1), ~ \underline{\xi}^{\underline{i}} \otimes' \underline{\xi}^{\underline{j}} \mapsto \underline{\xi}^{(\underline{i})_{\underline{\sigma}}} \otimes' \underline{\xi}^{(\underline{j})_{\underline{\sigma}}}
\]
is an isometry and thus $\lbrace \underline{\xi}^{(\underline{i})_{\underline{\sigma}}} \otimes' \underline{\xi}^{(\underline{j})_{\underline{\sigma}}}, \underline{i}, \underline{j} \in \N^d \rbrace$ forms an orthogonal Schauder basis. Thus, we can write the ring $\aff (1)$ and the norm on it also as
\[
\aff (1) := \left\lbrace \sum_{\underline{i}, \underline{j} \in \N^d} a_{\underline{i}, \underline{j}} (\underline{\xi}^{(\underline{i})_{\underline{\sigma}}} \otimes' \underline{\xi}^{(\underline{j})_{\underline{\sigma}}}), \lVert a_{\underline{i}, \underline{j}} \rVert \eta^{|\underline{i}|+|\underline{j}|} \rightarrow 0 \text{ when } |\underline{i}|+|\underline{j}| \rightarrow \infty \right\rbrace,
\]
\[
\left\lVert  \sum_{\underline{i}, \underline{j} \in \N^d} a_{\underline{i}, \underline{j}} (\underline{\xi}^{(\underline{i})_{\underline{\sigma}}} \otimes' \underline{\xi}^{(\underline{j})_{\underline{\sigma}}}) \right\rVert_\eta := \sup_{\underline{i}, \underline{j} \in \N^d} \lVert a_{\underline{i}, \underline{j}} \rVert \eta^{|\underline{i}|+|\underline{j}|}.
\]
We have the canonical continuous homomorphism
\[
p_1: \aff \rightarrow \aff (1), ~\sum_{\underline{i}  \in \N^d} a_{\underline{i}} \underline{\xi}^{(\underline{i})_{\underline{\sigma}}} \mapsto \sum_{\underline{i} \in \N^d} a_{\underline{i}} \underline{\xi}^{(\underline{i})_{\underline{\sigma}}} \otimes' 1,\]
and if $A$ is strongly $\eta$-convergent, we have another continuous homomorphism
\[
p_2: \aff \rightarrow \aff (1), ~ \sum_{\underline{i}  \in \N^d} a_{\underline{i}} \underline{\xi}^{(\underline{i})_{\underline{\sigma}}} \mapsto \sum_{\underline{i} \in \N^d} \theta_\eta(a_{\underline{i}}) (1\otimes' \underline{\xi}^{(\underline{i})_{\underline{\sigma}}} ).
\]
So, if $A$ is strongly $\eta$-convergent, these induce the diagram
\[
P_A \otimes_A' P_A \rightarrow \aff \otimes_A' \aff \rightarrow\aff(1).
\]
Also, for $m,n \in \N$, we have a surjective homomorphism
\begin{align*}  
\aff (1) &\rightarrow P_{A,(m)_{\underline{\sigma}}} \otimes_A' P_{A,(n)_{\underline{\sigma}}}\\ 
\sum_{\underline{i}, \underline{j} \in \N^d} a_{\underline{i}, \underline{j}} (\underline{\xi}^{(\underline{i})_{\underline{\sigma}}} \otimes' \underline{\xi}^{(\underline{j})_{\underline{\sigma}}})  &\mapsto \sum_{|\underline{i}| \leq n,  |\underline{j}| \leq m} a_{\underline{i}, \underline{j}} (\underline{\xi}^{(\underline{i})_{\underline{\sigma}}} \otimes' \underline{\xi}^{(\underline{j})_{\underline{\sigma}}}) 
\end{align*}
such that the limit
\[
\aff (1) \rightarrow \varprojlim_{m,n} (P_{A,(m)_{\underline{\sigma}}} \otimes_A' P_{A,(n)_{\underline{\sigma}}} ) 
\]
is injective.

\begin{prop}\label{4.5}
Suppose that $A$ is strongly $\eta$-convergent. Then there exists a unique continuous ring homomorphism
\[
\delta_\eta:\aff \rightarrow\aff (1)
\]
such that the following diagram commutes for all $m,n \in \N$:
\begin{equation}\label{rightsquare}
\xymatrix{
P_{A} \ar[r] \ar[d]^-{\delta} & \aff \ar[r] \ar[d]^-{\delta_\eta} & P_{A,(m+n)_{\underline{\sigma}}} \ar[d]^-{\delta_{m,n}}  \\
P_{A} \otimes_A' P_{A} \ar[r] & \aff(1) \ar[r] & P_{A,(m)_{\underline{\sigma}}} \otimes'_A P_{A,(n)_{\underline{\sigma}}} 
}
\end{equation}
Moreover, $\delta_\eta$ is a morphism of Huber $R$-algebras of norm $1$.
\end{prop}

\begin{proof}
If such a ring homomorphism $\delta_\eta$ exists, by taking the limit, we obtain the commutative diagram
\[
\xymatrix{
\aff \ar[d]^{\delta_\eta} \ar[r]  & \varprojlim_{n} P_{A,(n)_{\underline{\sigma}}} \ar[d] \\
\aff (1) \ar[r]&  \varprojlim_{m,n} (P_{A,(m)_{\underline{\sigma}}} \otimes_A P_{A,(n)_{\underline{\sigma}}}).
}
\]
Then the uniqueness of the map follows from the injectivity of horizontal maps (Proposition \ref{tenseurschauder} and the claim before the statement of this proposition).

Next, we prove the existence. We define the ring homomorphism $\delta_\eta: A[\underline{\xi}] \rightarrow \aff (1)$ by $\xi_i \mapsto 1 \otimes'\xi_i+\xi_i \otimes1.$ Since it has norm 1, it extends to the ring homomorphism $\delta_\eta : \aff \rightarrow \aff(1)$ of norm 1. The commutativity of the right square in (\ref{rightsquare}) follows from the definition. Also, since the bottom arrow of the above diagram is injective, to prove the commutativity of the left square of \eqref{rightsquare}, it suffices to prove the commutativity of the big rectangle. This follows from the assumption that the coordinates $\underline{x}$ are symmetric (see also the proof of Proposition \ref{symcom}). So the proof is finished.
\end{proof}

\begin{prop}\label{multiplicationf}
Suppose that $A$ is strongly $\eta$-convergent.
Then the $A$-module
$\operatorname{Hom}_{A\text{-cont}}(\aff, A)$
is an $R$-algebra for the multiplication defined by
\[
\psi \phi :
\aff \xrightarrow{\delta_\eta}
\aff(1) \xrightarrow{\phi(1)}
\aff \xrightarrow{\psi} A,
\]
where $\phi(1)$ is the map
\[
\sum_{\underline{i},\underline{j} \in \mathbb{N}^d}
a_{\underline{i},\underline{j}}
\bigl(
\underline{\xi}^{(\underline{i})_{\underline{\sigma}}}
\otimes'
\underline{\xi}^{(\underline{j})_{\underline{\sigma}}}
\bigr)
\longmapsto
\sum_{\underline{i},\underline{j} \in \mathbb{N}^d}
a_{\underline{i},\underline{j}}\,
\theta_\eta\!\bigl(
\phi(\underline{\xi}^{(\underline{j})_{\underline{\sigma}}})
\bigr)\,
\underline{\xi}^{(\underline{i})_{\underline{\sigma}}}.
\]
Moreover, if $A$ is very strongly $\eta$-convergent, it is a Banach $R$-algebra.
\end{prop}

\begin{proof}
Note that $\phi(1)$ is well-defined since
\[
\lVert a_{i,j}\,
\theta_\eta(\phi(\xi^{(j)_\sigma}))\,
\xi^{(i)_\sigma} \rVert
\le
\lVert \theta_\eta \rVert
\lVert \phi \rVert
\lVert a_{i,j} \rVert
\eta^{|i|+|j|}
\to 0
\quad \text{when } |i|+|j| \to \infty.
\]
Then it is straightforward to check that the multiplication in the
statement defines an $R$-algebra structure on
$\operatorname{Hom}_{A\text{-cont}}(\aff, A)$.
Also, if $A$ is very strongly $\eta$-convergent,
$\theta_\eta : A \to \aff$ is an isomorphic injection.
So $\lVert \theta_\eta \rVert = 1$ and the above inequality shows that
$\lVert \phi(1) \rVert \le \lVert \phi \rVert$.
On the other hand, if $A$ is very strongly $\eta$-convergent,
$p_2 : \aff \to \aff(1)$ is also an isometric injection. Also, a direct computation shows the equality $\phi(1)\circ p_2 = \theta_\eta \circ \phi.$ Thus we have for any $0 \neq s \in \aff$ the inequality
\[
\frac{\lVert \phi(s) \rVert}{\lVert s\rVert} = \frac{\lVert \theta_\eta\circ \phi(s) \rVert }{\lVert s \rVert} = \frac{\lVert \phi(1) \circ p_2 (s)\rVert }{\lVert s \rVert} \leq \lVert \phi(1)\rVert \frac{\lVert p_2(s) \rVert}{\lVert s \rVert}=\lVert \phi(1) \rVert.
\]
Thus we also have $\lVert \phi\rVert \leq \lVert \phi(1) \rVert$ and so $\lVert \phi\rVert =\lVert \phi(1) \rVert$. So
\[
\rVert \psi \phi \rVert \leq \lVert \psi\rVert \lVert \phi(1) \rVert \lVert \delta_\eta \rVert = \lVert \psi \rVert \lVert \phi \rVert.
\]
Hence $\text{Hom}_{A\text{-cont}}(\aff,A)$ is a Banach $R$-algebra.
\end{proof}

\begin{defn}
    For $A$-modules $M,N$ of finite type, we define the norm $\lVert \varphi \rVert_\eta$ of $\varphi \in \text{D}_{\underline{\sigma}}^{(\eta)}(M,N)$ by $\lVert \varphi \rVert_\eta :=\lVert \tilde{\varphi}_\eta\rVert.$
\end{defn}

\begin{prop}\label{9.25}
Suppose $A$ is very strongly $\eta$-convergent. The twisted differential operators of radius $\eta$ are stable under composition and $\lVert \varphi \circ \psi \rVert_\eta \leq \lVert \varphi \rVert_\eta \lVert \psi \rVert_\eta$. 
\end{prop}

\begin{proof}
Let $\varphi:M \rightarrow N$ and $\psi:L \rightarrow M$ be two twisted differential operators of radius $\eta$. The diagram below commutes
\[
\xymatrix{
P_{A} \otimes_A' L \ar[r]^-\delta \ar[d] & P_{A} \otimes_A'P_{A} \otimes_A'L \ar[d] \ar[r]^-{\text{Id} \otimes' \tilde{\psi}} & P_{A} \otimes_A' M \ar[r]^-{\tilde{\phi}} \ar[d] & N \ar[d] \\
\aff \otimes_A' L \ar[r]^-{\delta_\eta} & \aff(1) \otimes_A'L  \ar[r]^-{\tilde{\psi}(1)} & \aff \otimes_A' M \ar[r]^-{\tilde{\phi}_\eta}  & N,
}
\]
where $\tilde{\psi}(1)$ is defined by
\[
\left(\sum_{\underline{i}, \underline{j} \in \N^d} a_{\underline{i}, \underline{j}} (\underline{\xi}^{(\underline{i})_{\underline{\sigma}}} \otimes' \underline{\xi}^{(\underline{j})_{\underline{\sigma}}})\right) \otimes'l \rightarrow \sum_{\underline{i}, \underline{j} \in \N^d} a_{\underline{i}, \underline{j}} \underline{\xi}^{(\underline{i})_{\underline{\sigma}}} \otimes' \psi(\underline{\xi}^{(\underline{j})_{\underline{\sigma}}}\otimes'l)
\]
and the top row represents $\widetilde{\phi \circ \psi}$. Thus, twisted differential operators of radius $\eta$ are stable by composition.
By submultiplicativity of the norm, the following inequality is satisfied
\[
\lVert \phi \circ \psi \rVert_\eta \leq \lVert \widetilde{(\phi \circ \psi )}_\eta \rVert \leq \lVert \tilde{\phi}_\eta \rVert \lVert \tilde{\psi}(1) \rVert \lVert \delta_\eta \rVert = \lVert \tilde{\phi}_\eta \rVert \lVert  \tilde{\psi}_\eta \rVert=\lVert \phi \rVert_\eta \lVert \psi \rVert_\eta. 
\]
Here, the second last equality follows from the equality $\lVert \tilde{\psi}_\eta(1)\rVert =\lVert \tilde{\psi}_\eta \rVert$, which we can prove in the same way as the proof of Proposition \ref{multiplicationf}. So the claims hold.
\end{proof}

\begin{prop}\label{commu}
Let \( M \) be an \( A \)-module of finite type, and let \( \theta_\eta : M \rightarrow M \otimes_A \aff \) be an \( A \)-linear map with respect to the right structure of \( \aff \). Then \( \theta_\eta \) is a twisted Taylor map of radius \( \eta \) if and only if the following diagram commutes
\begin{equation}\label{9.14}
\xymatrix{
M \ar[r]^-{\theta_\eta} \ar[d]_-{\theta_\eta} & M \otimes_A \aff \ar[d]^-{\theta_\eta (1)} \\
M \otimes_A \aff \ar[r]_-{\emph{Id}\otimes \delta_\eta} & M \otimes_A \aff(1),
}
\end{equation}
where $\theta_\eta(1)$ is the map induced by the map
\[
M \times \aff \rightarrow M \otimes_A \aff(1), ~(m,f) \mapsto p_2(f)((\emph{Id} \otimes p_1) \circ \theta_\eta(m)).
\]
\end{prop}

\begin{proof}
Assume that the module \( M \) is endowed with a twisted Taylor map of radius \( \eta \). By Proposition \ref{3.3}, the twisted Taylor map $\widehat{\theta} = \varprojlim_n \theta_n$ factorizes through \( M \otimes_A \aff \) via the map \( \theta_\eta \). Then, the commutative diagram in Definition \ref{twistedtaylorstructure} and the commutative diagram in the proof of Proposition \ref{4.5} imply the commutativity of the diagram (\ref{9.14}), allowing us to conclude. Conversely, if the diagram (\ref{9.14}) commutes, we can follow the same reasoning to deduce that \( \theta_\eta \) is a twisted Taylor map of radius \( \eta \).
\end{proof}

\begin{defn}\label{ring of twisted differential operators of radius eta}
Assume $A$ is strongly $\eta$-convergent. Then the
\emph{ring of twisted differential operators of radius $\eta$} is
\[
\text{D}^{(\eta)}_{\underline{\sigma}}:=\text{Diff}^{(\eta)}_{A,\underline{\sigma}}(A,A).
\]
\end{defn}

\begin{cor}
There exists a canonical isomorphism of Banach $A$-modules
\[
\emph{D}^{(\eta)}_{\underline{\sigma}} \simeq \emph{Hom}_{A\emph{-cont}}(\aff,A).
\]
\end{cor}

\begin{rem}
For \( n \in \mathbb{N} \), we denote by \( K^{[n]} \) the $A$-submodule generated by \( \lbrace \partial_{\underline{\sigma}}^{[\underline{k}]} ,  |\underline{k}| \geq n \rbrace\), and define
\[
\widehat{\text{D}}_{\underline{\sigma}}^{(\infty)} = \varprojlim \text{D}^{(\infty)}_{A, \underline{\sigma}} / K^{[n]}.
\]
In general, \( \widehat{\text{D}}_{\underline{\sigma}}^{(\infty)} \) is not a ring. However, we have the following description
\[
\widehat{\text{D}}_{\underline{\sigma}}^{(\infty)} := \left\lbrace \sum_{\underline{k} \in \mathbb{N}^d} z_{\underline{k}} \partial_{\underline{\sigma}}^{[\underline{k}]}, ~ z_{\underline{k}} \in A \right\rbrace.
\]
By duality, the commutative diagram (\ref{diagram}) induces the following sequence of inclusions
\[
\text{D}^{(\infty)}_{\underline{\sigma}} \rightarrow \text{D}^{(\eta)}_{\underline{\sigma}} \rightarrow \widehat{\text{D}}_{\underline{\sigma}}^{(\infty)}.
\]
The right inclusion map is explicitly given by:
\[
\varphi \mapsto \sum_{\underline{k} \in \mathbb{N}^d} \tilde{\varphi}_\eta \left(\underline{\xi}^{(\underline{k})_{\underline{\sigma}}}\right) \partial_{\underline{\sigma}}^{[\underline{k}]}.
\]
\end{rem}

\begin{prop}\label{9.31}
Assume that $A$ is strongly $\eta$-convergent. The injective map $\emph{D}^{(\eta)}_{A,\underline{\sigma}} \rightarrow \widehat{\emph{D}}^{(\infty)}_{A,\underline{\sigma}}$ induces an isometric isomorphism of Banach $A$-modules
\[
\emph{D}^{(\eta)}_{\underline{\sigma}} \rightarrow \left\lbrace 
\sum_{\underline{k} \in \mathbb{N}^d} z_{\underline{k}} \der, ~ \exists C >0,  ~\forall\underline{k} \in \mathbb{N}^d, ~\lVert z_{\underline{k}} \rVert \leq C \eta^{|\underline{k}|}
\right\rbrace
\]
for the sup norm
\[
\left\lVert \sum_{\underline{k} \in \mathbb{N}^d} z_{\underline{k}} \der \right\rVert_\eta=\sup \left\lbrace \left\lVert z_{\underline{k}} \right\rVert / \eta^{|\underline{k}|} \right\rbrace.
\]
\end{prop}

\begin{proof}
Let \( \varphi \in \text{D}^{(\eta)}_{\underline{\sigma}} \). Since \( \tilde{\varphi}_\eta \) is continuous, for all \( \underline{k} \in \mathbb{N}^d \), we have:
\[
\left\lVert \tilde{\varphi}_\eta \left( \underline{\xi}^{(\underline{k})_{\underline{\sigma}}} \right) \right\rVert \leq \left\lVert \tilde{\varphi}_\eta \right\rVert \left\lVert \underline{\xi}^{(\underline{k})_{\underline{\sigma}}} \right\rVert = \left\lVert \tilde{\varphi}_\eta \right\rVert \eta^{|\underline{k}|}.
\]
Then the injection map \( \text{D}^{(\eta)}_{\underline{\sigma}} \rightarrow \widehat{\text{D}}^{(\infty)}_{\underline{\sigma}} \) has its image contained in the considered set, and its norm is at most 1.

Now, consider an element \( \sum_{\underline{k} \in \mathbb{N}^d} w_{\underline{k}} \partial_{\underline{\sigma}}^{[\underline{k}]} \) with norm at most \( C \) on the right-hand side. By Proposition \ref{dualité}, there exists a unique \( \varphi \in \text{D}^{(\eta)}_{\underline{\sigma}} \) such that, for every sequence  \( (z_{\underline{k}})_{\underline{k} \in \mathbb{N}^d} \) of elements of \( A \) satisfying \( \lVert z_{\underline{k}} \rVert \eta^{|\underline{k}|} \rightarrow 0 \), 
\[
\tilde{\varphi}_\eta\left( \sum_{\underline{k} \in \mathbb{N}^d} z_{\underline{k}} \underline{\xi}^{(\underline{k})_{\underline{\sigma}}} \right) = \sum_{\underline{k} \in \mathbb{N}^d} z_{\underline{k}} w_{\underline{k}} \in A,
\]
which follows from the fact that \( \lVert z_{\underline{k}} w_{\underline{k}} \rVert \leq C \lVert z_{\underline{k}} \rVert \eta^{|\underline{k}|} \rightarrow 0 \). This also implies that \( \lVert \varphi \rVert \leq C \). Therefore, this defines an inverse map, which is indeed an isometry.
\end{proof}

\begin{prop}\label{5.5}
Assume that $A$ is strongly $\eta$-convergent. The structure of $\emph{D}_{\underline{\sigma}}^{(\infty)}$-module over a $\eta$-convergent module of finite type over $A$ extends canonically to a structure of topological $\emph{D}^{(\eta)}_{\underline{\sigma}}$-module.
\end{prop}

\begin{proof}
Consider the \( \aff \)-linear continuous map
\[
M \otimes_A \aff \rightarrow \text{Hom}_{A\text{-cont}}(\text{Hom}_{A\text{-cont}}(\aff, A), M), ~s \otimes f \mapsto (\psi \mapsto \psi(f) s).
\]
The twisted Taylor map of radius \( \eta \) can be linearized to obtain a continuous map
\[
\aff \otimes'_A M \rightarrow M \otimes_A \aff.
\]
By composition, we obtain an \( \aff \)-linear continuous map:
\[
\aff \otimes'_A M \rightarrow \text{Hom}_{A\text{-cont}}(\text{D}^{(\eta)}_{\underline{\sigma}}, M),
\]
or equivalently, an \( \aff \)-linear continuous map
\[
\text{D}^{(\eta)}_{\underline{\sigma}} \rightarrow \text{Hom}_{A\text{-cont}}(\aff \otimes'_A M, M) = \text{D}^{(\eta)}_{\underline{\sigma}}(M, M).
\]
One can check that this is a morphism of rings by Proposition \ref{commu}. By construction, this map is compatible with the map
\[
\text{D}^{(\infty)}_{\underline{\sigma}} \rightarrow \text{D}^{(\infty)}_{\underline{\sigma}}(M, M)
\]
which defines the action of \( \text{D}^{(\infty)}_{\underline{\sigma}} \) on \( M \).
\end{proof}

\begin{prop}\label{5.6}
Suppose $A$ is strongly $\eta$-convergent.
If $M$ is a $\emph{D}^{(\eta)}_{\underline{\sigma}}$-module of finite type over $A$ which is $\eta$-convergent, then $M$ is $\eta^\dagger$-convergent.
\end{prop}

\begin{proof}
We need to show that \( M \) is \( \eta' \)-convergent for every \( \eta' < \eta \). Since the map
\[
\text{D}^{(\eta)}_{\underline{\sigma}} \rightarrow \text{Hom}_{A\text{-cont}}(\aff \otimes_A' M, M)
\]
is continuous, there exists a constant \( C \) such that for any $s \in M$ and \( \varphi \in \text{D}^{(\eta)}_{\underline{\sigma}} \),
\[
\lVert \tilde{\varphi}(1\otimes s) \rVert \leq C \lVert \varphi \rVert_\eta \lVert s \rVert.
\]
Considering the case $\varphi = \partial_{\underline{\sigma}}^{[\underline{k}]}$, we obtain the inequality 
\[
\lVert \partial_{\underline{\sigma}}^{[\underline{k}]} (s) \rVert \leq \frac{C \lVert s \rVert}{\eta^{|\underline{k}|}}.
\]
This implies that
\[
\lVert \partial_{\underline{\sigma}}^{[\underline{k}]} (s) \rVert \eta'^{|\underline{k}|} \leq C \lVert s \rVert \left(\frac{\eta'^{|\underline{k}|}}{\eta^{|\underline{k}|}}\right) \rightarrow 0.
\]
Thus, we have shown that the module \( M \) is \( \eta' \)-convergent for every \( \eta' < \eta \).
\end{proof}

\begin{prop}\label{6.1}
Let \( \underline{\tau}=(\tau_1, \ldots, \tau_d) \) be another family of \( R \)-linear continuous endomorphisms of \( A \) that commute. Assume that \( \underline{x} \) are classical and symmetric (good) \( \underline{\tau} \)-coordinates, and that \( \eta \geq \rho(\underline{\tau}) \). Then, if \( A \) is \( \eta \)-convergent with respect to \( \underline{\sigma} \), it is also \( \eta \)-convergent with respect to \( \underline{\tau} \), with the same twisted Taylor morphism as for \( \underline{\sigma} \). In particular, if $A$ is strongly $\eta$-convergent with respect to $\underline{\sigma}$, it is also strongly $\eta$-convergent with respect to $\underline{\tau}$.
\end{prop}

\begin{proof}
We can consider the commutative diagram below
\[
\xymatrix{
& & A \ar[ld]_-\theta \ar[d]^-{\theta_\eta} \\
A[\underline{\xi}] \ar@{^{(}->}[r] \ar@{->>}[d]_-{\Theta_{A, (n)_{\underline{\tau}}}} & P_{A} \ar@{->>}[d] \ar[r] & \aff \ar@{->>}[d] \\
A[\underline{\xi}]/(\underline{\xi}^{(\underline{k})_{\underline{\tau}}}; ~|\underline{k}|= n+1)\ar[r]^{~~~~~~~~~~~\simeq} \ar@/_2pc/[rr] & P_{A,(n)_{\underline{\tau}}} &    \aff / \left(\underline{\xi}^{(\underline{k})_{\underline{\tau}}}; ~|\underline{k}|= n+1\right).
}
\]
~\\
~\\
Here, $\theta_\eta$ is the twisted Taylor map of radius $\eta$ with respect to $\underline{\sigma}$. Since the diagram is commutative, the composition of the vertical maps on the right does not depend on $\underline{\sigma}$. Moreover, by hypothesis $ \eta \geq \rho (\underline{\tau}) $, the map
\[
A[\underline{\xi}]/\left(\underline{\xi}^{(\underline{k})_{\underline{\tau}}}; ~|\underline{k}|= n+1\right) \rightarrow \aff / \left(\underline{\xi}^{(k)_{\underline{\tau}}}; ~|\underline{k}|= n+1 \right)
\]
is an isomorphism by Proposition \ref{pr}. Taking the limit, we obtain the following diagram
\[
\xymatrix{
& A \ar[d]^-{\theta_\eta} \ar[ld]_-\theta \\
P_{A} \ar[d] & \aff \ar[ld] \\
\varprojlim_n  P_{A,(n)_{\underline{\tau}}}.
}
\]
By Proposition \ref{3.3}, it follows that $A$ is $\eta$-convergent with respect to $\underline{\tau}$, and  $\theta_\eta$ is the twisted Taylor map of radius $\eta$ with respect to $\underline{\tau}$.
\end{proof}

\begin{prop}
The ring structure of $\emph{Hom}_{A\emph{-cont}}(\aff,A)$ does not depend on $\underline{\sigma}$.
\end{prop}

\begin{proof}
This result follows from the multiplication defined on \( \text{Hom}_{A\text{-cont}}(A\left\lbrace \underline{\xi}/\eta \right\rbrace, A) \), from the diagram in the proof of Proposition \ref{9.25}, and from Proposition \ref{6.1}, which ensures that the ring structure is independent of \( \underline{\sigma} \).
\end{proof}

\begin{thm}\label{6.3}
Let \( \underline{\tau} =(\tau_1, \ldots, \tau_d) \) be another family of \( R \)-linear continuous endomorphisms of \( A \) that commute. Assume that \( \underline{x} \) are classical and symmetric (good) \( \underline{\tau} \)-coordinates, and that \( \eta \geq \rho(\underline{\tau}) \). Then, there exists an isometric \( A \)-linear isomorphism of \( R \)-algebras
\[
\emph{D}^{(\eta)}_{\underline{\sigma}} \simeq \emph{D}^{(\eta)}_{\underline{\tau}},
\]
which only depends on \( \underline{x} \).
\end{thm}

\begin{proof}
This result follows directly from the previous proposition and the fact that the isomorphism $\text{D}^{(\eta)}_{\underline{\sigma}} \simeq \text{Hom}_{A\text{-cont}}(A\left\lbrace \underline{\xi}/\eta \right\rbrace, A)$ is \( A \)-linear.
\end{proof}

\begin{rem}
In particular, it is possible to apply this theorem in the case where $\forall i=1, \ldots, d,~\tau_i=\text{Id}_A$ and $\underline{x}$ are good étale coordinates, thanks to Proposition \ref{bonsigmacord} and usual differential calculus.
\end{rem}

\begin{cor}\label{6.4}
If \( \underline{x} \) are symmetric (good) $\underline{\emph{Id}}$-coordinates as well as classical symmetric (good) $\underline{\sigma}$-coordinates, then there exists an isometric \( A \)-linear isomorphism of \( R \)-algebras:
\[
\emph{D}^{(\eta)}_{\underline{\sigma}} \simeq \emph{D}^{(\eta)}_{\underline{\emph{Id}}},
\]
where \( \emph{D}^{(\eta)}_{\underline{\emph{Id}}} \) denotes the ring of twisted differential operators of radius \( \eta \) with respect to \( \underline{\emph{Id}} = \left( \emph{Id}_A, \ldots, \emph{Id}_A \right) \).
\end{cor}

\subsection{Confluence}
From now on, we fix a $\eta > \rho(\underline{\sigma})$ such that $A$ is strongly $\eta^\dagger$-convergent with respect to $\underline{\sigma}$.

\begin{defn}\label{operators of radius eta dag}
The ring of \emph{twisted differential operators of radius $\eta^\dagger$} is defined as
\[
\text{D}_{\underline{\sigma}}^{(\eta^\dagger)}=\varinjlim_{\eta'<\eta} \text{D}_{\underline{\sigma}}^{(\eta')}.
\]
\end{defn}

\begin{prop}\label{7.2}
The category of $\emph{D}_{\underline{\sigma}}^{(\eta^\dagger)}$-modules of finite type over $A$ is equivalent to the subcategory of $\emph{D}^{(\infty)}_{\underline{\sigma}}$-modules $\eta^\dagger$-convergent of finite type over $A$.
\end{prop}

\begin{proof}
It is enough to apply Propositions \ref{5.5} and \ref{5.6}.
\end{proof}

\begin{thm}\label{7.3}
Let \( \underline{\tau}= (\tau_1, \ldots, \tau_d) \) be another family of \( R \)-linear continuous endomorphisms of \( A \) such that \( \underline{x} \) are also classical symmetric (good) \( \underline{\tau} \)-coordinates. If \( \eta > \rho(\underline{\tau}) \) and $A$ is strongly $\eta^\dagger$-convergent with respect to $\underline{\tau}$, then the categories of \( \emph{D}^{(\infty)}_{\underline{\sigma}} \)-modules of finite type over \( A \) which are also strongly \( \eta^\dagger \)-convergent with respect to \( \underline{\sigma} \) and \( \underline{\tau} \) are equivalent.
\end{thm}

\begin{proof}
This result follows directly from applying Theorem \ref{6.3} and Proposition \ref{7.2}.
\end{proof}

As in \cite{ADV04}, where the \( q \)-analogue of the exponential function is defined by \( \sum\limits_{n \geq 0} \frac{x^n}{(n)_q!} \), we are particularly interested in the case where the \( q \)-analogues of non-zero integers are invertible.

\begin{defn}
Let \( \underline{q} = ( q_1, \ldots, q_d ) \) be elements of \( R \) such that \( \underline{x} \) are classical (good) \( \underline{q} \)-coordinates and
\[
\forall i = 1, \ldots, d, \, \forall n \geq 1, \, (n)_{q_i} \in R^\times.
\]
An \( A \)-module \( M \) of finite type endowed with an integrable twisted connection \( \nabla_{\underline{q}} \) is said to be \emph{\( \eta^\dagger \)-convergent} if its image in the category of \( \text{D}_{\underline{q}}^{(\infty)} \)-modules (see Proposition \ref{equivate 3}) is \( \eta^\dagger \)-convergent. 
\end{defn}

Under the hypothesis of the previous definition, we denote by \( \nabla_{\underline{q}}^{\text{Int}}\text{-}\text{Mod}_{\text{tf}}^{(\eta^\dagger)}(A) \) the category of \( A \)-modules of finite type endowed with an integrable twisted connection that are \( \eta^\dagger \)-convergent. Additionally, we define \( \text{D}_{\underline{q}}^{(\eta^\dagger)}\text{-}\text{Mod}_{\text{tf}}(A) \) as the category of \( \text{D}_{\underline{q}}^{(\eta^\dagger)} \)-modules of finite type over \( A \). The cohomology of this category is given by \( \text{Ext}_{\text{D}_{\underline{q}}^{(\eta^\dagger)}}(A, \cdot) \). We also use the notation
\[
\text{SP}\left(\text{D}_{\underline{q}}^{(\eta^\dagger)}\right) = \text{Hom}_{\text{D}_{\underline{q}}^{(\eta^\dagger)\text{op}}}\left(\text{DR}\left(\text{D}_{\underline{q}}^{(\eta^\dagger)}\right), \text{D}_{\underline{q}}^{(\eta^\dagger)}\right).
\]
Here $\text{DR}(-)$ is as in Remark \ref{7.23}.

\begin{thm}\label{confluence}
Assume there exist elements \( \underline{q} = ( q_1, \ldots, q_d) \) in \( R \) such that \( \underline{x} \) are classical and symmetric (good) \( \underline{q} \)-coordinates, and that every nonzero \( q_i \)-adic integer is invertible in \( R \). Additionally, if \( A \) is strongly \( \eta^\dagger \)-convergent, then we have the following equivalence of categories
\[
\nabla_{\underline{q}}^{\emph{Int}}\emph{-}\emph{Mod}_{\emph{tf}}^{(\eta^\dagger)}(A) \simeq \emph{D}_{\underline{q}}^{(\eta^\dagger)}\emph{-}\emph{Mod}_{\emph{tf}}(A).
\]
This equivalence is compatible with de Rham cohomology on left hand side and $\emph{Ext}$ groups on the right hand side.
\end{thm}

\begin{proof}
By Proposition \ref{equivate 3}, the category of finite \( A \)-modules endowed with an integrable connection \( \nabla_{\underline{q}} \) is equivalent to the category of \( \text{D}_{A,\underline{q}}^{(\infty)} \)-modules of finite type over \( A \). This equivalence relies on the hypothesis regarding the elements \( q_1, \ldots, q_d \). The equivalence between the \( \eta^\dagger \)-convergent categories follows directly from Proposition \ref{7.2}.

It remains to show that this equivalence is compatible with the cohomologies on both sides.  Recall that, by Proposition \ref{9.31},
\[
\text{D}^{(\eta^\dagger)}_{\underline{q}} \simeq \left\lbrace 
\sum_{\underline{k} \in \mathbb{N}^d} z_{\underline{k}} \partial^{[\underline{k}]}_{\underline{q}}, ~ \exists C >0,  ~\forall\underline{k} \in \mathbb{N}^d, ~\lVert z_{\underline{k}} \rVert \leq C \eta^{|\underline{k}|}
\right\rbrace.
\]
For $0 \leq e \leq d,$ we define
\[
\text{D}^{(\eta^\dagger)}_{\underline{q},e} := \left\lbrace 
\sum_{\underline{k} \in \mathbb{N}^e} z_{\underline{k}} \partial^{[\underline{k}]}_{\underline{q}}, ~ \exists C >0,  ~\forall\underline{k} \in \mathbb{N}^e, ~\lVert z_{\underline{k}} \rVert \leq C \eta^{|\underline{k}|}
\right\rbrace,
\]
so that $\text{D}^{(\eta^\dagger)}_{\underline{q},d} = \text{D}^{(\eta^\dagger)}_{\underline{q}}$ and $\text{D}^{(\eta^\dagger)}_{\underline{q},0} = A$.

We prove that the multiplication by $\partial_{\underline{q},e}$ from the right gives an injective map $\cdot \partial_{\underline{q},e}: \text{D}^{(\eta^\dagger)}_{\underline{q},e} \rightarrow\text{D}^{(\eta^\dagger)}_{\underline{q},e}$ and its cokernel is isomorphic to $\text{D}^{(\eta^\dagger)}_{\underline{q},e-1}$. The injectivity is trivial from the above expression. To show the claim on the cokernel, it suffices to prove that any element in $\text{D}^{(\eta^\dagger)}_{\underline{q},e}$ of the form $\sum_{l=1}^\infty y_l \partial^{[l]}_{\underline{q},e}$ with $y_l \in \text{D}^{(\eta^\dagger)}_{\underline{q},e-1}$ is in the image of $\cdot \partial_{\underline{q},e}$. By definition of $\text{D}^{(\eta^\dagger)}_{\underline{q},e}$, there exists some $\eta'<\eta$ and $C>0$ such that $\lVert y_l \rVert_{\eta'}\eta'^{-l} \leq C$ for all $l$, where $\lVert - \rVert_{\eta'}$ is defined as in Proposition \ref{9.31}. Then, formally, it is the image of the series $\sum_{l=1}^\infty y_l(l)_q^{-1} \partial^{[l-1]}_{\underline{q},e}$ (where $q:=q_e$) by $\cdot \partial_{\underline{q},e}$, and it suffices to prove that this series belong to $\text{D}^{(\eta^\dagger)}_{\underline{q},e}$. We suppose $q\neq 1$ because the case $q=1$ is standard. For $\eta'<\eta''<\eta$, we have
\[
\frac{\bigl\|y_l (l)_q^{-1}\bigr\|_{\eta''}}{(\eta'')^{\,l-1}}
\;\le\;
\eta'' \,\frac{\|y_l\|_{\eta'}}{\eta'^l}\,
\left(\frac{\eta'}{\eta''}\right)^{\!l}
\left\lVert\frac{q-1}{q^{\,l}-1}\right\rVert
\;\le\;
\eta''\, C\,
\left(\frac{\eta'}{\eta''}\right)^{\!l}
\left\lVert\frac{q-1}{q^{\,l}-1}\right\rVert.
\]
So, to prove the claim, it suffices to prove that, for any $0<\delta <1, ~\lVert q^l-1\rVert \gg \delta^l$ when $l \rightarrow \infty$. If $\lVert q \rVert \neq 1$, this holds because $\lVert q^l-1\rVert \geq 1$. If $\lVert q \rVert=1$, take the minimal $a \geq 1$ such that $\lVert q^a-1 \rVert <1$. For $l=ap^bc+r \geq a \text{ with } b,c,r \in \N, (c,p)=1, 0\leq r \leq a-1$, put
\[
Q_b=q^{ap^b}-1 = (1+(q^a-1))^{p^b}-1=\sum_{i=1}^{p^b} \binom{p^b}{i}(q^a-1)^i.
\]
Note that $0<\lVert Q_b\rVert <1.$ Then, if $r\neq 0$,
\[
\lVert q^l-1\rVert=\lVert(q^{ap^bc+r}-q^{ap^bc})+(q^{ap^bc}-1) \rVert = 1,
\]
because
\[
\lVert q^{ap^bc+r}-q^{ap^bc} \rVert = \rVert q^{ap^bc}\rVert \lVert q^r-1\rVert=1
\]
by minimality assumption of $a$ and
\[
\rVert q^{ap^bc}-1\rVert = \lVert (1+Q_ b)^c-1 \rVert = \lVert Q_b\rVert<1.
\]
If $r=0, \lVert q^l-1 \rVert =\lVert (1+Q_b)^c-1 \rVert = \lVert Q_b \rVert$ and if $b$ is sufficiently large,
\[
\lVert Q_b \rVert = p^{-b} \lVert q^a-1 \rVert \geq a\lVert q^a-1\rVert/l.
\]
By these equations, we conclude that
\[
\lVert q^l-1\rVert \gg 1/l \gg \delta^l \text{ when } l \rightarrow \infty,
\]
which proves the claim.

By the claim in the previous paragraph, we see that the sequence \(\partial_{\underline{q},1}, \ldots, \partial_{\underline{q},d} \) is a regular sequence in \( \text{D}^{(\eta^\dagger)}_{\underline{q}} \), and by the same argument as Proposition 1.4.3 of \cite{Schap} we can check that \(\text{SP}\left(\text{D}_{\underline{q}}^{(\eta^\dagger)}\right) \) provides a free resolution of \( A \). Therefore, we have
\[
\begin{aligned}
    \text{RHom}_{\text{D}_{\underline{q}}^{(\eta^\dagger)}}(A, M) &= \text{Hom}_{\text{D}_{\underline{q}}^{(\eta^\dagger)}}\left( \text{SP}(\text{D}_{\underline{q}}^{(\eta^\dagger)}), M \right) \\
    &= \text{Hom}_{\text{D}_{\underline{q}}^{(\eta^\dagger)}}\left( \text{SP}(\text{D}_{\underline{q}}^{(\eta^\dagger)}), \text{D}_{\underline{q}}^{(\eta^\dagger)} \right) \otimes_{\text{D}_{\underline{q}}^{(\eta^\dagger)}} M \\
    &= \text{DR}(\text{D}_{\underline{q}}^{(\eta^\dagger)}) \otimes_{\text{D}_{\underline{q}}^{(\eta^\dagger)}} M \\
    &= \text{DR}(M).
\end{aligned} 
\]
This concludes the proof.
\end{proof}

The next and last result provides a generalization in several variables of the theorem in \cite{GLQ4} and its spirit aligns with the work presented in \cite{ADV04} and \cite{pulita}.

\begin{thm}\label{thm:confluence-bis}
Let \( K \) be a non-archimedean field of characteristic 0, complete with respect to a non-archimedean norm. Let \( A \) be a Tate \( K \)-algebra, and let \( \underline{q} = (q_1, \ldots, q_d) \) be elements of \( R \) such that, for every \( i = 1, \ldots, d \) and $n \geq 1$, the \( q_i \)-integers $(n)_{q_i}$ are invertible in $K$. Assume that $A$ admits elements $\underline{x}=(x_1,\ldots,x_d)$ which are classical and symmetric (good) $\underline{q}$-coordinates as well as symmetric (good) $\underline{\emph{Id}}$-coordinates, and assume moreover that $A$ is strongly $\eta^\dagger$-convergent with respect to $\underline{q}$ and with respect to $\underline{\emph{Id}}$, and that \( \eta >\rho(\underline{\sigma}) \). Then we have the following equivalence of categories
\[
\nabla_{\underline{q}}^{\emph{Int}}\text{-}\emph{Mod}_{\emph{tf}}^{(\eta^\dagger)}(A) \simeq \nabla_{\underline{\emph{Id}}}^{\emph{Int}}\emph{-}\emph{Mod}_{\emph{tf}}^{(\eta^\dagger)}(A),
\]
which is compatible with de Rham cohomologies on both sides.
\end{thm}

\begin{proof}
We are in a situation where Theorem \ref{confluence} applies. It suffices to show that there exists an isomorphism
\[
\text{D}^{(\eta^\dagger)}_{\underline{q}} \simeq \text{D}^{(\eta^\dagger)}_{\underline{\text{Id}}}.
\]
This follows directly from Corollary \ref{6.4}.
\end{proof}

\end{document}